\documentclass[a4paper]{amsart}
\usepackage[utf8]{inputenc}
\usepackage[T1]{fontenc}
\usepackage{lmodern}
\usepackage{amssymb}
\usepackage[all]{xy}
\usepackage{nicefrac,mathtools,enumitem}
\usepackage{microtype}
\usepackage{enumerate}
\usepackage[pdftitle={Twisted actions and regular Fell bundles over inverse semigroups},
  pdfauthor={Alcides Buss and Ruy Exel},
  pdfsubject={Mathematics; ...}]{hyperref}
\usepackage[lite]{amsrefs}
\newcommand*{\MRref}[2]{ \href{http://www.ams.org/mathscinet-getitem?mr=#1}{MR #1}}
\newcommand*{\arxiv}[1]{ \href{http://www.arxiv.org/abs/#1}{arXiv:#1}}

\numberwithin{equation}{section}
\theoremstyle{plain}
\newtheorem{theorem}[equation]{Theorem}
\newtheorem{lemma}[equation]{Lemma}
\newtheorem{proposition}[equation]{Proposition}
\newtheorem{corollary}[equation]{Corollary}
\theoremstyle{definition}
\newtheorem{definition}[equation]{Definition}

\theoremstyle{remark}
\newtheorem{remark}[equation]{Remark}
\newtheorem{example}[equation]{Example}

\DeclareMathOperator*{\dom}{dom}
\DeclareMathOperator*{\ran}{ran}
\DeclareMathOperator*{\cspn}{\overline{span}}
\DeclareMathOperator*{\Rep}{Rep}
\DeclareMathOperator*{\supp}{supp}

\newcommand*{\nb}{\nobreakdash}
\newcommand*{\Star}{\texorpdfstring{$^*$\nobreakdash-\hspace{0pt}}{*-}}
\newcommand*{\Ad}{\mathrm{Ad}}

\newcommand*{\N}{\mathbb N} 
\newcommand*{\C}{\mathbb C} 
\newcommand*{\Torus}{\mathbb T} 

\newcommand*{\sbe}{\subseteq}
\newcommand*{\bound}{\mathcal{L}} 
\newcommand*{\ssim}{{\sim^{\!\!\!^{s}}}}

\newcommand*{\red}{\textup{r}} 
\newcommand*{\uni}{\mathcal{U}} 

\newcommand*{\cstar}{\texorpdfstring{$C^*$\nobreakdash-\hspace{0pt}}{C*-}}
\newcommand*{\cont}{\mathcal C} 
\newcommand*{\contz}{\cont_0} 
\newcommand*{\contb}{\cont_\textup{b}} 
\newcommand*{\id}{\textup{id}} 

\newcommand*{\hils}{\mathcal H} 
\newcommand*{\A}{\mathcal{A}} 
\newcommand*{\B}{\mathcal{B}} 
\newcommand*{\D}{\mathcal{D}} 
\newcommand*{\F}{\mathcal{F}} 
\newcommand*{\troa}{M} 
\newcommand*{\trob}{N} 
\newcommand*{\mult}{\mathcal{M}} 
\newcommand*{\E}{\mathcal{E}} 
\newcommand*{\U}{\mathcal{U}} 

\newcommand*{\defeq}{\mathrel{\vcentcolon=}}

\newcommand*{\into}{\hookrightarrow}
\newcommand*{\onto}{\twoheadrightarrow}
\newcommand*{\congto}{\xrightarrow\sim}
\newcommand*{\open}{\mathcal{O}}
\newcommand*{\K}{\mathcal{K}} 
\newcommand*{\Ls}{\mathcal{L}} 
\newcommand*{\SSS}{\mathcal I} 
\newcommand*{\tro}{ternary ring of operators} 
\newcommand*{\tros}{ternary rings of operators} 
\newcommand*{\inv}{^{-1}} 
\newcommand*{\rest}[1]{|_{#1}} 
\newcommand*{\Res}{\mathcal{R}} 

\newcommand*{\braket}[2]{\langle#1\!\mid\!#2\rangle} 
\newcommand*{\gen}[1]{\langle #1 \rangle} 

\newcommand*{\G}{\mathcal{G}} 
\newcommand*{\s}{\mathrm{d}} 
\renewcommand*{\r}{\mathrm{r}} 
\newcommand*{\domain}{\mathrm{d}} 
\newcommand*{\range}{\mathrm{r}} 
\newcommand*{\g}{\gamma}
\newcommand*{\Gz}{\G^{(0)}}
\newcommand*{\Gt}{\G^{(2)}}
\newcommand*{\qtrip}[3]{[#1,#2,#3]}
\newcommand*{\germ}[2]{[#1, #2]}

\begin{document}
\title[Twisted actions and regular Fell bundles over inverse semigroups]{Twisted actions and regular Fell bundles \\ over inverse semigroups}

\author{Alcides Buss}
\email{alcides@mtm.ufsc.br}
\author{Ruy Exel}
\email{exel@mtm.ufsc.br}

\address{Departamento de Matemática\\
  Universidade Federal de Santa Catarina\\
  88.040-900 Florianópolis-SC\\
  Brasil}

\begin{abstract}
  We introduce a new notion of twisted actions of inverse semigroups and show
that they correspond bijectively to certain \emph{regular} Fell bundles over
inverse semigroups, yielding in this way a structure classification of such
bundles.  These include as special cases all the \emph{stable} Fell bundles.

Our definition of twisted actions properly generalizes a previous one introduced
by Sieben and corresponds to Busby-Smith twisted actions in the group case. As
an application we describe twisted \'etale groupoid \cstar{}algebras in terms of
crossed products by twisted actions of inverse semigroups and show that Sieben's
twisted actions essentially correspond to twisted \'etale groupoids with
topologically trivial twists.
\end{abstract}



\thanks{This research was partially supported by CNPq.}

\maketitle

\tableofcontents

\section{Introduction}
\label{sec:introduction}

In \cite{Exel:twisted.partial.actions} the second named author introduced a method for constructing a Fell bundle
(also called a \cstar{}algebraic bundle \cite{fell_doran}) over a group $G$, starting
from a \emph{twisted partial action} of $G$ on a \cstar{}algebra. The relevance of
this construction is due to the fact that a very large number of Fell bundles, including all stable,
second countable ones, arise from a twisted partial action of the base group on the unit fiber algebra \cite[Theorem~7.3]{Exel:twisted.partial.actions}.
  It is the purpose of the present article to extend the above ideas in order to
embrace Fell bundles over \emph{inverse semigroups}.

The notion of Fell bundles over inverse semigroups was introduced by Sieben \cite{SiebenFellBundles}
and further developed in \cite{Exel:noncomm.cartan} and \cite{BussExel:Fell.Bundle.and.Twisted.Groupoids}.
Among its important occurrences one has that every twisted étale groupoid or, more generally, every Fell bundle over an étale groupoid
(in the sense of Kumjian \cite{Kumjian:fell.bundles.over.groupoids}) gives rise to a Fell bundle
over the inverse semigroup of its open bisections (see \cite[Example~2.11]{BussExel:Fell.Bundle.and.Twisted.Groupoids}).

Given a Fell bundle $\A =\{\A_s\}_{s\in S}$ over an inverse
semigroup $S$, one says that $\A$ is \emph{saturated} if $\A_{st}$
coincides with the closed linear span of $\A_s\A_t$ for all $s$
and $t$ in $S$. The requirement that a Fell bundle over an inverse
semigroup be saturated is not as severe as in the case of groups because in
the former situation one may apply the process of \emph{refinement} (see
Section~\ref{sec:Refinements}), transforming any Fell bundle into a saturated one (albeit
over a different inverse semigroup). We will therefore mostly restrict our
attention to saturated Fell bundles.

Both in the case of groups and of inverse semigroups each fiber of a Fell
bundle possesses the structure of a \emph{\tro} (see \cite{Zettl:TROs}, \cite[Section~4]{Exel:twisted.partial.actions}),
namely a mathematical structure isomorphic to a closed linear space of
operators on a Hilbert space which is invariant under the ternary operation
\begin{equation*}
 (x,y,z) \mapsto xy^*z.
\end{equation*}
Under special circumstances (see Definition~\ref{def:RegularTRO} below or \cite[Section~5]{Exel:twisted.partial.actions}) a {\tro}
$M$ admits a partial isometry $u$ (acting on the Hilbert space where $M$ is
represented) such that $M^*u=M^*M$, and $uM^*=MM^*$, in which case we say that
$u$ is \emph{associated to} $M$ and that $M$ is regular.

Likewise, given a Fell bundle $\A =\{\A_s\}_{s\in S}$ over an
inverse semigroup $S$, we say that $\A$ is \emph{regular} if every
$\A_s$ is regular as a {\tro}. Assuming this is the case, one may
consequently choose a family of partial isometries $\{u_s\}_{s\in S}$,
where each $u_s$ is associated to $\A_s$. These give rise to two
important families of objects, namely the automorphisms
\begin{equation}\label{FormulaForBeta}
\beta_s\colon a\in \A_{s^*s} \mapsto u_s a u_s^*\in \A_{ss^*},
\end{equation}
and the \emph{cocycle} $\omega=\{\omega(s,t)\}_{s, t\in S}$, given by
\begin{equation}\label{FormulaForOmega}
\omega(s,t) = u_su_tu_{st}^*.
\end{equation}

Although the partial isometries $u_s$ live outside our Fell bundle, the
presence of the $\beta_s$ and of the $\omega(s,t)$ is felt at the level of the
fibers over idempotent elements of $S$: this is obviously so with respect to
$\beta_s$, as its domain and range are fibers over idempotents, and one may show
that $\omega(s,t)$ is nothing but a unitary multiplier of the fiber over the
idempotent element $st(st)^*$.

The main point of departure for our research was the challenge of identifying a
suitable set of properties satisfied by the $\beta_s$ and the $\omega(s,t)$
which, when taken as axioms referring to abstractly given collections
$\{\beta_s\}_{s\in S}$ and $\{\omega(s,t)\}_{s,t\in S}$, could be used to
construct a Fell bundle over $S$.  The properties we decided to pick are to be
found in Definition~\ref{def:twisted action} below, describing our notion of a \emph{twisted action of an inverse semigroup}.

Returning to the $\beta_s$ and the $\omega(s,t)$ of \eqref{FormulaForBeta} and
\eqref{FormulaForOmega}, there is in fact a myriad of algebraic relations which
can be proven for these, a sample of which is listed under Proposition~\ref{prop:properties of twisted action coming from partial isometries}.
Having as our main goal to generalize the group case of \cite{Exel:twisted.partial.actions},
we were happy to welcome into our definition two of the main axioms adopted in the group case \cite[Definition 2.1]{Exel:twisted.partial.actions},
namely properties~\eqref{prop:properties of twisted action coming from partial isometries:item:beta_r beta_s=Ad_omega(r,s)beta_rs} and
\eqref{prop:properties of twisted action coming from partial isometries:item:CocycleCondition}
of Proposition~\ref{prop:properties of twisted action coming from partial isometries}.

However, choosing the appropriate replacement for \cite[Definition 2.1.d]{Exel:twisted.partial.actions}, namely
\begin{equation}\label{SiebenCondition}
  \omega(t,e)=\omega(e,t)=1,
\end{equation}
where $e$ denotes the unit of the group, turned out to be a major difficulty.
Nonetheless this dilemma proved to be the gateway to interesting and profound
phenomena relating, among other things, to topological aspects of twisted
groupoids, as we hope to be able to convey below.

Sieben \cite{SiebenTwistedActions} has considered a similar notion of twisted action, where he postulates
\eqref{SiebenCondition} for every $t$ and $e$ in $S$, such that $e$ is
idempotent. Although this may be justified by the fact that the idempotents of
$S$ play a role similar to the unit in a group, it cannot be proved in the model
situation where the $\omega(s,t)$ are given by \eqref{FormulaForOmega},  so we
decided not to adopt it.

After groping our way in the dark for quite some time (at one point we had
an untold number of strange looking axioms) we settled for~\eqref{def:twisted action:item:omega(r,r*r)=omega(e,f)=1_ef_and_omega(rr*,r)=1} and~\eqref{def:twisted action:item:omega(s*,e)omega(s*e,s)x=omega(s*,s)x} in Definition~\ref{def:twisted action}
which the reader may still find a bit awkward but which precisely fulfills our expectations.

Needless to say, regular Fell bundles over an inverse semigroup $S$ are then
shown to be in a one-to-one correspondence with twisted actions of $S$, as
proven in Corollary~\ref{cor:CorrespondenceRegularFellBundlesAndTwistedActions},
hence providing the desired generalization of the main result of \cite{Exel:twisted.partial.actions}.

Sieben's condition~\eqref{SiebenCondition}, although not satisfactory for our
goals, is quite relevant in a number of ways.  Among other things it implies our
axioms and hence Sieben's definition \cite{SiebenTwistedActions} of twisted actions
is a special case of ours (see Proposition~\ref{prop:SiebensTwistedActions}).

As our main application we consider an important class of Fell bundles over
inverse semigroups, obtained from twisted groupoids \cite[Section~2]{Kumjian:cstar.diagonals}, \cite[Section~4]{RenaultCartan}.
Given an étale groupoid $\G$ equipped with a twist $\Sigma$,
one may consider the associated complex line bundle $L =(\Sigma\times\C)/\Torus$ \cite[Example 5.5]{Deaconi_Kumjian_Ramazan:Fell.Bundles},
which is a Fell bundle over $\G$ in the sense of \cite{Kumjian:fell.bundles.over.groupoids}.
Therefore, applying the procedure outlined in \cite[Example 2.11]{BussExel:Fell.Bundle.and.Twisted.Groupoids}
to $(\G, L)$, we may form a Fell bundle $\{\A_s\}_s$, over the inverse
semigroup of all open bisections $s\subseteq \G$.

En passant, the class of such Fell bundles consists precisely of the \emph{semi-abelian} ones, namely those having commutative fibers over idempotent
semigroup elements \cite[Theorem~3.36]{BussExel:Fell.Bundle.and.Twisted.Groupoids}.

Recall from \cite[Example 2.11]{BussExel:Fell.Bundle.and.Twisted.Groupoids} that for each bisection $s\subseteq \G$,
$\A_s$ is defined as the space of all continuous sections of $L$ over $s$ vanishing at infinity.  Regularity
being the key hypothesis of our main Theorem, one should ask whether or not $\A_s$ is
regular.  The answer is affirmative when $L$ is topologically trivial
over $s$, because one may then choose a nonvanishing continuous section over $s$
which turns out to be associated to $\A_s$.

Restricting our Fell bundle to a subsemigroup consisting of \emph{small}
bisections, i.e.~bisections where $L$ is topologically trivial (recall that $L$ is necessarily locally trivial), we get a
regular one to which we may apply our main result, describing the bundle in question
in terms of a twisted inverse semigroup action.

As an immediate consequence (Theorem~\ref{theo:CorrespondenceTwistedGroupoidsAndTwistedActions}) we prove that the twisted groupoid
\cstar{}algebra is isomorphic to the crossed product of $\contz\big(\G^{(0)}\big)$ by a
twisted inverse semigroup action.  This simultaneously generalizes
\cite[Theorem~3.3.1]{Paterson:Groupoids}, \cite[Theorem~8.1]{Quigg.Sieben.C.star.actions.r.discrete.groupoids.and.inverse.semigroups} and \cite[Theorem~9.9]{Exel:inverse.semigroups.comb.C-algebras}.

Still speaking of Fell bundles arising from a twisted groupoid, we found a very
nice characterization of Sieben's condition (Proposition~\ref{prop:SiebenTwistedActions=TopologicalTrivial}): it holds if and only if the line
bundle $L$ is topologically trivial (although it needs not be trivial as a Fell bundle).

Our method sheds new light over the historic evolution of the notion of \emph{twisted groupoids}.
Indeed, recall that Renault \cite{RenaultThesis} first viewed twisted groupoids as those equipped with continuous two-cocycles,
although more recently the most widely accepted notion of a twist over a groupoid is that of a groupoid extension
\begin{equation*}
  \Torus\times\G^{(0)} \to \Sigma \to \G
\end{equation*}
\cite[Section~2]{Kumjian:cstar.diagonals}, \cite[Section~2]{Muhly.Williams.Continuous.Trace.Groupoid}, \cite[Section~4]{RenaultCartan}.
While a two-cocycle is known to give rise to a groupoid extension,
the converse is not true, the obstruction being that $\Sigma$ may be
topologically nontrivial as a circle bundle over $\G$ \cite[Example~2.1]{Muhly.Williams.Continuous.Trace.Groupoid}.

Our point of view, based on inverse semigroups, may be seen as unifying the \emph{cocycle} and the
\emph{extension} points of view, because the Fell bundle over
the inverse semigroup of \emph{small} bisections may always be described by a
\emph{cocycle}, namely the $\omega(s,t)$ of our twisted action, even when the
twist itself is topologically nontrivial and hence does
not come from a 2-cocycle in Renault's sense.  Nevertheless, when a twisted
étale groupoid does come from a 2-cocycle $\tau$ in Renault's sense, then
there is a close relationship between $\tau$ and our cocycle $\omega$ (Proposition~\ref{prop:RelationRenaultAndSiebensCocycles})

Most of our results are based on the regularity property of Fell bundles but in
some cases we can do without it.  The key idea behind this generalization is
based on the notion of \emph{refinement} (see Definition~\ref{def:refinement}).  In topology it
is often the case that assertions fail to hold globally, while holding locally.
In many of these circumstances one may successfully proceed after restricting
oneself to a covering of the space formed by small open sets where the assertion
in question holds.  The idea of Fell bundle refinement is the inverse semigroup equivalent
of this procedure.

Although the process of refining a Fell bundle changes everything, including the
base inverse semigroup, the cross-sectional algebras are unchanged (Theorem~\ref{theo:RefinementPreserveFullC*Algebras}),
as one would expect.  Moreover, if a semi-abelian Fell bundle $\A$ admits
$\B$ as a refinement, then the underlying twisted groupoids $(\G_\A,\Sigma_\A)$ and
$(\G_\B, \Sigma_\B)$ obtained by \cite[Section~3.2]{BussExel:Fell.Bundle.and.Twisted.Groupoids}
are isomorphic (Proposition~\ref{prop:RefinementPreserveTwistedGroupoids}).

A {\tro} is said to be \emph{locally regular} if it generated by its regular
ideals (Definition (2.8)).  Likewise a Fell bundle over an inverse semigroup is
said to be \emph{locally regular} if its fibers are locally regular as {\tros}.

Our main reason for considering locally regular Fell bundles is that all semi-abelian
ones possess this property (Proposition~\ref{prop:commutative=>loc.regular}). Furthermore, our motivation for
studying the notion of refinement is that every locally regular bundle admits a
saturated regular refinement (Proposition~\ref{prop:RefinementForLocRegularFellBundle}), even if the original bundle is
not saturated.  Therefore the scope of our theory is extended to include all locally
regular bundles.

\section{Regular ternary rings and imprimitivity bimodules}
\label{sec:regular TROs}

In this section we shall study a special class of ternary rings and imprimitivity bimodules named \emph{regular}.
This notion, which is the basis for our future considerations, has been first introduced in \cite{Exel:twisted.partial.actions}
building on ideas from \cite{BGR:MoritaEquivalence}.

Let $\hils$ be a Hilbert space, and let $\bound(\hils)$ denote the space of all bounded linear operators on $\hils$.
Recall that a ternary ring of operators is a closed subspace $\troa\sbe\bound(\hils)$ such that $\troa\troa^*\troa\sbe \troa$. It is a consequence that $\troa\troa^*\troa=\troa$ (see \cite[Corollary~4.10]{Exel:twisted.partial.actions}), where $\troa\troa^*\troa$ means the closed linear span of $\{xy^*z\colon x,y,z\in \troa\}$. Moreover, it is easy to see that $\troa^*\troa$ and $\troa\troa^*$ (closed linear spans) are \cstar{}subalgebras of $\bound(\hils)$. We refer to \cite{Exel:twisted.partial.actions} and \cite{Zettl:TROs} for more details on ternary rings.

\begin{lemma}\label{lem:ElementsAssociatedToTRO}
Let $\troa \sbe\bound(\hils)$ be a {\tro} and let $u\in \bound(\hils)$. Consider the following properties\textup:
\begin{enumerate}[(a)]
\item $\troa^*u=\troa^*\troa$;\label{lem:ElementsAssociatedToTRO:item:M*u=M*M}
\item $u\troa^*=\troa\troa^*$;\label{lem:ElementsAssociatedToTRO:item:uM*=MM*}
\item $uu^*\troa=\troa$;\label{lem:ElementsAssociatedToTRO:item:uu*M=M}
\item $\troa u^*u=\troa$.\label{lem:ElementsAssociatedToTRO:item:Mu*u=M}
\end{enumerate}
Then
\begin{enumerate}[(i)]
\item~\eqref{lem:ElementsAssociatedToTRO:item:M*u=M*M} and~\eqref{lem:ElementsAssociatedToTRO:item:uM*=MM*} imply \eqref{lem:ElementsAssociatedToTRO:item:uu*M=M};
\item~\eqref{lem:ElementsAssociatedToTRO:item:M*u=M*M} and~\eqref{lem:ElementsAssociatedToTRO:item:uM*=MM*} imply \eqref{lem:ElementsAssociatedToTRO:item:Mu*u=M};
\item~\eqref{lem:ElementsAssociatedToTRO:item:M*u=M*M} and~\eqref{lem:ElementsAssociatedToTRO:item:uu*M=M} imply \eqref{lem:ElementsAssociatedToTRO:item:uM*=MM*};
\item~\eqref{lem:ElementsAssociatedToTRO:item:uM*=MM*} and~\eqref{lem:ElementsAssociatedToTRO:item:Mu*u=M} imply \eqref{lem:ElementsAssociatedToTRO:item:uu*M=M}.
\end{enumerate}
\end{lemma}
\begin{proof}
\begin{enumerate}[(i)]
\item $uu^*\troa=u\troa^*\troa=\troa\troa^*\troa=\troa$.
\item $\troa u^*u=\troa \troa^* u=\troa\troa^*\troa=\troa$.
\item $u\troa^*=u\troa^*\troa\troa^*=uu^*\troa\troa^*=\troa\troa^*$.
\item $\troa^*u=\troa^*\troa\troa^*u=\troa^*\troa u^*u=\troa^*\troa$.
\end{enumerate}
\vskip-14pt
\end{proof}

\begin{definition}\label{def:RegularTRO}
We say that a {\tro} $\troa\sbe\bound(\hils)$ is regular if there is $u\in \bound(\hils)$ satisfying the
properties~\eqref{lem:ElementsAssociatedToTRO:item:M*u=M*M}-\eqref{lem:ElementsAssociatedToTRO:item:Mu*u=M} of the lemma above.
In this case, we say that $u$ is \emph{associated to $\troa$} and write $u\sim \troa$.
\end{definition}

By \cite[Proposition~5.3]{Exel:twisted.partial.actions}, the above definition is equivalent to Definition~5.1 in \cite{Exel:twisted.partial.actions}.

\begin{example}
Essentially, every regular {\tro} $\troa\sbe \bound(\hils)$ has the following form (this will become clearer as we proceed):
suppose $B\sbe \bound(\hils)$ is a \cstar{}subalgebra and suppose that $m$ is an element of $\bound(\hils)$
such that $m^*mB$ (closed linear span) is equal to $B$. Then it is easy to see that $\troa\defeq mB$ (closed linear span) is a regular {\tro}.
Note that $\troa^*\troa=B$ and that $A\defeq \troa\troa^*=uBu^*$ is a \cstar{}subalgebra of $\bound(\hils)$ which is isomorphic to $B$ (see also Corollary~\ref{cor:RegularIFFPartialIsometry} below).
\end{example}

\begin{proposition}\label{prop:PolarDecompositionAssociated}
Let $\troa\sbe\bound(\hils)$ be a {\tro}. Suppose $m\in \troa$ is associated to $\troa$ and let $m=u|m|$ be the polar decomposition of $m$.
Then $u$ is associated to $\troa$, $m^*$ and $u^*$ are associated to $\troa^*$,
$|m|$ and $m^*m$ are associated to $E^*E$, and $|m^*|$ and $mm^*$ are associated to $\troa\troa^*$.
\end{proposition}
\begin{proof}
Since $m\troa^*\troa=\troa$, we get $m^*m\troa^*\troa=m^*\troa=\troa^*\troa$. Taking adjoints, we also get $\troa^*\troa m^*m=\troa^*\troa$.
This means that $m^*m$ is associated to $\troa^*\troa$. This implies, in particular, that $m^*m$ and hence also $|m|=(m^*m)^{\frac{1}{2}}$
are multipliers of the \cstar{}algebra $\troa^*\troa$. Thus
\begin{equation*}
|m|\troa^*\troa\sbe \troa^*\troa= m^*m \troa^*\troa=|m|^2\troa^*\troa = |m||m|\troa^*\troa\sbe |m|\troa^*\troa.
\end{equation*}
This yields $\troa^*\troa|m|=|m|\troa^*\troa=\troa^*\troa$, so that $|m|$ is associated to $\troa^*\troa$.
Applying the involution to the properties~\eqref{lem:ElementsAssociatedToTRO:item:M*u=M*M}-\eqref{lem:ElementsAssociatedToTRO:item:Mu*u=M} of Lemma~\ref{lem:ElementsAssociatedToTRO}, it is easy to see that $n\in \bound(\hils)$ is associated to $\troa$ if and only if $n^*$ is associated to $\troa^*$. Thus $m^*$ is associated to $\troa^*$, and hence by the argument above,
$mm^*$ and $|m^*|$ are associated to $\troa\troa^*$. It remains to see that $u$ is associated to $\troa$. We have
\begin{equation*}
u^*\troa=u^*\troa\troa^*\troa=u^*m\troa^*\troa=|m|\troa^*\troa=\troa^*\troa
\end{equation*}
and (using that $um^*=|m^*|$)
\begin{equation*}
u\troa^*=u\troa^*\troa\troa^*=um^*\troa\troa^*=|m^*|\troa\troa^*=\troa\troa^*
\end{equation*}
\vskip-14pt
\end{proof}

Notice that if $u\sim \troa$, then
\begin{equation*}
\troa H=u\troa^*\troa H\sbe uH\quad\mbox{and}\quad \troa^*H=u^*\troa\troa^*H\sbe u^*H.
\end{equation*}
If both inclusions above are equalities, we say that $u$ is \emph{strictly associated to $\troa$} and write $u\ssim \troa$.
We can always assume that $u$ is strict by replacing $u$ by $up$, where
$p$ is the projection onto $\troa^*\troa H$. In fact, $upH=u\troa^*\troa H=\troa H$ and $up\sim \troa$ because
\begin{equation*}
up\troa^*=up\troa^*\troa\troa^*=u\troa^*=\troa\troa^*\quad\mbox{and}\quad \troa^*up=\troa^*\troa p=\troa^*\troa.
\end{equation*}
Given a \cstar{}subalgebra $A\sbe \bound(\hils)$, we write $1_A$ for the unit of the multiplier \cstar{}algebra $\mult(A)$ of $A$
(which we also represent in $\bound(\hils)$ is the usual way).

\begin{proposition}\label{prop:PartialIsometryStrictlyAssociated}
Let $\troa\sbe\bound(\hils)$ be a {\tro}, and let $u\in \bound(\hils)$ be a partial isometry associated to $\troa$.
Then the following assertions are equivalent\textup:
\begin{enumerate}[(i)]
\item $u$ is strictly associated to $\troa$;\label{prop:PartialIsometryStrictlyAssociated:item:uStrictlyAssociatedToM}
\item $\troa H=uH$ or $\troa^*H=u^*H$;\label{prop:PartialIsometryStrictlyAssociated:item:MH=uH_or_M*H=u*H}
\item $u^*u=1_{\troa^*\troa}$ and $uu^*=1_{\troa \troa^*}$;\label{prop:PartialIsometryStrictlyAssociated:item:u*u=1_and_uu*=1}
\item $u^*u=1_{\troa^*\troa}$ or $uu^*=1_{\troa \troa^*}$.\label{prop:PartialIsometryStrictlyAssociated:item:u*u=1_or_uu*=1}
\end{enumerate}
\end{proposition}
\begin{proof}
If $uH=\troa H$, then $p=u^*u$ is the projection onto
\begin{equation*}
u^*uH=u^*H=\troa^*H=\troa^*\troa H.
\end{equation*}
Thus, $px=x=xp$ for all $x\in \troa^*\troa$, so that $p=1_{\troa^*\troa}$.
Conversely, if $u^*u=1_{\troa^*\troa}$, then
\begin{equation*}
uH=uu^*uH=u1_{\troa^*\troa}H=u\troa^*\troa H=\troa\troa^*\troa H=\troa H.
\end{equation*}
Similarly, $u^*H=\troa H$ if and only if $uu^*=1_{\troa\troa^*}$.
Now if $\troa H=uH$, then
\begin{equation*}
u^*H=u^*uu^*H=u^*\troa\troa^*H=\troa^*\troa\troa^*H=\troa^*H.
\end{equation*}
Similarly, if $u^*H=\troa H$, then $u H=\troa H$. All these considerations imply the
equivalences~\eqref{prop:PartialIsometryStrictlyAssociated:item:uStrictlyAssociatedToM}$\Leftrightarrow$
~\eqref{prop:PartialIsometryStrictlyAssociated:item:MH=uH_or_M*H=u*H}$\Leftrightarrow$
~\eqref{prop:PartialIsometryStrictlyAssociated:item:u*u=1_and_uu*=1}$\Leftrightarrow$
~\eqref{prop:PartialIsometryStrictlyAssociated:item:u*u=1_or_uu*=1}.
\end{proof}

If $m=u|m|$ is the polar decomposition of an element $m\in \bound(\hils)$ associated to a {\tro} $\troa\sbe\bound(\hils)$, then
$m$ is strictly associated to $\troa$ if and only $u$ is strictly associated to $\troa$. In fact, this follows from the equalities
\begin{equation*}
uH=u|m|H=mH\quad\mbox{and}\quad u^*H=u^*|m^*|H=m^*H.
\end{equation*}
As a consequence of propositions~\ref{prop:PolarDecompositionAssociated} and~\ref{prop:PartialIsometryStrictlyAssociated}, we get the following:

\begin{corollary}\label{cor:RegularIFFPartialIsometry}
A {\tro} $\troa\sbe\bound(\hils)$ is regular if and only if there is a partial isometry $u\in \bound(\hils)$ satisfying
\begin{equation*}
u\troa^*\troa=\troa,\quad \troa\troa^*u=\troa,\quad u^*u=1_{\troa^*\troa}\quad\mbox{and}\quad uu^*=1_{\troa\troa^*}.
\end{equation*}
In particular, the map $a\mapsto u^*au$ is a \Star{}isomorphism from $\troa\troa^*$ to $\troa^*\troa$.
\end{corollary}

Later, we shall need the following fact:

\begin{proposition}\label{prop:strictlyAssociated=>weakClosure}
Let $\troa\sbe\bound(\hils)$ be a {\tro} and let $u\in \bound(\hils)$ be strictly associated to $\troa$.
Then $u$ lies in the weak closure of $\troa$ within $\bound(\hils)$.
\end{proposition}
\begin{proof} Let $\{e_i\}_i$ be an approximate unit for the \cstar{}algebra $\troa\troa^*$. We claim that
$u =\lim\limits_i e_iu$ in the strong topology of $\bound(\hils)$. In order to prove this let $\xi\in\hils$. Then
\begin{equation*}
u\xi\in u\hils = \troa\hils = \troa\troa^*\troa\hils\sbe \troa\troa^*\hils,
\end{equation*}
so by Cohen's Factorization Theorem, we may write $u\xi = a\eta$, with $a \in \troa\troa^*$ and $\eta\in \hils$. Therefore
\begin{equation*}
\lim\limits_i e_iu\xi =\lim\limits_i e_ia\eta = a\eta = u\xi,
\end{equation*}
thus proving our claim. It follows that $u$ belongs to the weak closure of $\troa\troa^*u=\troa\troa^*\troa = \troa$.
\end{proof}

Let $\troa$ be a {\tro}. An \emph{ideal} of $\troa$ is
a closed subspace $\trob\sbe\troa$ satisfying $\trob\troa^*\troa\sbe \trob$
and $\troa\troa^*\trob\sbe\trob$. Alternatively, $\trob$ is an ideal of $\troa$ if and only if $\trob$ is a sub-{\tro} of $\troa$
for which $\trob^*\trob$ is an ideal of $\troa^*\troa$ and $\trob\trob^*$ is an ideal of $\troa\troa^*$.
It is easy to see that our definition of ideals in {\tros} is equivalent to Definition~6.1 in \cite{Exel:twisted.partial.actions}. By Proposition~6.3 in \cite{Exel:twisted.partial.actions},
if $\troa$ is a regular {\tro}, then so is any ideal $\trob\sbe\troa$. Moreover, if $u\sim \troa$, then $u\sim\trob$.

\begin{definition}\label{def:TROLocallyRegular}
We say that a {\tro} $\troa$ is \emph{locally regular} if it generated by regular ideals
in the sense that there is a family $\{\troa_i\}_{i\in I}$ of ideals $\troa_i\sbe\troa$ such that each $\troa_i$ is a regular {\tro} and
$$\troa=\overline{\sum\limits_{i\in I}\troa_i}.$$
\end{definition}

\begin{proposition}\label{prop:commutative=>loc.regular}
Let $\troa\sbe\bound(\hils)$ be a {\tro}. If $A\defeq\troa\troa^*$ and $B\defeq\troa^*\troa$
are commutative \cstar{}algebras, then $\troa$ is locally regular.
\end{proposition}
\begin{proof}
Given $m\in \troa$, we show that the ideal $\gen{m}\sbe\troa$ generated by $m$,
namely $\gen{m}\defeq\troa\troa^*m\troa^*\troa$ (closed linear span) is regular.
Since $B=\troa^*\troa$ is commutative, we have
$$\gen{m}=AmB=\troa(\troa^*m)(\troa^*\troa)=\troa(\troa^*\troa)(\troa^*m)=\troa\troa^*m=Am.$$
Similarly, since $A=\troa\troa^*$ is commutative, we have $\gen{m}=mB$. Since $m\in \gen{m}$, this implies that $\gen{m}$ is regular.
\end{proof}

Although we have chosen to work with {\tro} concretely represented in Hilbert spaces, there is an abstract approach free of representations.
In fact, note that a {\tro} $\troa\sbe\bound(\hils)$ may be viewed as an imprimitivity Hilbert $A,B$-bimodule, where $A=\troa\troa^*$ and
$B=\troa^*\troa$ and the Hilbert bimodule structure (left $A$-action, right $B$-action and inner products) is given by the operations in $\bound(\hils)$. Conversely, any imprimitivity Hilbert $A,B$-bimodule $\F$ is isomorphic to some {\tro} $\troa\sbe\bound(\hils)$ (viewed as a bimodule). We refer the reader to \cite{Echterhoff.et.al.Categorical.Imprimitivity}
for more details on Hilbert modules.
Given Hilbert $B$\nb-modules $\F_1$ and $\F_2$,
we write $\K(\F_1,\F_2)$ and $\Ls(\F_1,\F_2)$ for the spaces of compact and adjointable operators $\F_1\to \F_2$, respectively.

The \emph{multiplier bimodule} of $\F$ is defined as $\mult(\F)\defeq\Ls(B,\F)$. This is, indeed, a Hilbert $\mult(A),\mult(B)$-bimodule with respect
to the canonical structure (see \cite[Section~1.5]{Echterhoff.et.al.Categorical.Imprimitivity} for details).

The notion of (local) regularity defined previously may be translated to the setting of Hilbert bimodules as follows.

\begin{definition}\label{def:RegularBimodule}
Let $\F$ be an imprimitivity Hilbert $A,B$-bimodule. We say that $\F$ \emph{regular} if there is a unitary multiplier
$u\in \mult(\F)$, that is, an adjointable operator $u\colon B\to \F$ such that $u^*u=1_{B}$ and $uu^*=1_{A}$ (here we tacitly identify
$\mult(A)\cong\Ls(\F)\cong\mult(\K(\F))$ in the usual way). We say that $\F$ is \emph{locally regular} if it is generated
by regular sub-$A,B$-bimodules, that is, if there is a family $\{\F_i\}_{i\in I}$ of sub-$A,B$-bimodules $\F_i\sbe\F$ such that
each $\F_i$ is is regular as an imprimitivity $A_i,B_i$-bimodule,
where $A_i=\cspn{_A}\braket{\F_i}{\F_i}$ and $B_i=\cspn\braket{\F_i}{\F_i}_B$, and
\begin{equation*}
\F=\overline{\sum\limits_{i\in I}\F_i}.
\end{equation*}
\end{definition}

If a (locally) regular Hilbert bimodule is concretely represented as a {\tro} on some Hilbert space, then this {\tro} is also
(locally) regular. Conversely, if a (locally) regular {\tro} is viewed as a Hilbert bimodule, this Hilbert bimodule is (locally) regular.
Of course, every result on (local) regularity of ternary rings of operators can be translated into an equivalent result on (local) regularity of
imprimitivity bimodules. For instance, a Hilbert sub-$A,B$-bimodule of a regular imprimitivity Hilbert $A,B$-bimodule $\F$ is again regular, and
if $A$ and $B$ are commutative, then $\F$ is locally regular.

We end this section discussing some examples. Given a \cstar{}algebra $B$, we may consider the \emph{trivial} $B,B$-imprimitivity bimodule $\F=B$
with the obvious structure. It is easy to see that it is regular. More generally, given \cstar{}algebras $A$ and $B$ and an isomorphism
$\phi\colon A\to B$, we may give $\F=B$ the structure of an imprimitivity Hilbert $A,B$-bimodule with the canonical structure:
\begin{align*}
a\cdot \xi & =\phi(a)b,\quad  \mbox{ (product in $B$) },\quad _A\braket{\xi}{\eta}=\phi\inv(\xi\eta^*)\quad\mbox{and}\\
\xi\cdot b & =\xi b\quad  \mbox{and}\quad  \braket{\xi}{\eta}_B=\xi^*\eta \mbox{ (product and involution of $B$) }
\end{align*}
for all $a\in A$ and $\xi,\eta,b\in B$. We write $_\phi B$ for this $A,B$-bimodule.
The associated multiplier $\mult(A),\mult(B)$-bimodule is just $\mult(\F)=\mult(B)$,
the multiplier algebra of $B$ viewed as a $\mult(A),\mult(B)$-bimodule using the (unique) strictly continuous extension
$\bar\phi\colon\mult(A)\to\mult(B)$ of $\phi$. In other words, $\mult(_\phi B)=_{\bar\phi}\mult(B)$. From this, we
see that $_\phi B$ is regular. In fact, it is enough to take a unitary multiplier of $B$ (for instance, one may take the unit $1_B\in \mult(B)$)
and view it as a unitary multiplier of $_\phi B$. Moreover, up to isomorphism, every regular $A,B$-bimodule has the
form $_\phi B$ for some isomorphism $\phi\colon A\to B$:

\begin{proposition}\label{prop:Regular=>ABisomorphicAndBimoduleTrivial}
An imprimitivity $A,B$-bimodule $\F$ is regular if and only if there is an isomorphism
$\phi\colon A\to B$ such that $\F\cong _\phi B$ as $A,B$-bimodules.
\end{proposition}
\begin{proof}
Let $u$ be a unitary multiplier of $\F$ such that $\F=A\cdot u=u\cdot B$. Then it easy to see that
the map $\phi\colon A\to B$ defined by $\phi(a)=u^*au$ is an isomorphism of \cstar{}algebras.
Moreover, the map $b\mapsto ub$ from $B$ to $\F$ induces an isomorphism $_\phi B\congto\F$ of $A,B$-bimodules.
\end{proof}

The above result says that the structure of a regular $A,B$-bimodule is, up to isomorphism, essentially trivial.
Notice that the imprimitivity $A,B$-bimodule $_\phi B$ (induced from an isomorphism $\phi\colon A\to B$) is isomorphic to
the trivial imprimitivity $B,B$-bimodule $B=_{\id}\!\!B$ (induced from the identity map $\id\colon B\to B$).
More precisely, the identity map $\id\colon B\to B$ together with the coefficient homomorphisms $\phi\colon A\to B$ and $\id\colon B\to B$
defines an isomorphism $_\phi\id_\id$ from the $A,B$-bimodule $_AB_B=_\phi B$ to the $B,B$-bimodule $_\id B={_B}B_B$
(see \cite[Definition~1.16]{Echterhoff.et.al.Categorical.Imprimitivity} for the precise meaning of homomorphisms between bimodules with possibly different coefficient algebras).

From the above characterization of regular bimodules, we can also easily give examples of non-regular bimodules,
simply taking any imprimitivity $A,B$-bimodule for which $A$ and $B$ are non-isomorphic \cstar{}algebras.
For instance, one can take a Hilbert space $\hils$ and view it as an imprimitivity $\K(\hils),\C$-bimodule.
Then $\hils$ is regular if and only if it is one-dimensional.

There is one special case where regularity is always present: this is the case where the coefficient algebras are \emph{stable}.
Recall that a \cstar{}algebra $A$ is stable if $A\otimes\K\cong A$, where $\K=\K(l^2\N)$. If $\F$ is an imprimitivity
Hilbert $A,B$-bimodule with $A$ and $B$ separable (or, more generally, $\sigma$-unital) stable \cstar{}algebras,
then $\F$ is regular (this follows from \cite[Theorem~3.4]{BGR:MoritaEquivalence}).
Hence, after stabilization (tensoring with $\K$), we can make any separable (or, more generally, countably generated)
imprimitivity Hilbert bimodule into a regular one.

Let us now consider the case where $A=B=\contz(X)$ is a commutative \cstar{}algebra,
where $X$ is some locally compact Hausdorff space. It is well-known that imprimitivity $\contz(X),\contz(X)$-bimodules correspond bijectively
to (Hermitian) complex line bundles over $X$ (see \cite[Appendix~A]{Raeburn:PicardGroup}):
given a complex line bundle $L$ over $X$, it can be endowed with
an Hermitian structure, that is, there is a continuous family of inner products $\braket{\cdot}{\cdot}_x$ (which we assume to be linear on the
second variable) on the fibers $L_x$. Then the space $\contz(L)$ of continuous sections of $L$ vanishing at infinity has a canonical structure of an
imprimitivity $\contz(X),\contz(X)$-bimodule with the obvious actions of $\contz(X)$ and the inner products:
\begin{equation*}
_{\contz(X)}\braket{\xi}{\eta}(x)=\braket{\eta(x)}{\xi(x)}_x\quad\mbox{and}\quad\braket{\xi}{\eta}_{\contz(X)}(x)=\braket{\xi(x)}{\eta(x)}_x.
\end{equation*}
When is $\contz(L)$ regular? To answer this question, let us first observe that the associated multiplier $\contb(X),\contb(X)$-bimodule
is (isomorphic to) $\contb(L)$, the space of continuous bounded sections of $L$ endowed with the canonical bimodule structure as above.
Thus $\contz(L)$ is regular if and only if there is a unitary element $u\in \U\contb(L)$, that is,
a continuous unitary section for $L$. In other words, we have proved the following:

\begin{proposition}\label{prop:Contz(L)RegularIFFLTopologicallyTrivial}
The Hilbert $\contz(X),\contz(X)$-bimodule $\contz(L)$ described above is regular if and only if the line bundle $L$ is topologically trivial.
\end{proposition}

The above result illustrates again the already mentioned triviality restriction imposed by regularity.
From this point of view, the study of regular bimodules seems to be trivial. However, the theory becomes interesting when one
studies bundles of such bimodules. In this work we are interested in regular Fell bundles (over inverse semigroups), meaning
Fell bundles for which all the fibers are regular bimodules. Although each fiber is then isolatedly isomorphic to a trivial bimodule,
this isomorphism is not canonical, and the relation between these isomorphisms among different fibers is what makes the object interesting.

On the other hand, local regularity is much more flexible. For instance, the bimodule $\contz(L)$ is always locally regular.
In fact, notice that we already know this from Proposition~\ref{prop:commutative=>loc.regular},
but essentially this is a manifestation of the fact that $L$ is locally trivial. It is also easy to give examples of non-locally trivial
bimodules. It is enough to take an imprimitivity Hilbert $A,B$-bimodule $\F$ for which $A$ and $B$ are non-isomorphic simple \cstar{}algebras.
In this case $\F$ is \emph{simple}, that is, there is no Hilbert sub-$A,B$-bimodule of $\F$, except for the trivial ones $\{0\}$ and $\F$.
Hence $\F$ is locally regular if and only if it is regular, and in this case $A$ and $B$ are isomorphic by Proposition~\ref{prop:Regular=>ABisomorphicAndBimoduleTrivial}.
As a simple example, consider any Hilbert space $\hils$ and view it as an imprimitivity $\K(\hils),\C$-bimodule.
This is not locally trivial, unless $\hils$ is one-dimensional.

\section{Regular Fell bundles}
\label{sec:regular Fell bundles}

First, let us recall the definition of Fell bundles over inverse semigroups
(a concept first introduced by Nándor Sieben in~\cite{SiebenFellBundles} and later used by the
second author in~\cite{Exel:noncomm.cartan}).

\begin{definition}\label{def:Fell bundles over ISG}
Let $S$ be an inverse semigroup. A \emph{Fell bundle} over $S$ is a
collection $\A=\{\A_s\}_{s\in S}$ of Banach spaces $\A_s$
together with a \emph{multiplication} $\cdot\colon\A\times\A\to \A$, an
\emph{involution} $^{*}\colon\A\to \A$, and isometric linear maps
$j_{t,s}\colon\A_s\to \A_t$ whenever $s\leq t$, satisfying the following properties\textup:
\begin{enumerate}[(i)]
\item $\A_s\cdot\A_t\sbe\A_{st}$ and the multiplication
is bilinear from $\A_s\times\A_t$ to $\A_{st}$ for all $s,t\in S$;\label{def:Fell bundles over ISG:item:MultiplicationBilinear}
\item the multiplication is associative, that is, $a\cdot(b\cdot c)=(a\cdot b)\cdot c$
    for all $a,b,c\in \A$;\label{def:Fell bundles over ISG:item:MultiplicationAssociative}
\item $\|a\cdot b\|\leq \|a\|\|b\|$ for all $a,b\in \A$;\label{def:Fell bundles over ISG:item:MultiplicationBounded}
\item $\A_s^*\sbe\A_{s^*}$ and the involution is conjugate linear
        from $\A_s$ to $\A_{s^*}$;\label{def:Fell bundles over ISG:item:InvolutionConjugateLinear}
\item $(a^*)^*=a$, $\|a^*\|=\|a\|$ and
        $(a\cdot b)^*=b^*\cdot a^*$;\label{def:Fell bundles over ISG:item:InvolutionIdempotentIsometricAntiMultiplicative}
\item $\|a^*a\|=\|a\|^2$ and $a^*a$ is a positive element of the \cstar{algebra} $\A_{s^*s}$
        for all $s\in S$ and $a\in \A_s$;\label{def:Fell bundles over ISG:item:C*-Condition}
\item if $r\leq s\leq t$ in $S$, then $j_{t,r}=j_{t,s}\circ j_{s,r}$;\label{def:Fell bundles over ISG:item:CompatibilityInclusions}
\item if $s\leq t$ and $u\leq v$ in $S$, then $j_{t,s}(a)\cdot j_{v,u}(b)=j_{tv,su}(a\cdot b)$
                        for all $a\in \A_s$ and $b\in \A_u$. In other words, the following diagram commutes:
\begin{equation*}
\xymatrix{
\A_s\times\A_u \ar[r]^{\mu_{s,u}}\ar[d]_{j_{t,s}\times j_{v,u}} & \A_{su}\ar[d]^{j_{tv,su}}\\
\A_t\times\A_v \ar[r]_{\mu_{t,v}} & \A_{tv}}
\end{equation*}
where $\mu_{s,u}$ and $\mu_{t,v}$
denote the multiplication maps.\label{def:Fell bundles over ISG:item:CompatibilityMultiplicationWithInclusions}
\item if $s\leq t$ in $S$, then $j_{t,s}(a)^*=j_{t^*,s^*}(a^*)$ for all $a\in \A_s$, that is, the diagram
\begin{equation*}
\xymatrix{
\A_s\ar[r]^{^{*}}\ar[d]_{j_{t,s}} & \A_{s^*}\ar[d]^{j_{t^*,s^*}}\\
\A_t\ar[r]_{^{*}} & \A_{t^*}
}
\end{equation*}
commutes.\label{def:Fell bundles over ISG:item:CompatibilityInvolutionWithInclusions}

If $\A_s\cdot\A_t$ spans a dense subspace of $\A_{st}$ for all $s,t\in S$, we say that $\A$ is \emph{saturated}.
\end{enumerate}
\end{definition}

Later, we shall need the following technical result:

\begin{lemma}\label{lem:TechnicalLemmaDefFellBundle}
Let $S$ be an inverse semigroup and let $\A=\{\A_s\}_{s\in S}$ be a collection of Banach spaces satisfying
all the conditions of Definition~\ref{def:Fell bundles over ISG} except,
possibly for axiom~\eqref{def:Fell bundles over ISG:item:CompatibilityMultiplicationWithInclusions}. Then the following
assertions are equivalent:
\begin{enumerate}[(a)]
\item axiom~\eqref{def:Fell bundles over ISG:item:CompatibilityMultiplicationWithInclusions}
                of Definition~\ref{def:Fell bundles over ISG}
                holds (and hence $\A$ is a Fell bundle);\label{lem:TechnicalLemmaDefFellBundle:item:CompatibilityMultiplicationWithInclusions}
\item $j_{t,s}(a)\cdot b=j_{tu,su}(a\cdot b)$ whenever $s,t,u\in S$, $s\leq t$, $a\in \A_s$ and $b\in \A_u$
        (that is, axiom~\eqref{def:Fell bundles over ISG:item:CompatibilityMultiplicationWithInclusions}
        holds for $u=v$);\label{lem:TechnicalLemmaDefFellBundle:item:CompatibilityMultiplicationWithInclusionsForu=v}
\item $a\cdot j_{v,u}(b)=j_{sv,su}(a\cdot b)$ whenever $s,u,v\in S$, $u\leq v$, $a\in \A_s$ and $b\in \A_u$ (that is,
        axiom~\eqref{def:Fell bundles over ISG:item:CompatibilityMultiplicationWithInclusions} holds
        for $s=t$);\label{lem:TechnicalLemmaDefFellBundle:item:CompatibilityMultiplicationWithInclusionsFors=t}
\end{enumerate}
\end{lemma}
\begin{proof}
Of course,~\eqref{lem:TechnicalLemmaDefFellBundle:item:CompatibilityMultiplicationWithInclusions}
implies~\eqref{lem:TechnicalLemmaDefFellBundle:item:CompatibilityMultiplicationWithInclusionsForu=v}
and~\eqref{lem:TechnicalLemmaDefFellBundle:item:CompatibilityMultiplicationWithInclusionsFors=t}
(note that $j_{s,s}$ is the identity map for all $s\in S$).
Applying the involution and using axioms~\eqref{def:Fell bundles over ISG:item:InvolutionIdempotentIsometricAntiMultiplicative}
and~\eqref{def:Fell bundles over ISG:item:CompatibilityInvolutionWithInclusions} of Definition~\ref{def:Fell bundles over ISG},
it follows that~\eqref{lem:TechnicalLemmaDefFellBundle:item:CompatibilityMultiplicationWithInclusionsForu=v}
and~\eqref{lem:TechnicalLemmaDefFellBundle:item:CompatibilityMultiplicationWithInclusionsFors=t} are equivalent.
Finally, suppose that ~\eqref{lem:TechnicalLemmaDefFellBundle:item:CompatibilityMultiplicationWithInclusionsForu=v}
and (hence also)~\eqref{lem:TechnicalLemmaDefFellBundle:item:CompatibilityMultiplicationWithInclusionsFors=t} hold.
This means that the diagrams
\begin{equation*}
\xymatrix{
\A_s\times\A_u \ar[r]^{\mu_{s,u}}\ar[d]_{j_{t,s}\times \id} & \A_{su}\ar[d]^{j_{tu,su}}\\
\A_t\times\A_u \ar[r]_{\mu_{t,u}} & \A_{tu}}
\quad\quad
\xymatrix{
\A_t\times\A_u \ar[r]^{\mu_{t,u}}\ar[d]_{\id\times j_{v,u}} & \A_{tu}\ar[d]^{j_{tv,tu}}\\
\A_t\times\A_v \ar[r]_{\mu_{t,v}} & \A_{tv}}
\end{equation*}
commute for all $s,t,u,v\in S$ with $s\leq t$ and $u\leq v$. Gluing these diagrams, we get the commutative diagram
\begin{equation*}
\xymatrix{
\A_s\times\A_u \ar[r]^{\mu_{s,u}}\ar[d]_{j_{t,s}\times \id} & \A_{su}\ar[d]^{j_{tu,su}}\\
\A_t\times\A_u \ar[r]_{\mu_{t,u}}\ar[d]_{\id\times j_{v,u}} & \A_{tu}\ar[d]^{j_{tv,tu}}\\
\A_t\times\A_v \ar[r]_{\mu_{t,v}} & \A_{tv}}
\end{equation*}
Since $(\id\times j_{v,u})\circ (j_{t,s}\times \id)=j_{t,s}\times j_{v,u}$ and
$j_{tv,tu}\circ j_{tu,su}=j_{tv,su}$ (see axiom~\eqref{def:Fell bundles over ISG:item:CompatibilityInclusions}
in Definition~\ref{def:Fell bundles over ISG}), this yields the desired commutativity of the diagram
in Definition~\ref{def:Fell bundles over ISG}\eqref{def:Fell bundles over ISG:item:CompatibilityMultiplicationWithInclusions}.
\end{proof}

We shall say that a Fell bundle $\A=\{\A_s\}_{s\in S}$ is \emph{concrete} if all the Banach spaces
$\A_s$ are concretely represented as operators on some Hilbert space $\hils$ in such a way that
all the algebraic operations of $\A$ are realized by the operations in $\bound(\hils)$ and, in addition, the inclusion maps
$A_s\into\A_t$ for $s\leq t$ become real inclusions $\A_s\sbe\A_t$ in $\bound(\hils)$.
In other words, $\A$ is a concrete Fell bundle in $\bound(\hils)$ if $\A_s\sbe\bound(\hils)$ for all $s\in S$ and
the inclusion map $\A\to \bound(\hils)$ defines a \emph{representation} of $\A$ (see
Section~\ref{sec:RepresentationsCrossedProducts} below and \cite{Exel:noncomm.cartan}
for more details on representations of Fell bundles). Note that, in this case, each fiber $\A_s\sbe\bound(\hils)$ is a {\tro}.

Observe that every Fell bundle is isomorphic to some concrete Fell bundle.
In fact, it is enough to take a faithful representation of $C^*(\A)$ on some Hilbert space $\hils$ and
consider the universal representation $\pi^u$ of $\A$ into $C^*(\A)\sbe\bound(\hils)$. By Corollary~8.9 in \cite{Exel:noncomm.cartan},
the components $\pi^u_s\colon\A_s\to\bound(\hils)$ of $\pi^u$ are injective. It follows that the collection $\{\pi^u_s(\A_s)\}_{s\in S}$ is a concrete Fell bundle in $\bound(\hils)$ which is isomorphic to $\A$. Thus there is no loss of generality in assuming that a Fell bundle is concrete.

\begin{definition}
Let $S$ be an inverse semigroup, and let $\A=\{\A_s\}_{s\in S}$ be a (concrete) Fell bundle.
We say that $\A$ is \emph{(locally) regular} if each fiber $\A_s\sbe\bound(\hils)$ is (locally) regular as an
imprimitivity Hilbert $I_s,J_s$-bimodule (or, equivalently, as a {\tro} if $\A$ is concrete),
where $I_s=\A_s\A_s^*$ and $J_s=\A_s^*\A_s$ (closed linear spans).
\end{definition}

\begin{remark}
We shall work almost exclusively with saturated Fell bundles in this work.
This is not a strong restriction since we are going to see later (see Proposition~\ref{prop:RefinementForLocRegularFellBundle}) that there is a way to
replace $\A$ by another Fell bundle $\B$ (over a different inverse semigroup) such that $\B$ is saturated and $\A$ and $\B$ have isomorphic
cross-sectional \cstar{}algebras. Note that for a saturated Fell bundle $\A$,
we have $\A_s\A_s^*=\A_{ss^*}$ and $\A_s^*\A_s=\A_{s^*s}$.
\end{remark}

\begin{example}
Consider the inverse semigroup $S=\{s,s^*,s^*s,ss^*,0\}$ with five elements,
where $s^2=0$ and $0$ is a zero for $S$.
A Fell bundle $\A$ over $S$ is essentially the same as a Hilbert bimodule. In fact, if $\F$ is a Hilbert $A,B$-module, we define
$\A_s=\F$, $\A_{s^*}=\F^*$ (the Hilbert $B,A$-bimodule \emph{dual} of $\F$; see \cite[Example~1.6(3)]{Echterhoff.et.al.Categorical.Imprimitivity}
for the precise definition), $\A_{s^*s}=B$, $\A_{ss^*}=A$ and $\A_0=\{0\}$.
Note that the only inequalities in the canonical order relation of $S$ are $0\leq t$ for all $t\in S$.
Thus the inclusion maps are trivial: they do not play any important role.
Conversely, if $\A$ is a Fell bundle over $S$, then $\F=\A_s$
is a Hilbert $A,B$-bimodule in the natural way, where $A=\A_{ss^*}$ and $B=\A_{s^*s}$.
All the aspects of the Fell bundle $\A$ can be described in terms of the Hilbert bimodule $\F$.
For instance, $\A$ is saturated iff $\F$ is an imprimitivity Hilbert $A,B$-bimodule
and $\A$ is (locally) regular iff $\F$ is.
\end{example}

Let $\A=\{\A_s\}_{s\in S}$ be a regular, saturated, concrete Fell bundle in $\bound(\hils)$. Then each fiber $\A_s$ becomes a
regular {\tro} in $\bound(\hils)$, so that, for each $s\in S$, we can choose a partial isometry $u_s\in\bound(\hils)$
which is strictly associated to $\A_s$. Moreover, if $e\in E(S)$ is idempotent,
then we may choose $u_e$ to be the unit of $\mult(\A_e)$ (the multiplier algebra of $\A_e\sbe\bound(\hils)$).
We are going to denote the unit of $\mult(\A_e)$ by $1_e$ to simplify the notation.

Let $A\defeq C^*(\A)$ be the (full) $C^*$-algebra of $\A$ and let $B$ be the $C^*$-algebra $C^*(\E)$, where $\E=\A|_{E}$ is the restriction of $\A$
to the idempotent semilattice $E=E(S)$ of $S$. We may assume that $B\sbe A\sbe\bound(\hils)$ (we could have taken $\hils$ to be the Hilbert space of a faithful representation of $A$ to start with) and that $\A_s\sbe A$ for all $s\in S$. Note that $\A_e$ is an ideal of $B$ for all $e\in E$.

With the assumptions above, we define the ideals $\D_s\defeq\A_s\A_s^*=\A_{ss^*}$ in $B$, $s\in S$, and the maps
\begin{equation*}
\beta_s\colon\D_{s^*}\to \D_s\quad\mbox{by }\beta_s(a)\defeq u_sau_s^* \mbox{ for all }a\in \D_{s^*}.
\end{equation*}
It is easy to see that $\beta_s$ is well-defined and, since $u_su_s^*$ is the unit of $\mult(\D_{s^*})$, $\beta_s$ is a \Star{}isomorphism.
Given $s,t\in S$, we also define $\omega(s,t)\defeq u_su_tu_{st}^*\in \bound(\hils)$. Note that $\omega(s,t)\in \uni\mult(\D_{st})$ (the set of unitary multipliers). In fact, by \cite[Proposition~6.5]{Exel:twisted.partial.actions} $\omega(s,t)$ is a unitary multiplier of
\begin{equation*}
\A_s\A_t\A_t^*\A_s^*=\A_{stt^*s^*}=\D_{st}.
\end{equation*}

\begin{proposition}\label{prop:properties of twisted action coming from partial isometries}
With the notations above, the following properties hold for all $r,s,t\in S$ and $e,f\in E(S)$\textup:
\begin{enumerate}[(i)]
\item
    the linear span of $\cup_{e\in E}\D_e$ is dense in $B$;
    \label{prop:properties of twisted action coming from partial isometries:item:D_e_Generate_B}
\item
    $\beta_e\colon\D_e\to\D_e$
    is the identity map;\label{prop:properties of twisted action coming from partial isometries:item:beta_e=Identity}
\item
    $\beta_r(\D_{r^*}\cap\D_s)=\D_{rs}$;\label{prop:properties of twisted action coming from partial isometries:item:beta_r(D_r*inter D_s)=D_rs}
\item
    $\beta_r\circ\beta_s=\Ad_{\omega(r,s)}\circ\beta_{rs}$;
    \label{prop:properties of twisted action coming from partial isometries:item:beta_r beta_s=Ad_omega(r,s)beta_rs}
\item
    $\beta_r(x\omega(s,t))\omega(r,st)=\beta_r(x)\omega(r,s)\omega(rs,t)$ whenever $x\in\D_{r^*}\cap\D_{st}$;
    \label{prop:properties of twisted action coming from partial isometries:item:CocycleCondition}
\item
    $\omega(e,f)=1_{ef}$ and $\omega(r,r^*r)=\omega(rr^*,r)=1_{rr^*}$;
    \label{prop:properties of twisted action coming from partial isometries:item:omega(e,f)=1_ef AND and omega(r,r^*r)=omega(rr^*,r)=1_rr*}
\item
    $\omega(s^*,e)\omega(s^*e,s)x=\omega(s^*,s)x$ for all $x\in \D_{s^*es}$;
    \label{prop:properties of twisted action coming from partial isometries:item:omega(s*,e)omega(s*e,s)x=omega(s*,s)x}
\item
    $\omega(t^*,s)=\omega(t^*,ss^*)\omega(s^*,s)$ whenever $s\leq t$;
    \label{prop:properties of twisted action coming from partial isometries:item:omega(t*,s)=omega(t*,ss*)omega(s*,s)}
\item
    $\omega(t,r^*r)=\omega(t,s^*s)\omega(s,r^*r)$ whenever $r\leq s\leq t$.
    \label{prop:properties of twisted action coming from partial isometries:item:omega(t,r^*r)=omega(t,s^*s)omega(s,r^*r)}
\end{enumerate}
\end{proposition}
\begin{proof}
Property~\eqref{prop:properties of twisted action coming from partial isometries:item:D_e_Generate_B}
is clear because $\D_e=\A_e$ for all $e\in E$. Since $u_e=1_e$ for all $e\in E$,~\eqref{prop:properties of twisted action coming from partial isometries:item:beta_e=Identity}
is also clear. To prove~\eqref{prop:properties of twisted action coming from partial isometries:item:beta_r(D_r*inter D_s)=D_rs}, note that
\begin{multline*}
\beta_r(\D_r^*\cap\D_s)=u_r\D_{r^*}\D_su_r^*=u_r\D_{r^*}\D_s\D_{r^*}u_r^*
\\=u_r\A_{r^*r}\A_{ss^*}\A_{r^*r}u_r^*=\A_r\A_{ss^*}\A_r^*=\A_{rss^*r^*}=\D_{rs}.
\end{multline*}
Note that $1_e1_f=1_{ef}$, so that
\begin{equation*}
\omega(e,f)=u_eu_fu_{ef}^*=1_e1_f1_{ef}^*=1_{ef}.
\end{equation*}
This together with the following calculation proves~\eqref{prop:properties of twisted action coming from partial isometries:item:omega(e,f)=1_ef AND and omega(r,r^*r)=omega(rr^*,r)=1_rr*}:
\begin{multline*}
\omega(r,r^*r)=u_ru_{r^*r}u_{rr^*r}^*=u_r1_{r^*r}u_r^*=u_ru_r^*=1_{rr^*}\\
=1_{rr^*}1_{rr^*}=1_{rr^*}u_ru_r^*=u_{rr^*}u_ru_{rr^*r}^*=\omega(rr^*,r).
\end{multline*}
To prove~\eqref{prop:properties of twisted action coming from partial isometries:item:beta_r beta_s=Ad_omega(r,s)beta_rs},
first note that the domains of $\beta_r\circ\beta_s$ and $\beta_{rs}$ coincide. In fact, by definition,
$\dom(\beta_{rs})=\D_{(rs)^*}=\A_{s^*r^*rs}$ and, on the other hand,
\begin{equation*}
\dom(\beta_r\circ\beta_s)=\{x\in \D_{s^*}\colon \beta_s(x)\in \D_s\cap\D_{r^*}\}=\beta_s\inv(\D_s\cap\D_{r^*}).
\end{equation*}
Now observe that
\begin{multline*}\beta_s\inv(\D_s\cap\D_{r^*})=\beta_s\inv(\D_s\D_{r^*})=u_s^*\A_{ss^*}\A_{r^*r}u_s=\A_{s^*}\A_{r^*r}u_s\\
=\A_{s^*}\A_{ss^*}\A_{r^*r}u_s=\A_{s^*}\A_{r^*r}\A_{ss^*}u_s=\A_{s^*r^*r}\A_s=\A_{s^*r^*rs}=\D_{(rs)^*}.
\end{multline*}
Thus $\dom(\beta_r\circ\beta_s)=\dom(\beta_{rs})=\dom(\Ad_{\omega(r,s)}\circ\beta_{rs})=\D_{(rs)^*}$.
Moreover, since $\omega(r,s)u_{rs}=u_ru_s$, for all $x\in \D_{(rs)^*}$ we get
\begin{equation*}
\beta_r(\beta_s(x))=u_r(u_sxu_s^*)u_r^*=\omega(r,s)u_{rs}xu_{rs}^*\omega(r,s)^*=\omega(r,s)\beta_{rs}(x)\omega(r,s)^*.
\end{equation*}
Item~\eqref{prop:properties of twisted action coming from partial isometries:item:CocycleCondition} is equivalent to the equality
\begin{equation}\label{eq:cocycle identity for partial isometries}
u_rxu_su_tu_{st}^*u_r^*u_ru_{st}u_{rst}^*=u_rxu_r^*u_ru_su_{rs}^*u_{rs}u_tu_{rst}^*.
\end{equation}
The left hand side of this equation equals $u_rxu_su_tu_{st}^*1_{r^*r}u_{st}u_{rst}^*$. Since $re\leq r$, where $e=(st)(st)^*$, we have
\begin{equation*}
u_rxu_su_tu_{st}^*\in u_r\D_{r^*}\D_{st}\mult(\D_{st})=\A_r\D_{st}=\A_{re}\sbe\A_r.
\end{equation*}
Thus
\begin{equation*}
u_rxu_su_tu_{st}^*1_{r^*r}u_{st}u_{rst}^*=u_rxu_su_tu_{st}^*u_{st}u_{rst}^*=u_rxu_su_t1_{(st)^*(st)}u_{rst}^*
\end{equation*}
Now note that
\begin{equation*}
xu_su_t\in \D_{st}u_su_t=A_{st}\A_{t^*}\A_{s^*}u_su_t=\A_{st}\A_{t^*s^*s}u_t\sbe \A_{st}\A_{t^*}u_t=\A_{st}\A_{t^*t}=\A_{st}.
\end{equation*}
Therefore, the left hand side of Equation~\eqref{eq:cocycle identity for partial isometries} equals
\begin{equation*}
u_rxu_su_t1_{(st)^*(st)}u_{rst}^*=u_rxu_su_tu_{rst}^*.
\end{equation*}
Similarly, the right hand side of Equation~\eqref{eq:cocycle identity for partial isometries} equals
\begin{multline*}
u_rxu_r^*u_ru_s1_{(rs)^*(rs)}u_tu_{rst}^*=u_rxu_r^*u_ru_su_tu_{rst}^*\\=u_rx1_{r^*r}u_su_tu_{rst}^*=u_rxu_su_tu_{rst}^*.
\end{multline*}
In order to prove~\eqref{prop:properties of twisted action coming from partial isometries:item:omega(s*,e)omega(s*e,s)x=omega(s*,s)x},
take $s\in S$, $e\in E(S)$ and $x\in \D_{s^*es}$. Note that
\begin{equation*}
u_sx\in  u_s\D_{s^*es}=u_s\A_{s^*es}=u_s\A_{s^*}\A_{es}=\A_{ss^*}\A_{es}=\A_{es}=\A_e\A_s
\end{equation*}
so that $1_eu_sx=u_sx$. Therefore
\begin{multline*}
\omega(s^*,e)\omega(s^*e,s)x=u_{s^*}u_eu_{s^*e}^*u_{s^*e}u_su_{s^*es}^*x=u_{s^*}1_e1_{ess^*}u_s1_{s^*es}x\\
=u_{s^*}1_e1_{ss^*}u_sx=u_{s^*}1_eu_sx=u_{s^*}u_sx=u_{s^*}u_s1_{s^*s}x=\omega(s^*,s)x.
\end{multline*}
To prove~\eqref{prop:properties of twisted action coming from partial isometries:item:omega(t*,s)=omega(t*,ss*)omega(s*,s)},
suppose $s\leq t$. Then $t^*s=s^*s$ and $t^*ss^*=s^*ss^*=s^*$, so that
\begin{multline*}
\omega(t^*,ss^*)\omega(s^*,s)=u_{t^*}u_{ss^*}u_{t^*ss^*}^*u_{s^*}u_su_{s^*s}^*=u_{t^*}1_{ss^*}u_{s^*}^*u_{s^*}u_s 1_{s^*s}=\\
u_{t^*}1_{ss^*}1_{ss^*}u_s1_{s^*s}=u_{t^*}u_{s}=u_{t^*}u_{s}1_{s^*s}^*=u_{t^*}u_{s}u_{s^*s}^*=u_{t^*}u_{s}u_{t^*s}^*=\omega(t^*,s).
\end{multline*}
Finally, to prove~\eqref{prop:properties of twisted action coming from partial isometries:item:omega(t,r^*r)=omega(t,s^*s)omega(s,r^*r)},
assume that $r\leq s\leq t$. Note that $\omega(t,s^*s)=u_tu_{s^*s}u_{ts^*s}^*=u_t1_{s^*s}u_s^*=u_tu_s^*$ and
$1_{s^*s}u_r^*=1_{s^*s}1_{r^*r}u_r^*=1_{r^*r}u_r^*=u_r$. Therefore
\begin{equation*}
\omega(t,s^*s)\omega(s,r^*r)=u_tu_s^*u_su_r^*=u_t1_{s^*s}u_r^*=u_tu_r^*=\omega(t,r^*r).
\end{equation*}
\vskip-16pt
\end{proof}

\section{Twisted actions}
\label{sec:TwistedActions}

In this section, we give a relatively simple definition of inverse semigroup twisted action
generalizing Busby-Smith twisted actions \cite{Busby-Smith:Representations_twisted_group} and closely related to twisted partial actions of
\cite{Exel:twisted.partial.actions} for (discrete) groups.
It also generalizes the twisted actions in the sense of Sieben \cite{SiebenTwistedActions} for (unital) inverse semigroups.
We then extend the main result in \cite{Exel:twisted.partial.actions}
proving that our twisted actions give rise to regular, saturated Fell bundles yielding in this way a
structure classification of such bundles.

\begin{definition}\label{def:twisted action}
A \emph{twisted action} of an inverse semigroup $S$ on a \cstar{algebra} $B$ is
a triple $\big(\{\D_s\}_{s\in S},\{\beta_s\}_{s\in S},\{\omega(s,t)\}_{s,t\in S}\big)$ consisting of a family of
(closed, two-sided) ideals $\D_s$ of $B$ whose linear span is dense in $B$, a family of
\Star{}isomorphisms $\beta_s\colon\D_{s^*}\to \D_{s}$, and a family $\{\omega(s,t)\}_{s,t\in S}$
of unitary multipliers $\omega(s,t)\in \U\mult(\D_{st})$ satisfying properties~\eqref{prop:properties of twisted action coming from partial isometries:item:beta_r beta_s=Ad_omega(r,s)beta_rs}-\eqref{prop:properties of twisted action coming from partial isometries:item:omega(s*,e)omega(s*e,s)x=omega(s*,s)x} of Proposition~\ref{prop:properties of twisted action coming from partial isometries}, that is,
for all $r,s,t\in S$ and $e,f\in E=E(S)$ we have:
\begin{enumerate}[(i)]
\item $\beta_r\circ\beta_s=\Ad_{\omega(r,s)}\circ\beta_{rs}$;\label{def:twisted action:item:beta_r beta_s=Ad_omega(r,s)beta_rs}
\item $\beta_r(x\omega(s,t))\omega(r,st)=\beta_r(x)\omega(r,s)\omega(rs,t)$
        whenever $x\in\D_{r^*}\cap\D_{st}$;\label{def:twisted action:item:CocycleCondition}
\item $\omega(e,f)=1_{ef}$ and $\omega(r,r^*r)=\omega(rr^*,r)=1_{r}$, where $1_r$ is the unit of $\mult(\D_r)$;
        \label{def:twisted action:item:omega(r,r*r)=omega(e,f)=1_ef_and_omega(rr*,r)=1}
\item $\omega(s^*,e)\omega(s^*e,s)x=\omega(s^*,s)x$
for all $x\in \D_{s^*es}$.\label{def:twisted action:item:omega(s*,e)omega(s*e,s)x=omega(s*,s)x}
\end{enumerate}
We sometimes write $(\beta,\omega)$ to refer to a twisted action, implicitly assuming that $\beta$ is a family of \Star{}isomorphisms
$\{\beta_s\}_{s\in S}$ between ideals $\D_{s^*}=\dom(\beta_s)$ and $\D_s=\ran(\beta_s)$, and
$\omega$ is a family $\{\omega(s,t)\}_{s,t\in S}$ of unitary multipliers $\omega(s,t)\in \U\mult(\D_{st})$.
\end{definition}

\begin{remark}
If $S$ has a unit $1$, then the condition in the above definition
that the closed linear span of the ideals $\D_e$ for $e\in E(S)$ is equal to $B$ is equivalent to the requirement $\D_1=B$
because $e\leq 1$ so that $\D_e\sbe\D_1$ for all $e\in E(S)$ (see Lemma~\ref{lem:ConsequencesDefTwistedAction}\eqref{lem:ConsequencesDefTwistedAction:item:D_r sbe D_s} below).

For a discrete group $G$, the only idempotent is the unit element $1\in G$. In this case,
axiom~\eqref{def:twisted action:item:omega(r,r*r)=omega(e,f)=1_ef_and_omega(rr*,r)=1} in the above definition is
equivalent to the condition $\omega(1,s)=\omega(s,1)=1_s$ for all $s\in G$.
Moreover, it is easy to see that this implies axiom~\eqref{def:twisted action:item:omega(s*,e)omega(s*e,s)x=omega(s*,s)x}.
Hence, in the group case, our definition of twisted action is the same
studied by Busby and Smith in \cite{Busby-Smith:Representations_twisted_group} (see also \cite{Packer-Raeburn:Stabilisation}).
It can also be seem as a special case of the twisted partial actions defined by the second named author in \cite{Exel:twisted.partial.actions},
where the condition $\beta_r(\D_{r\inv}\cap\D_s)=\D_r\cap\D_{rs}$ (axiom (b) in \cite[Definition~2.1]{Exel:twisted.partial.actions})
is replaced by $\beta_r(\D_{r\inv}\cap\D_s)=\D_{rs}$ (see Lemma~\ref{lem:ConsequencesDefTwistedAction}\eqref{lem:ConsequencesDefTwistedAction:item:beta_r(D_r* inter D_s)=D_rs} below).
\end{remark}

\begin{remark}
In \cite{SiebenTwistedActions}, Nándor Sieben has
considered a similar definition of twisted action (for unital inverse semigroups) in which the axiom
\begin{equation}\label{eq:SiebensCondition}
\omega(s,t)=1_{st}\quad\mbox{whenever }s\mbox{ or }t\mbox{ is an idempotent}
\end{equation}
is included, but we will see later (Section~\ref{sec:RelationToSiebensTwistedActions})
that this is too strong in general, that is, it can not be derived
from the axioms in Definition~\ref{def:twisted action} and that our definition really generalizes Sieben's one.
In some sense, the axioms~\eqref{def:twisted action:item:omega(r,r*r)=omega(e,f)=1_ef_and_omega(rr*,r)=1} and~\eqref{def:twisted action:item:omega(s*,e)omega(s*e,s)x=omega(s*,s)x} appearing in Definition~\ref{def:twisted action}
are designed to replace~\eqref{eq:SiebensCondition} in a compatible way.

Let us say some words about axiom~\eqref{def:twisted action:item:omega(s*,e)omega(s*e,s)x=omega(s*,s)x} in Definition~\ref{def:twisted action}.
First, note that it can be rewritten as
\begin{equation*}
\omega(s^*,e)\omega(s^*e,s)=\Res_{s^*es}(\omega(s^*,s)),
\end{equation*}
where $\Res_{s^*es}\colon\mult(\D_{s^*s})\to \mult(\D_{s^*es})$ is the restriction homomorphism.
Observe that $\D_{s^*es}$ is an ideal of $\D_{s^*s}$ because $s^*es\leq s^*s$ (see
Lemma~\ref{lem:ConsequencesDefTwistedAction}\eqref{lem:ConsequencesDefTwistedAction:item:D_r sbe D_s} below).
It is not clear for us, whether axiom~\eqref{def:twisted action:item:omega(s*,e)omega(s*e,s)x=omega(s*,s)x} is a
consequence of the others. The closest property we were able to prove is the following equality very similar
to~\eqref{def:twisted action:item:omega(s*,e)omega(s*e,s)x=omega(s*,s)x}:
\begin{equation}\label{eq:omega(e,s^*)omega(es*,s)x=omega(s*,s)x}
\omega(e,s^*)\omega(es^*,s)x=\omega(s^*,s)x\quad\mbox{for all }e\in E(S),\, s\in S\mbox{ and }x\in \D_{es^*s}.
\end{equation}
To prove this, we only need axioms~\eqref{def:twisted action:item:CocycleCondition},~\eqref{def:twisted action:item:omega(r,r*r)=omega(e,f)=1_ef_and_omega(rr*,r)=1} and the fact that $\beta_e$ is the identity on $\D_e$
(see Lemma~\ref{lem:ConsequencesDefTwistedAction}\eqref{lem:ConsequencesDefTwistedAction:item:beta_e=identity}).
In fact, note that $\omega(e,s^*)\omega(es^*,s)x$ and $\omega(s^*,s)x$ belong to $\D_{es^*s}=\D_e\D_{s^*s}\sbe\D_e$
(see Lemma~\ref{lem:ConsequencesDefTwistedAction} below). Moreover, if $y\in \D_{es^*s}$, then
\begin{multline*}
y\omega(s^*,s)x=\beta_e\big(y\omega(s^*,s)\big)1_{es^*s}x=\beta_e\big(y\omega(s^*,s)\big)\omega(e,s^*s)x\\
=\beta_e(y)\omega(e,s^*))\omega(es^*,s)x=y\omega(e,s^*))\omega(es^*,s)x.
\end{multline*}
Since $y$ was arbitrary, this implies~\eqref{eq:omega(e,s^*)omega(es*,s)x=omega(s*,s)x}.
As we have seen in Proposition~\ref{prop:properties of twisted action coming from partial isometries}, all the
axioms in Definition~\ref{def:twisted action} (and many others) are satisfied in case the cocycles $\omega(s,t)$ come from partial isometries $u_s$ associated to a regular, saturated Fell bundle as in Section~\ref{sec:regular Fell bundles}. Later, we are going to see (Proposition~\ref{prop:SiebensTwistedActions}) that both axioms~\eqref{def:twisted action:item:omega(r,r*r)=omega(e,f)=1_ef_and_omega(rr*,r)=1} and~\eqref{def:twisted action:item:omega(s*,e)omega(s*e,s)x=omega(s*,s)x} are automatically satisfied in the presence of~\eqref{def:twisted action:item:beta_r beta_s=Ad_omega(r,s)beta_rs},~\eqref{def:twisted action:item:CocycleCondition}
and, in addition, Sieben's condition~\eqref{eq:SiebensCondition}.
\end{remark}

The following result gives some consequences of the axioms of twisted action (compare with Proposition~\ref{prop:properties of twisted action coming from partial isometries}).

\begin{lemma}\label{lem:ConsequencesDefTwistedAction}
If $\bigl(\{\D_s\}_{s\in S},\{\beta_s\}_{s\in S},\{\omega(s,t)\}_{s,t\in S}\bigr)$ is a twisted action of $S$ on $B$, then
the following properties hold for all $r,s,t\in S$ and $e,f\in E(S)$\textup:
\begin{enumerate}[(i)]
\item $\D_s=\D_{ss^*}$;\label{lem:ConsequencesDefTwistedAction:item:Ds=Dss*}
\item $\beta_e\colon\D_e\to\D_e$ is the identity map;\label{lem:ConsequencesDefTwistedAction:item:beta_e=identity}
\item $\beta_r(\D_{r^*}\cap\D_s)=\D_{rs}$;\label{lem:ConsequencesDefTwistedAction:item:beta_r(D_r* inter D_s)=D_rs}
\item $\D_r\sbe\D_s$ if $r\leq s$;\label{lem:ConsequencesDefTwistedAction:item:D_r sbe D_s}
\item $\D_r\D_s=\D_{rr^*s}=\D_{ss^*r}$ and $\D_{rs}=\D_{rss^*}$.
In particular, $\D_e\D_f=\D_{ef}$;\label{lem:ConsequencesDefTwistedAction:item:D_rD_s=...}
\item $\beta_{s^*}=\Ad_{\omega(s^*,s)}\circ\beta_{s}\inv$;\label{lem:ConsequencesDefTwistedAction:item:beta_s*=Ad_omega(s*,s)beta_s inv}
\item $\beta_t\rest{\D_s}=\Ad_{\omega(t,s^*s)}\circ\beta_s$
            whenever $s\leq t$;\label{lem:ConsequencesDefTwistedAction:item:beta_t restriction=...}
\item $\beta_s(\omega(s^*,s))=\omega(s,s^*)$. Here we have implicitly extended
$\beta_s\colon\D_{s^*}\to\D_{s}$ to the multiplier algebras $\beta_s\colon\mult(\D_{s^*})\to\mult(\D_{s})$;\label{lem:ConsequencesDefTwistedAction:item:beta_s(omega(s*,s))=omega(s,s*)}
\item $\beta_r(x\omega(s,t)^*)\omega(r,s)=\beta_r(x)\omega(r,st)\omega(rs,t)^*$ for all $x\in\D_{r^*}\cap\D_{st}$;\label{lem:ConsequencesDefTwistedAction:item:CocycleConditionWith*}
\item $\omega(s,e)=\omega(s,s^*se)$ and $\omega(e,s)=\omega(ess^*,s)$;\label{lem:ConsequencesDefTwistedAction:item:omega(s,e)=w(s,s*se)...}
\item $\omega(r,e)=1_{rr^*}$ whenever $e\geq r^*r$,
                     and $\omega(f,s)=1_{ss^*}$ whenever $f\geq ss^*$;\label{lem:ConsequencesDefTwistedAction:item:omega(r,e)=1_If_eGEQr*r}
\item $\omega(t^*,s)=\omega(t^*,ss^*)\omega(s^*,s)$ whenever $s\leq t$;\label{lem:ConsequencesDefTwistedAction:item:omega(t*,s)=...sLEQt}
\item $\omega(t,r^*r)x=\omega(t,s^*s)\omega(s,r^*r)x$ whenever $r\leq s\leq t$
            and $x\in \D_r$.\label{lem:ConsequencesDefTwistedAction:item:omega(t,r*r)x=omega(t,s*s)omega(s,r*r)x}
\end{enumerate}
\end{lemma}
\begin{proof}
Note that $\D_s=\dom\big(\beta_s\circ\beta_{s^*}\big)=\dom\big(\Ad_{\omega(s,s^*)}\circ\beta_{ss^*}\big)=\D_{ss^*}$,
which proves~\eqref{lem:ConsequencesDefTwistedAction:item:Ds=Dss*}.
To prove~\eqref{lem:ConsequencesDefTwistedAction:item:beta_e=identity},
observe that $\beta_e\circ\beta_e=\Ad_{\omega(e,e)}\circ\beta_e=\beta_e$ because $\omega(e,e)=1_e$.
Since $\beta_e\colon\D_e\to \D_e$ is an automorphism, it must be the identity map.
To check~\eqref{lem:ConsequencesDefTwistedAction:item:beta_r(D_r* inter D_s)=D_rs}, notice that
\begin{equation*}
\beta_r(\D_{r^*}\cap\D_s)=\ran(\beta_r\circ\beta_s)=\ran(\Ad_{\omega(r,s)}\circ\beta_{rs})=\D_{rs}.
\end{equation*}
To prove~\eqref{lem:ConsequencesDefTwistedAction:item:D_r sbe D_s},
first observe that $sr^*r=r$ because $r\leq s$. Using (the just checked)
property~\eqref{lem:ConsequencesDefTwistedAction:item:beta_r(D_r* inter D_s)=D_rs}, we get
\begin{equation*}
\D_r=\D_{sr^*r}=\beta_s(\D_{s^*}\cap\D_{r^*r})\sbe\beta_s(\D_{s^*})=\D_s.
\end{equation*}
To prove~\eqref{lem:ConsequencesDefTwistedAction:item:D_rD_s=...} we use that $\beta_{rr^*}$ is the identity
on $\D_{rr^*}$, and also~\eqref{lem:ConsequencesDefTwistedAction:item:Ds=Dss*}
and~\eqref{lem:ConsequencesDefTwistedAction:item:beta_r(D_r* inter D_s)=D_rs} to get
\begin{equation*}
\D_r\D_s=\D_{rr^*}\D_s=\beta_{rr^*}(\D_{rr^*}\D_s)=\D_{rr^*s}.
\end{equation*}
Since $\D_r\D_s=\D_r\cap\D_s=\D_s\D_t$ (this holds for ideals of a \cstar{}algebra), we also have $\D_r\D_s=\D_{ss^*r}$.
And by~\eqref{lem:ConsequencesDefTwistedAction:item:Ds=Dss*}, we have $\D_{rss^*}=\D_{rss^*ss^*r^*}=\D_{rss^*r^*}=\D_{rs}$.
To prove~\eqref{lem:ConsequencesDefTwistedAction:item:beta_s(omega(s*,s))=omega(s,s*)},
we use axioms~~\eqref{def:twisted action:item:CocycleCondition}
and~\eqref{def:twisted action:item:omega(r,r*r)=omega(e,f)=1_ef_and_omega(rr*,r)=1} in Definition~\ref{def:twisted action} to get
\begin{multline*}
\beta_s(\omega(s^*,s))=\beta_s(\omega(s^*,s))1_s=\beta_s(\omega(s^*,s))\omega(s,s^*s)\\
=\omega(s,s^*)\omega(ss^*,s)=\omega(s,s^*)1_s=\omega(s,s^*).
\end{multline*}
Property~\eqref{lem:ConsequencesDefTwistedAction:item:CocycleConditionWith*} is a consequence
of~\eqref{def:twisted action:item:CocycleCondition} in Definition~\ref{def:twisted action}. In fact,
applying Definition~\ref{def:twisted action}\eqref{def:twisted action:item:CocycleCondition} with $x\omega(s,t)^*$ in place of $x$, we get
\begin{equation*}
\beta_r(x)\omega(r,st)=\beta_r(x\omega(s,t)^*)\omega(r,s)\omega(rs,t).
\end{equation*}
Multiplying this last equation by $\omega(rs,t)^*$ on the right one arrives at~\eqref{lem:ConsequencesDefTwistedAction:item:CocycleConditionWith*}.
In order to check~\eqref{lem:ConsequencesDefTwistedAction:item:omega(s,e)=w(s,s*se)...},
take $y\in \D_{s^*se}$ and define $x\defeq \beta_s(y)\in \D_{ses^*}$ (note that every element
of $\D_{ses^*}$ has this form for some $y\in \D_{s^*se}$). Note that $\omega(s,s^*se)$ and $\omega(s,e)$ are unitary
multipliers of $\D_{se}=\D_{ses^*}$. Using axioms~\eqref{def:twisted action:item:CocycleCondition}
and~\eqref{def:twisted action:item:omega(r,r*r)=omega(e,f)=1_ef_and_omega(rr*,r)=1} in Definition~\ref{def:twisted action}, we get
\begin{multline*}
x\omega(s,s^*se)=\beta_s(y)\omega(s,s^*se)=\beta_s(y1_{s^*se})\omega(s,s^*se)\\
=\beta_s(y\omega(s^*s,e))\omega(s,s^*se)=\beta_s(y)\omega(s,s^*s)\omega(ss^*s,e)=x1_{ss^*}\omega(s,e)=x\omega(s,e).
\end{multline*}
Since $x$ is an arbitrary element of $\D_{ses^*}$, we must have $\omega(s,s^*se)=\omega(s,e)$. Similarly, we
can check the second part of~\eqref{lem:ConsequencesDefTwistedAction:item:omega(s,e)=w(s,s*se)...}: note that
$\omega(e,s)$ and $\omega(ess^*,s)$ are multipliers of $\D_{es}=\D_{ess^*}$.
Take an arbitrary element $x\in \D_{ess^*}=\D_e\cap\D_{ss^*}$. Then, using
again~\eqref{def:twisted action:item:CocycleCondition} and~\eqref{def:twisted action:item:omega(r,r*r)=omega(e,f)=1_ef_and_omega(rr*,r)=1}
in Definition~\ref{def:twisted action} and that $\beta_e$ is the identity map on $\D_e$, we get
\begin{multline*}
x\omega(e,s)=\beta_e(x1_{ss^*})\omega(e,s)=\beta_e(x\omega(ss^*,s))\omega(e,s)\\
=\beta_e(x)\omega(e,ss^*)\omega(ess^*,s)=x1_{ess^*}\omega(ess^*,s)=x\omega(ess^*,s).
\end{multline*}
Therefore $\omega(e,s)=\omega(ess^*,s)$. To prove~\eqref{lem:ConsequencesDefTwistedAction:item:omega(r,e)=1_If_eGEQr*r},
take $r\in S$ and $e\in E(S)$ with $e\geq r^*r$. Then, by
Definition~\ref{def:twisted action}\eqref{def:twisted action:item:omega(r,r*r)=omega(e,f)=1_ef_and_omega(rr*,r)=1} and the property~\eqref{lem:ConsequencesDefTwistedAction:item:omega(s,e)=w(s,s*se)...} just checked, we
have $\omega(r,e)=\omega(r,r^*re)=\omega(r,r^*r)=1_{r}=1_{rr^*}$. The second part of~\eqref{lem:ConsequencesDefTwistedAction:item:omega(r,e)=1_If_eGEQr*r} is proved similarly.
In order to prove~\eqref{lem:ConsequencesDefTwistedAction:item:omega(t*,s)=...sLEQt}, assume that $s\leq t$ and take
$x\in \D_s\sbe\D_t$. Note that $y\defeq \beta_{t^*}(x)\in \D_{s^*}=\D_{s^*s}$ and every element of $\D_{s^*s}$ has this form.
Observe that $\omega(t^*,s)$, $\omega(t^*,ss^*)$ and $\omega(s^*,s)$ are unitary multipliers of
$\D_{t^*s}=\D_{s^*s}$ and by axioms~\eqref{def:twisted action:item:CocycleCondition}
and~\eqref{def:twisted action:item:omega(r,r*r)=omega(e,f)=1_ef_and_omega(rr*,r)=1}, we have
\begin{multline*}
y\omega(t^*,s)=\beta_{t^*}(x1_{ss^*})\omega(t^*,s)=\beta_{t^*}(x\omega(ss^*,s))\omega(t^*,s)\\
=\beta_{t^*}(x)\omega(t^*,ss^*)\omega(s^*,s)=y\omega(t^*,ss^*)\omega(s^*,s).
\end{multline*}
Thus $\omega(t^*,s)=\omega(t^*,ss^*)\omega(s^*,s)$. Finally, to prove~\eqref{lem:ConsequencesDefTwistedAction:item:omega(t,r*r)x=omega(t,s*s)omega(s,r*r)x},
assume that $r\leq s\leq t$ and $x\in \D_r=\D_{rr^*}$. Observe that $\omega(t,r^*r)x$ and $\omega(t,s^*s)\omega(s,r^*r)x$ are elements
of $\D_r$ every element of $\D_r$ has the form $z=\beta_t(y)$ for some $y\in \D_{r^*}=\D_{r^*r}$.
By axioms~\eqref{def:twisted action:item:CocycleCondition} and~\eqref{def:twisted action:item:omega(r,r*r)=omega(e,f)=1_ef_and_omega(rr*,r)=1} in
Definition~\ref{def:twisted action}, we have
\begin{multline*}
z\omega(t,r^*r)x=\beta_t(y1_{r^*r})\omega(t,r^*r)x=\beta_t(y\omega(s^*s,r^*r))\omega(t,r^*r)x\\
=\beta_t(y)\omega(t,s^*s)\omega(ts^*s,r^*r)x=z\omega(t,s^*s)\omega(s,r^*r)x.
\end{multline*}
Therefore $\omega(t,r^*r)x=\omega(t,s^*s)\omega(s,r^*r)x$ as desired.
\end{proof}

Given a twisted action $\big(\{\D_s\}_{s\in S},\{\beta_s\}_{s\in S},\{\omega(s,t)\}_{s,t\in S}\big)$,
we would like to define a Fell bundle $\B$ over $S$ as follows:
\begin{equation}\label{eq:DefFellBundleFromTwistedAction}
\B\defeq \{(b,s)\in B\times S\colon b\in \D_s\}
\end{equation}
Writing $b\delta_s$ for $(b,s)\in \B$, we define the operations:
\begin{equation}\label{eq:DefProductFellBundleFromTwistedAction}
(b_s\delta_s)\cdot(b_t\delta_t)\defeq\beta_s\big(\beta_s^{-1}(b_s)b_t\big)\omega(s,t)\delta_{st}
\end{equation}
and
\begin{equation}\label{eq:DefInvolutionFellBundleFromTwistedAction}
(b_s\delta_s)^*\defeq \beta_s^{-1}(b_s^*)\omega(s^*,s)^*\delta_{s^*}
\end{equation}
for all $b_s\in \D_s$ and $b_t\in \D_t$. The product~\eqref{eq:DefProductFellBundleFromTwistedAction} is well-defined because
$b_s\in \D_s=\dom(\beta_s\inv)$ and hence $\beta_s\inv(b_s)b_t\in \D_{s^*}\D_t\sbe\D_{s^*}=\dom(\beta_s)$. Moreover,
axiom~\eqref{def:twisted action:item:beta_r beta_s=Ad_omega(r,s)beta_rs} of Definition~\ref{def:twisted action} shows that
\begin{equation*}
\beta_s\big(\beta_s^{-1}(b_s)b_t\big)\in \beta_s(\D_{s^*}\D_t)=\beta_s(\D_{s^*}\cap\D_t)=\D_{st}.
\end{equation*}
Since $\omega(s,t)\in \mult(\D_{st})$, we get $\beta_s\big(\beta_s^{-1}(b_s)b_t\big)\omega(s,t)\in \D_{st}$.
It is easy to see that the involution~\eqref{eq:DefInvolutionFellBundleFromTwistedAction} is also well-defined.
Moreover, by Lemma~\ref{lem:ConsequencesDefTwistedAction}\eqref{lem:ConsequencesDefTwistedAction:item:beta_s*=Ad_omega(s*,s)beta_s inv}, $\beta_s\inv(a)=\omega(s^*,s)^*\beta_{s^*}(a)\omega(s^*,s)$, so that
\begin{equation}\label{eq:OtherFormulaForInvolution}
(a\delta_s)^*=\beta_s\inv(a^*)\omega(s^*,s)^*\delta_{s^*}=\omega(s^*,s)^*\beta_{s^*}(a^*)\delta_{s^*}.
\end{equation}
Observe that the formulas~\eqref{eq:DefProductFellBundleFromTwistedAction} and~\eqref{eq:DefInvolutionFellBundleFromTwistedAction}
are exactly the same ones appearing in \cite[Section~2]{Exel:twisted.partial.actions}
for a Fell bundle defined from a twisted partial action of a group.

The main difficulty is to define the inclusion maps $j_{t,s}\colon\B_s\hookrightarrow \B_t$ whenever $s\leq t$.
One first obvious choice would be $j_{t,s}(b_s\delta_s)=b_s\delta_t$, but this does not
work in general although it would work if we had Sieben's condition~\eqref{eq:SiebensCondition}.
The problem is to prove that the inclusion maps $j_{t,s}$ are compatible with the operations above.
To motivate the correct definition, let us temporarily assume that
the twisted action comes from partial isometries $u_s$ associated to a regular, saturated, concrete Fell bundle $\A=\{\A_s\}_{s\in S}$
in $\bound(H)$ as in Section~\ref{sec:regular Fell bundles}. In this case, if $s\leq t$ in $S$,
then $\A_s\sbe \A_t\sbe\bound(H)$. And by regularity, $\A_s=\D_su_s$ and $\A_t=\D_tu_t$.
So, any element $x\in \A_s$ can be written in two ways: $x=au_s=bu_t$ for some $a\in \D_s$ and $b\in \D_t$. The relation between $a$ and $b$ is
$b=au_su_t^*$. Now note that $u_su_t^*=u_su_{s^*s}u_t^*=(u_tu_{s^*s}u_{ts^*s}^*)^*=\omega(t,s^*s)^*$, so that $b=a\omega(t,s^*s)^*$. Thus
the inclusion $\A_s\sbe\A_t$ determines a map $\D_s\to \D_t$ by the rule $a\mapsto a\omega(t,s^*s)^*$. Hence, it is natural to define
\begin{equation}\label{eq:DefInclusionsFellBundleFromTwistedAction}
j_{t,s}\colon\B_s\to \B_t\quad\mbox{by}\quad j_{t,s}(a\delta_s)\defeq a\omega(t,s^*s)^*\delta_t.
\end{equation}
Note that $\omega(t,s^*s)\in \mult(\D_{ts^*s})=\mult(\D_s)$ and hence $a\omega(t,s^*s)\in \D_s\sbe\D_t$ by
Lemma~\ref{lem:ConsequencesDefTwistedAction}\eqref{lem:ConsequencesDefTwistedAction:item:D_r sbe D_s}, so that $j_{t,s}$ is well-defined.

\begin{theorem}\label{theo:FellBundleFromTwistedAction}
The bundle $\B$ defined by Equation~\eqref{eq:DefFellBundleFromTwistedAction} with the obvious
projection $\B\onto S$, with the linear and norm structure on each fiber $\B_s$ inherited from $\D_s$, with the algebraic
operations~\eqref{eq:DefProductFellBundleFromTwistedAction} and~\eqref{eq:DefInvolutionFellBundleFromTwistedAction},
and the inclusion maps~\eqref{eq:DefInclusionsFellBundleFromTwistedAction} is a saturated, regular Fell bundle over $S$.
\end{theorem}
\begin{proof}
Of course, the multiplication~\eqref{eq:DefProductFellBundleFromTwistedAction} is a bilinear map from $\A_s\times\A_t$ to $\A_{st}$
and the involution~\eqref{eq:DefInvolutionFellBundleFromTwistedAction} is a conjugate-linear map from $\A_s$ to $\A_{s^*}$ for all $s,t\in S$.
The associativity of the multiplication is proved in the same way as in \cite[Proposition~2.4]{Exel:twisted.partial.actions} (see
also the proof of Proposition~3.1 in \cite{SiebenTwistedActions}; it is also possible to use the idea appearing in \cite[Theorem~2.4]{DokuchaevExelSimon:twisted.partial.actions} where no approximate unit is needed).
Moreover, it is easy to see that $\|(a\delta_s)\cdot(b\delta_t)\|\leq\|a\delta_s\|\|b\delta_t\|$ for all $s,t\in S$, $a\in \D_s$ and $b\in \D_t$.
All this together gives us the axioms~\eqref{def:Fell bundles over ISG:item:MultiplicationBilinear},~\eqref{def:Fell bundles over ISG:item:MultiplicationAssociative},~\eqref{def:Fell bundles over ISG:item:MultiplicationBounded} and~\eqref{def:Fell bundles over ISG:item:InvolutionConjugateLinear} in Definition~\ref{def:Fell bundles over ISG}.

Let us check that $\big((a\delta_s)^*\big)^*=a\delta_s$ for all $s\in S$ and $a\in \D_s$. By
Equation~\eqref{eq:OtherFormulaForInvolution}, we have
\begin{align*}
\big((a\delta_s)^*\big)^*&=\big(\beta_s\inv(a^*)\omega(s^*,s)^*\delta_{s^*}\big)^*
                        =\omega(s,s^*)^*\beta_s(\omega(s^*,s)\beta_s\inv(a))\delta_s\\
                        &=\omega(s,s^*)^*\beta_s(\omega(s^*,s))a\delta_s
                        =\omega(s,s^*)^*\omega(s,s^*)a\delta_s=a\delta_s.
\end{align*}
Next, we prove that the involution on $\B$ is anti-multiplicative, that is, we show that
$\big((a\delta_s)\cdot(b\delta_t)\big)^*=(b\delta_t)^*\cdot (a\delta_s)^*$ for all $a\in \D_s$ and $b\in \D_t$.
We have
\begin{multline*}
\big((a\delta_s)\cdot(b\delta_t)\big)^*=\big(\beta_s(\beta_t\inv(a)b)\omega(s,t)\delta_{st}\big)^*\\
=\beta_{st}\inv\big(\omega(s,t)^*\beta_s(b^*\beta_s\inv(a^*))\big)\omega(t^*s^*,st)^*\delta_{t^*s^*}=\ldots
\end{multline*}
Let $x\defeq b^*\beta_s\inv(a^*)\in \D_t\cap\D_{s^*}$, so that $\beta_t\inv(x)\in \D_{t^*}\cap\D_{t^*s^*}$ and hence
\begin{equation*}
\beta_s(x)=\beta_s(\beta_t(\beta_t\inv(s)))=\omega(s,t)\beta_{st}(\beta_t\inv(x))\omega(s,t)^*.
\end{equation*}
Thus, the above equals
\begin{multline*}
\ldots=\beta_{st}\inv\big(\beta_{st}(\beta_t\inv(x))\omega(s,t)^*\big)\omega(t^*s^*,st)^*\delta_{t^*s^*}\\
        =\omega(t^*s^*,st)^*\beta_{t^*s^*}\big(\beta_{st}(\beta_t\inv(x))\omega(s,t)^*\big)\delta_{t^*s^*}=\ldots
\end{multline*}
which by Lemma~\ref{lem:ConsequencesDefTwistedAction}\eqref{lem:ConsequencesDefTwistedAction:item:CocycleConditionWith*} is equal to
\begin{align*}
\ldots=\omega(t^*s^*,st)^*&\beta_{t^*s^*}\big(\beta_{st}(\beta_t\inv(x))\big)\omega(t^*s^*,st)\omega(t^*s^*s,t)^*\omega(t^*s^*,s)^*\delta_{t^*s^*}\\
    &=\beta_{st}\inv\big(\beta_{st}(\beta_t\inv(x))\big)\omega(t^*s^*s,t)^*\omega(t^*s^*,s)^*\delta_{t^*s^*}\\
    &=\beta_t\inv(x)\omega(t^*s^*s,t)^*\omega(t^*s^*,s)^*\delta_{t^*s^*}\\
    &=\beta_t\inv(b^*\beta_s\inv(a^*))\omega(t^*s^*s,t)^*\omega(t^*s^*,s)^*\delta_{t^*s^*}.
\end{align*}
On the other hand,
\begin{align*}
(b\delta_t)^*\cdot (a\delta_s)^*&=\big(\beta_t\inv(b^*)\omega(t^*,t)^*\delta_{t^*}\big)\cdot\big(\beta_s\inv(a^*)\omega(s^*,s)^*\delta_{s^*}\big)\\
    &=\beta_{t^*}\Big(\beta_{t^*}\inv\big(\beta_t\inv(b^*)\omega(t^*,t)^*\big)\beta_s\inv(a^*)\omega(s^*,s)^*\Big)\omega(t^*,s^*)\delta_{t^*s^*}\\
    &=\beta_{t^*}\Big(\beta_{t^*}\inv\big(\omega(t^*,t)^*\beta_{t^*}(b^*)\big)\beta_s\inv(a^*)\omega(s^*,s)^*\Big)\omega(t^*,s^*)\delta_{t^*s^*}=\ldots
\end{align*}
Let $(u_i)$ be an approximate unit for $\D_t$ and define $y\defeq \beta_{t^*}(b^*)\in \D_{t^*}$ and
$z\defeq\beta_s\inv(a^*)\omega(s^*,s)^*\in \D_{s^*}$. Then
\begin{align*}
\ldots&=\beta_{t^*}\Big(\beta_{t^*}\inv\big(\omega(t^*,t)^*y\big)z\Big)\omega(t^*,s^*)\delta_{t^*s^*}\\
      &=\lim\limits_{i}\beta_{t^*}\Big(\beta_{t^*}\inv\big(\omega(t^*,t)^*y\big)u_iz\Big)\omega(t^*,s^*)\delta_{t^*s^*}\\
      &=\lim\limits_{i}\omega(t^*,t)^*y\beta_{t^*}(u_iz)\omega(t^*,s^*)\delta_{t^*s^*}\\
      &=\lim\limits_{i}\omega(t^*,t)^*\beta_{t^*}(b^*u_iz)\omega(t^*,s^*)\delta_{t^*s^*}\\
      &=\omega(t^*,t)^*\beta_{t^*}(b^*z)\omega(t^*,s^*)\delta_{t^*s^*}\\
      &=\omega(t^*,t)^*\beta_{t^*}\big(b^*\beta_s\inv(a^*)\omega(s^*,s)^*\big)\omega(t^*,s^*)\delta_{t^*s^*}=\ldots
\end{align*}
Using Lemma~\ref{lem:ConsequencesDefTwistedAction}\eqref{lem:ConsequencesDefTwistedAction:item:CocycleConditionWith*} again, the above equals
\begin{align*}
\ldots&=\omega(t^*,t)^*\beta_{t^*}\big(b^*\beta_s\inv(a^*)\big)\omega(t^*,s^*s)\omega(t^*s^*,s)^*\omega(t^*,s^*)^*\omega(t^*,s^*)\delta_{t^*s^*}\\
    &=\beta_t\inv\big(b^*\beta_s\inv(a^*)\big)\omega(t^*,t)^*\omega(t^*,s^*s)\omega(t^*s^*,s)^*\delta_{t^*s^*}.
\end{align*}
We conclude that $\big((a\delta_s)\cdot(b\delta_t)\big)^*=(b\delta_t)^*\cdot (a\delta_s)^*$ if and only if
\begin{equation*}
c\,\omega(t^*s^*s,t)^*\omega(t^*s^*,s)^*=c\,\omega(t^*,t)^*\omega(t^*,s^*s)\omega(t^*s^*,s)^*,
\end{equation*}
where $c\defeq \beta_t\inv(b^*\beta_s\inv(a^*))\in\D_{t^*s^*}$. Multiplying the above equation on the right by $\omega(t^*s^*,s)\in\U\mult(\D_{t^*s^*})$ and taking adjoints, we see that it is equivalent to
\begin{equation*}
\omega(t^*,s^*s)\omega(t^*s^*s,t)c^*=\omega(t^*,t)c^*.
\end{equation*}
And this last equation is a consequence of axiom~\eqref{def:twisted action:item:omega(s*,e)omega(s*e,s)x=omega(s*,s)x}
in Definition~\ref{def:twisted action}. Therefore the involution on $\B$ is
anti-multiplicative. It is easy to see that the involution is isometric, that is, $\|(a\delta_s)^*\|=\|a\delta_s\|$ for all $a\in \D_s$.
Thus, we have checked axiom~\eqref{def:Fell bundles over ISG:item:InvolutionIdempotentIsometricAntiMultiplicative}
in Definition~\ref{def:Fell bundles over ISG}.

To prove axiom~\eqref{def:Fell bundles over ISG:item:C*-Condition}
in Definition~\ref{def:Fell bundles over ISG}, take $s\in S$ and $a\in \D_s$. Then
\begin{align*}
(a\delta_s)^*\cdot(a\delta_s)&=\big(\omega(s^*,s)^*\beta_{s^*}(a^*)\delta_{s^*}\big)\cdot (a\delta_s)\\
                             &=\beta_{s^*}\Big(\beta_{s^*}\inv\big(\omega(s^*,s)^*\beta_{s^*}(a^*)\big)a\Big)\omega(s^*,s)\delta_{s^*s}\\
                             &=\beta_{s^*}\big(\beta_{s^*}\inv(\omega(s^*,s)^*)a^*a\big)\omega(s^*,s)\delta_{s^*s}\\
                             &=\omega(s^*,s)^*\beta_{s^*}(a^*a)\omega(s^*,s)\delta_{s^*s}.
\end{align*}
Now note that $\omega(s^*,s)^*\beta_{s^*}(a^*a)\omega(s^*,s)$ is a positive element of $\D_{s^*s}$ and its norm equals $\|a^*a\|=\|a\|^2$.

In order to prove axiom~\eqref{def:Fell bundles over ISG:item:CompatibilityInclusions}
in Definition~\ref{def:Fell bundles over ISG}, take $r,s,t\in S$ with $r\leq s\leq t$ and let $x\in \D_r$.
Then we have
\begin{align*}
(j_{t,s}\circ j_{s,r})(a\delta_r)&=j_{t,s}(j_{s,r}(a\delta_r))=j_{t,s}(a\omega(s,r^*r)^*\delta_s)\\
                                 &=a\omega(s,r^*r)^*\omega(r,s^*s)^*\delta_t\\
                                 &=(\omega(r,s^*s)\omega(s,r^*r)a^*)^*\delta_t=a\omega(t,r^*r)^*\delta_t,
\end{align*}
where in the last equation we have used
Lemma~\ref{lem:ConsequencesDefTwistedAction}\eqref{lem:ConsequencesDefTwistedAction:item:omega(t,r*r)x=omega(t,s*s)omega(s,r*r)x}.
Next, let us check that $j_{t,s}(a\delta_s)^*=j_{t^*,s^*}\big((a\delta_s)^*\big)$ for all $a\in \D_s$ and $s\leq t$ in $S$. We have
\begin{multline}\label{eq:PartCalculationInvolution}
j_{s,t}(a\delta_s)^*=\big(a\omega(t,s^*s)^*\delta_t\big)^*\\=\omega(t^*,t)^*\beta_{t^*}(\omega(t,s^*s)a^*)\delta_{t^*}
            =\omega(t^*,s)^*\beta_{t^*}(a^*)\delta_{t^*},
\end{multline}
where the last equation follows from the identity
\begin{equation}\label{eq:identityConseq.CocycleCondition}
\beta_{t^*}(a)\omega(t^*,s)=\beta_{t^*}(a\omega(t,s^*s)^*)\omega(t^*,t)
\end{equation}
which in turn is a consequence of axiom~\eqref{def:twisted action:item:CocycleCondition}
in Definition~\ref{def:twisted action}. In fact, by Definition~\ref{def:twisted action}\eqref{def:twisted action:item:CocycleCondition}, we have
\begin{equation*}
\beta_{t^*}(a\omega(t,s^*s))\omega(t^*,s)=\beta_{t^*}(a)\omega(t^*,t)\omega(t^*t,s^*s)=\beta_{t^*}(a)\omega(t^*,t)1_{s^*s}.
\end{equation*}
By Lemma~\ref{lem:ConsequencesDefTwistedAction}\eqref{lem:ConsequencesDefTwistedAction:item:beta_t restriction=...}, $\beta_{t^*}(a)\in \D_{s^*s}$ so that $\beta_{t^*}(a)\omega(t^*,t)1_{s^*s}=\beta_{t^*}(a)\omega(t^*,t)$. Replacing $a$ by $a\omega(t,s^*s)^*$, this yields~\eqref{eq:identityConseq.CocycleCondition}
and hence also~\eqref{eq:PartCalculationInvolution}. On the other hand, using again Equation~\eqref{eq:OtherFormulaForInvolution}
and Lemma~\ref{lem:ConsequencesDefTwistedAction}\eqref{lem:ConsequencesDefTwistedAction:item:beta_t restriction=...},
\begin{align*}\label{eq:SecondPartCalculationInvolution}
j_{t^*,s^*}\big((a\delta_s)^*\big)&=j_{t^*,s^*}\big(\omega(s^*,s)^*\beta_{s^*}(a^*)\delta_{s^*}\big)\\
                                  &=\omega(s^*,s)^*\beta_{s^*}(a^*)\omega(t^*,ss^*)^*\delta_{t^*}\\
                                  &=\omega(s^*,s)^*\omega(t^*,ss^*)^*\beta_{t^*}(a^*)\omega(t^*,ss^*)\omega(t^*,ss^*)^*\delta_{t^*}\\
                                  &=\omega(s^*,s)^*\omega(t^*,ss^*)^*\beta_{t^*}(a^*)\delta_{t^*}.
\end{align*}
Comparing this last equation with~\eqref{eq:PartCalculationInvolution}, we see that they are equal by
Lemma~\ref{lem:ConsequencesDefTwistedAction}\eqref{lem:ConsequencesDefTwistedAction:item:omega(t*,s)=...sLEQt}.
This proves axiom~\eqref{def:Fell bundles over ISG:item:CompatibilityInvolutionWithInclusions} in Definition~\ref{def:Fell bundles over ISG}.

To prove the missing axiom~\eqref{def:Fell bundles over ISG:item:CompatibilityMultiplicationWithInclusions}
in Definition~\ref{def:Fell bundles over ISG}, we shall use Lemma~\ref{lem:TechnicalLemmaDefFellBundle} proving
that $a\delta_s\cdot j_{v,u}(b\delta_u)=j_{sv,su}(a\delta_s\cdot b\delta_u)$ for all $s,v,u\in S$ with $u\leq v$, $a\in \D_s$ and $b\in \D_u$.
Defining $c=\beta_s\inv(a)b\in \D_{s^*}\cap\D_u$ and using Lemma~\ref{lem:ConsequencesDefTwistedAction}\eqref{lem:ConsequencesDefTwistedAction:item:CocycleConditionWith*}, we get
\begin{align*}
a\delta_s\cdot j_{v,u}(b\delta_u)&=(a\delta_s)\cdot (b\omega(v,u^*u)^*\delta_v)\\
                                 &=\beta_{s}\big(\beta_s\inv(a)b\omega(v,u^*u)^*\big)\omega(s,v)\delta_{sv}\\
                                 &=\beta_{s}\big(c\,\omega(v,u^*u)^*\big)\omega(s,v)\delta_{sv}\\
                                 &=\beta_{s}(c)\omega(s,vu^*u)\omega(sv,u^*u)^*\delta_{sv}\\
                                 &=\beta_{s}(c)\omega(s,u)\omega(sv,u^*u)^*\delta_{sv},
\end{align*}
where in the last equation we have used that $u\leq v$ so that $vu^*u=u$. On the other hand,
\begin{align*}
j_{sv,su}(a\delta_s\cdot b\delta_u)&=j_{sv,su}\big(\beta_s(\beta_s\inv(a)b)\omega(s,u)\delta_{su}\big)\\
                                   &=\beta_s(c)\omega(s,u)\omega(sv,u^*s^*su)^*\delta_{sv}
\end{align*}
Thus, to see that $a\delta_s\cdot j_{v,u}(b\delta_u)=j_{sv,su}(a\delta_s\cdot b\delta_u)$, it is enough to check
the equality $\omega(sv,u^*u)=\omega(sv,u^*s^*su)$. Since $u\leq v$, we have $v^*s^*svu^*u=v^*s^*su=v^*s^*suu^*u=v^*uu^*s^*su=u^*s^*su$.
Using Lemma~\ref{lem:ConsequencesDefTwistedAction}\eqref{lem:ConsequencesDefTwistedAction:item:omega(s,e)=w(s,s*se)...}, we conclude that
\begin{equation*}
\omega(sv,u^*u)=\omega(sv,v^*s^*svu^*u)=\omega(sv,u^*s^*su).
\end{equation*}
Therefore $\B$ is a Fell bundle over $S$. Note that $\B$ is saturated because the element $\beta_s(\beta_s\inv(a)b)\omega(s,t)$
appearing in the product
\begin{equation*}
(a\delta_s)\cdot (b\delta_t)=\beta_s(\beta_s\inv(a)b)\omega(s,t)\delta_{st}
\end{equation*}
is an arbitrary element of $\D_{st}$. In fact, $\beta_s\inv(a)b$ for $a\in \D_s$ and $b\in \D_t$ defines an arbitrary element of $\D_{s^*}\cap \D_t$,
and by Lemma~\ref{lem:ConsequencesDefTwistedAction}\eqref{lem:ConsequencesDefTwistedAction:item:beta_r(D_r* inter D_s)=D_rs}, $\beta_s(\D_{s^*}\cap\D_t)=\D_{st}$. The conclusion follows because
$\omega(s,t)$ is a unitary multiplier of $\D_{st}$. Finally, we leave the reader to check that $\B$ is regular with respect to the unitary multipliers
$v_s\in \mult(\B_s)$ defined by
\begin{equation}\label{eq:definitionOfu_s}
v_s\cdot (a\delta_{s^*s})\defeq \beta_s(a)\delta_s\quad\mbox{for all }a\in \D_{s^*s}.
\end{equation}
Here $\mult(\B_s)=\Ls(\B_{s^*s},\B_s)$ denotes the multiplier of $\B_s$ considered as an imprimitivity
Hilbert $\B_{ss^*},\B_{s^*s}$-bimodule.
Observe that, formally, we have $v_s=\delta_s=1_{ss^*}\delta_s$.
Recall that $1_{e}$ denotes the unit of the multiplier algebra of $\D_e$.
\end{proof}

Summarizing our results, we have shown that there is a correspondence between twisted actions and regular, saturated Fell bundles.
Given a twisted action $(\beta,\omega)$ of $S$ on a \cstar{}algebra $B$, we have constructed above a regular, saturated Fell
bundle $\B$. Moreover, the original twisted action $(\beta,\omega)$ can be recovered from the Fell bundle $\B$ using the
unitary multipliers $v_s=1_{ss^*}\delta_s\in\mult(\B_s)$ defined by Equation~\eqref{eq:definitionOfu_s}. More precisely, starting
with $\B$ and the unitary multipliers $v_s$, and proceeding as in Section~\ref{sec:regular Fell bundles}, we get a
twisted action $(\tilde\beta,\tilde\omega)$ of $S$ on the \cstar{}algebra $C^*(\E)$, where $\E=\{\B_e\}_{e\in E}$
is the restriction of $\B$ to $E=E(S)$. The \cstar{}algebra $\B_e=\D_e\delta_e$ is canonically isomorphic to the ideal $\D_e\sbe B$
and the sum of these ideals is dense in $B$. By Proposition~4.3 in \cite{Exel:noncomm.cartan}, this yields a canonical
isomorphism $C^*(\E)\cong B$ extending the isomorphisms $\B_e\cong \D_e$. Under these isomorphisms, $\tilde\beta_s\colon\B_{s^*s}\to \B_{ss^*}$
corresponds to $\beta_s\colon\D_{s^*s}\to \D_{ss^*}$ and the unitary
multipliers $\tilde\omega(s,t)\in \U\mult(\B_{stt^*s^*})$ correspond to $\omega(s,t)\in \U\mult(\D_{st})=\U\mult(\D_{stt^*s^*})$.
In fact, first note that (by Equation~\eqref{eq:DefInvolutionFellBundleFromTwistedAction})
\begin{equation*}
v_s^*=(1_{ss^*}\delta_s)^*=\omega(s^*,s)^*\delta_{s^*}=\omega(s^*,s)^*1_{s^*s}\delta_{s^*}=\omega(s^*,s)^*v_{s^*}\in \mult(\B_{s^*}).
\end{equation*}
Thus, using Lemma~\ref{lem:ConsequencesDefTwistedAction}\eqref{lem:ConsequencesDefTwistedAction:item:beta_s(omega(s*,s))=omega(s,s*)}, we get
\begin{align*}
\tilde\beta_s(a\delta_{s^*s})&=v_s\cdot a\delta_{s^*s}\cdot v_s^*=(1_{ss^*}\delta_s)\cdot(a\delta_{s^*s})\cdot(\omega(s^*,s)^*\delta_{s^*})\\
&=(\beta_s(a)\delta_s)\cdot(\omega(s^*,s)^*\delta_{s^*})=\beta_s(a\omega(s^*,s)^*)\omega(s,s^*)\delta_{ss^*}\\
&=\beta_s(a)\omega(s,s^*)^*\omega(s,s^*)\delta_{ss^*}=\beta_s(a)\delta_{ss^*}.
\end{align*}
and (using that $\beta_s(1_{s^*s}1_e)=1_{ses^*}$ for all $e\in E(S)$)
\begin{align*}
\tilde\omega(s,t)&=v_su_tu_{st}^*=(1_{ss^*}\delta_s)\cdot (1_{tt^*}\delta_t)\cdot(1_{stt^*s^*}\delta_{st})\\
                &=\big(\beta_s(\beta_s\inv(1_{ss^*})1_{tt^*})\omega(s,t)\delta_{st}\big)\cdot \big(\omega(t^*s^*,st)^*\delta_{t^*s^*}\big)\\
                &=\big(1_{stt^*s^*}\omega(s,t)\delta_{st}\big)\cdot \big(\omega(t^*s^*,st)^*\delta_{t^*s^*}\big)\\
                &=\beta_{st}\big(\beta_{st}\inv(\omega(s,t))\omega(t^*s^*,st)^*\big)\omega(st,t^*s^*)\delta_{stt^*s^*}\\
                &=\omega(s,t)\omega(st,t^*s^*)^*\omega(st,t^*s^*)\delta_{stt^*s^*}=\omega(s,t)\delta_{stt^*s^*}.\\
\end{align*}
This shows that the twisted action $(\tilde\beta,\tilde\omega)$ of $S$ on $C^*(\E)$ is isomorphic to the
twisted action $(\beta,\omega)$ of $S$ on $B$. Now let us assume that $(\beta,\omega)$ already comes from a
regular, saturated Fell bundle $\A=\{\A_s\}_{s\in S}$ as in Section~\ref{sec:regular Fell bundles} for some family
$u=\{u_s\}_{s\in S}$ of unitary multipliers $u_s\in\mult(\A_s)$ with $u_e=1_e$ for all $e\in E(S)$. In other words,
we are assuming that $\beta_s(a)=u_sau_s^*$ and $\omega(s,t)=u_su_tu_{st}^*$. Then it is not difficult to see that
the map $\B\to\A$ given on the fibers $\B_s\to \A_s$ by $a\delta_s\mapsto au_s$ is an isomorphism of Fell bundles $\B\cong \A$.
Moreover, under this isomorphism, the unitary multiplier $v_s\in \mult(\B_s)$ corresponds to $u_s\in \mult(\A_s)$.
All this together essentially proves the following result:

\begin{corollary}\label{cor:CorrespondenceRegularFellBundlesAndTwistedActions}
There is a bijective correspondence between isomorphism classes of twisted actions $(\beta,\omega)$ of $S$
and isomorphism classes of pairs $(\A,u)$ consisting of regular, saturated Fell bundles $\A$ over $S$ and families $u=\{u_s\}_{s\in S}$ of
unitary multipliers $u_s\in \mult(\A_s)$ satisfying $u_e=1_e$ for all $e\in E(S)$.
\end{corollary}

An isomorphism of pairs $(\B,v)\cong (\A,u)$ has the obvious meaning: it is an isomorphism $\B\cong \A$ of the Fell bundles
under which $v_s\in \mult(\B_s)$ corresponds to $u_s\in \mult(\A_s)$ for all $s\in S$.

As already observed in Section~\ref{sec:regular TROs}, imprimitivity bimodules over $\sigma$\nb-unital stable \cstar{}algebras are automatically regular. This immediately implies the following:

\begin{corollary}
Let $\A=\{\A_s\}_{s\in S}$ be a saturated Fell bundle for which all the fibers $\A_e$ with $e\in E(S)$ are
$\sigma$\nb-unital and stable. Then $\A$ is isomorphic to a Fell bundle associated to some twisted action of $S$.
\end{corollary}

\begin{remark}
Let $\E$ be the restriction of $\A$ to $E(T)$. If $B=C^*(\E)$ is stable, then so are the fibers $\A_e$ for all $e\in E(S)$ because
each $\A_e$ may be viewed as an ideal of $B$. The converse is not clear and is related to Question~2.6 proposed in \cite{Rordam:StableExtensions}
of whether the sum of stable ideals in a \cstar{}algebra is again stable.
\end{remark}

\section{Relation to Sieben's twisted actions}
\label{sec:RelationToSiebensTwistedActions}

In this section we are going to see that our notion of twisted action generalizes Sieben's definition appearing in \cite{SiebenTwistedActions}.

\begin{proposition}\label{prop:SiebensTwistedActions}
Let $S$ be an inverse semigroup, let $B$ be a \cstar{}algebra, and consider the data $\bigl(\{\D_s\}_{s\in S},\{\beta_s\}_{s\in S},\{\omega(s,t)\}_{s,t\in S}\bigr)$ satisfying~\eqref{def:twisted action:item:beta_r beta_s=Ad_omega(r,s)beta_rs} and
~\eqref{def:twisted action:item:CocycleCondition}
as in Definition~\ref{def:twisted action}, and in addition assume that
\begin{equation}\label{eq:SiebenAxiomTwistedAction}
\omega(s,e)=1_{se}\quad\mbox{and}\quad\omega(e,s)=1_{es}\quad\mbox{for all }s\in S\mbox{ and }e\in E(S).
\end{equation}
Then $\bigl(\{\D_s\}_{s\in S},\{\beta_s\}_{s\in S},\{\omega(s,t)\}_{s,t\in S}\bigr)$ is a twisted action of $S$ on $B$.
\end{proposition}
\begin{proof}
If \eqref{eq:SiebenAxiomTwistedAction} holds, then axiom~\eqref{def:twisted action:item:omega(r,r*r)=omega(e,f)=1_ef_and_omega(rr*,r)=1} in Definition~\ref{def:twisted action} is trivially satisfied.
And axiom~\eqref{def:twisted action:item:omega(s*,e)omega(s*e,s)x=omega(s*,s)x} will follow from the following relation:
\begin{equation}\label{eq:SiebenAxiomConsequences}
\omega(s^*,s)1_{s^*es}=\omega(s^*e,s)=\omega(s^*,es)
\end{equation}
for all $s\in S$ and $e\in E(S)$. This is a consequence of properties (i) and (j) of Lemma~2.3 in \cite{SiebenTwistedActions}.
Now, the first equation in \eqref{eq:SiebenAxiomConsequences} immediately
implies~\eqref{def:twisted action:item:omega(s*,e)omega(s*e,s)x=omega(s*,s)x} in Definition~\ref{def:twisted action} because
$\omega(s^*,e)=1_{s^*e}=1_{s^*es}$.
\end{proof}

\begin{definition}\label{def:SiebensTwistedAction}
A twisted action $(\beta,\omega)$ satisfying \emph{Sieben's condition}~\eqref{eq:SiebenAxiomTwistedAction}
is called a \emph{Sieben's twisted action} (see \cite[Definition~2.2]{SiebenTwistedActions}).
\end{definition}

The following result gives conditions on a twisted action coming from a regular Fell bundle to be a Sieben's twisted action.

\begin{proposition}\label{prop:CoherentHomogeneous}
Let $\A=\{\A_s\}_{s\in S}$ be a concrete, saturated, regular Fell bundle in $\bound(\hils)$. Given $s\in S$,
let $u_s\in \bound(\hils)$ be a partial isometry strictly associated to the {\tro} $\A_s$ such that $u_e=1_e$ for all $e\in E(S)$.
Given $s,t\in S$, we write $\omega(s,t)\defeq u_su_tu_{st}^*$.
Then the following assertions are equivalent\textup:
\begin{enumerate}[(i)]
\item $\omega(s,e)=1_{se}$ for all $s\in S$ and $e\in E(S)$;\label{prop:CoherentHomogeneous:item:omega(s,e)=1_se}
\item $u_su_e=u_{se}$ for all $s\in S$ and $e\in E(S)$;
\item $u_rx=u_sx$ for all $r,s\in S$ with $r\leq s$ and $x\in \A_{r^*r}$;
\item $u_r\leq u_s$ \textup(as partial isometries of $\bound(\hils)$\textup) for all $r,s\in S$ with $r\leq s$;
        \label{prop:CoherentHomogeneous:item:u_r leq u_s}
\item $yu_r=yu_s$ for all $r,s\in S$ with $r\leq s$ and $y\in \A_{rr^*}$;
\item $u_eu_s=u_{es}$ for all $s\in S$ and $e\in E(S)$;
\item $\omega(e,s)=1_{es}$ for all $s\in S$ and $e\in E(S)$;\label{prop:CoherentHomogeneous:item:omega(e,s)=1_es}
\end{enumerate}
\end{proposition}
\begin{proof}
Given $s,t\in S$, note that $u_su_t$ is associated to $\A_{st}$ because
\begin{equation*}
u_su_t\A_{t^*s^*st}=u_su_t\A_t^*\A_{s^*st}=u_s\A_{tt^*}\A_{s^*st}=u_s\A_{s^*st}=u_s\A_{s^*s}\A_t=A_s\A_t=\A_{st}
\end{equation*}
and similarly $\A_{stt^*s^*}u_su_t=\A_{st}$. Moreover $u_su_t$ is strictly associated to $\A_{st}$ because
$(u_su_t)^*(u_su_t)=u_t^*u_s^*u_su_t=u_t^*u_{s^*s}u_t=1_{t^*s^*st}$ (see Remark~\ref{rem:CoherenceOfPartialIsometries} and Proposition~\ref{prop:PartialIsometryStrictlyAssociated}).
Thus, we may view both $u_su_t$ and $u_{st}$ as unitary multipliers of the Hilbert bimodule $\A_{st}$.
The equation $\omega(s,t)=1_{st}$ is the same as $u_su_tu_{st}^*=1_{st}$ which is equivalent to
$u_su_t1_{t^*s^*st}=u_{st}$ because $u_{st}^*u_{st}=1_{t^*s^*st}$ and $1_{st}u_{st}=u_{st}$.
Moreover, since $u_su_t$ is a multiplier of $\A_{st}$, we have $u_su_t1_{t^*s^*st}=u_su_t$. Therefore the equation $\omega(s,t)=1_{st}$
is equivalent to $u_su_t=u_{st}$. From this we see that (i) is equivalent to (ii) and (vi) is equivalent to (vii).
Now, if $r\leq s$ in $S$, then $r=se$ for $e=r^*r$. Assuming (ii) and remembering that $u_e=1_e$, we get
\begin{equation*}
u_sx=u_su_ex=u_{se}x=u_rx\quad\mbox{for all }x\in \A_{r^*r}.
\end{equation*}
Thus (ii) implies (iii). Conversely, if $s\in S$ and $e\in E(S)$, and if we apply (iii) for $r=se$, then we get
$u_{se}x=u_rx=u_sx=u_su_ex$ for all $x\in A_{r^*r}=\A_{s^*se}=\A_{s^*s}\A_e$. This implies that $u_{se}=u_su_e$ because both $u_{se}$ and $u_su_e$
are multipliers of $\A_{se}$. Hence (ii) is equivalent to (iii). In the same way one can prove that (v) is equivalent to (vi).
Finally, condition (iv) is equivalent to $u_su_{r^*r}=u_su_r^*u_r=u_r$ whenever $r\leq s$ in $S$. Applying this for $r=se$ for some fixed $e\in E(S)$,
and using that $u_{r^*r}=u_{s^*se}=u_{s^*s}u_e$ and $u_su_{s^*s}=u_s$, we get condition (ii). Conversely, it is easy to see that (ii) implies (iv).
Similarly, observing that (iv) is equivalent to $u_{rr^*}u_s=u_{r}u_r^*u_s=u_r$ for all $r\leq s$, one proves that (iv) is equivalent to (vi).
\end{proof}

\begin{remark}\label{rem:CoherenceOfPartialIsometries}
Let notation be as in Proposition~\ref{prop:CoherentHomogeneous}.
Observe that the family of partial isometries $\{u_s\}_{s\in S}$ generates an inverse semigroup $\SSS$
under the product and involution of $\bound(\hils)$. Recall that the product $uv$ of two partial isometries
$u,v\in \bound(\hils)$ is again a partial isometry provided $u^*u$ and $vv^*$ commute.
This happens for the partial isometries $\{u_s\}_{s\in S}$ because $u_s^*u_s=1_{s^*s}$, $u_su_s^*=u_{ss^*}$
and $1_e1_f=1_{ef}=1_{fe}=1_{f}1_e$ for all $e,f\in E(S)$.
Thus the product $u_su_t$ is a partial isometry for all $s,t\in S$. Moreover, any finite product involving the partial isometries
$\{u_s\}_{s\in S}$ and their adjoints $\{u_s^*\}_{s\in S}$ is again a partial isometry. The idea is that for any such product $u$, say,
we have $u^*u=1_e$ and $uu^*=1_f$ for some $e,f\in E(S)$, and as we already said, $1_e$ and $1_f$ commute.
For example, to prove that $u_ru_su_t$ is a partial isometry,
one observes that $(u_ru_s)^*(u_ru_s)$ commutes with $u_tu_t^*=1_{tt^*}$. In fact, note that
\begin{equation*}
(u_ru_s)^*(u_ru_s)=u_s^*u_r^*u_ru_s=u_s^*1_{r^*r}u_s
\end{equation*}
and for every $e\in E(S)$, we have $u_s^*1_eu_s=1_{s^*es}$ because, for all $x\in \D_{s^*es}=\A_{s^*es}=\A_{s^*}\A_{es}$ we have
$u_sx\in u_s\A_{s^*}\A_{es}=\A_{ss^*}\A_{es}=\A_{es}=\A_e\A_s$, so that
\begin{equation*}
u_s^*1_eu_sx=u_s^*u_sx=1_{s^*s}x=x,
\end{equation*}
where in last equation we have used that $\D_{s^*es}\sbe\D_{s^*s}$ because $s^*es\leq s^*s$.
More generally, one can prove that
\begin{equation*}
(u_{s_1}\ldots u_{s_n})^*(u_{s_1}\ldots u_{s_n})=1_{s_n^*\ldots s_1^*s_1\ldots s_n}.
\end{equation*}
The inequality appearing in Proposition~\ref{prop:CoherentHomogeneous}\eqref{prop:CoherentHomogeneous:item:u_r leq u_s}
can be interpreted as the order relation of $\SSS$.
This justifies the argument at the end of the proof of Proposition~\ref{prop:CoherentHomogeneous} where we have used
that for $s\leq t$ in $S$, we have $u_s\leq u_t$ if and only if $u_s=u_tu_s^*u_s$ if and only if $u_s=u_su_s^*u_t$.

Note that the cocycles $\omega(s,t)$ belong to the inverse semigroup $\SSS$.
Moreover, if the partial isometries $\{u_s\}_{s\in S}$ satisfy the equivalent conditions~\eqref{prop:CoherentHomogeneous:item:omega(s,e)=1_se}-\eqref{prop:CoherentHomogeneous:item:omega(e,s)=1_es} of
Proposition~\ref{prop:CoherentHomogeneous}, then the cocycles $\{\omega(s,t)\}_{s,t\in S}$ satisfy
\begin{equation}\label{eq:CoherenceCocycles}
\omega(s,t)\leq\omega(s',t')\quad\mbox{for all }s,s',t,t'\in S\mbox{ with }s\leq s'\mbox{ and }t\leq t'.
\end{equation}
Here we view the unitaries $\omega(s,t)$ and $\omega(s',t')$ as elements of $\SSS$ in
order to give a meaning to the inequality~\eqref{eq:CoherenceCocycles}. Note that this follows directly from
condition~\eqref{prop:CoherentHomogeneous:item:u_r leq u_s} in Proposition~\ref{prop:CoherentHomogeneous} because the
product and the involution in $\SSS$ preserve the order relation. As in Proposition~\ref{prop:CoherentHomogeneous},
condition~\eqref{eq:CoherenceCocycles} can be also rewritten as
\begin{equation*}
\omega(s,t)x=\omega(s',t')x\quad\mbox{for all }s,s',t,t'\in S\mbox{ with }s\leq s'\mbox{ and }t\leq t'\mbox{ and }x\in \D_{st},
\end{equation*}
or equivalently as
\begin{equation*}
x\omega(s,t)=x\omega(s',t')\quad\mbox{for all }s,s',t,t'\in S\mbox{ with }s\leq s'\mbox{ and }t\leq t'\mbox{ and }x\in \D_{st}.
\end{equation*}
\end{remark}

\section{Representations and crossed products}
\label{sec:RepresentationsCrossedProducts}

In this section we prove that the correspondence between regular Fell bundles and twisted actions
obtained in the previous sections extends to the level of representations and yields an isomorphism of the associated
universal \cstar{}algebras.

We start recalling the definition of representations of Fell bundles (see \cite{Exel:noncomm.cartan}):

\begin{definition}\label{def:RepresentationFellBundle}
Let $\A=\{\A_s\}_{s\in S}$ be a Fell bundles over an inverse semigroup $S$.
A \emph{representation} of $\A$ on a Hilbert space $\hils$ is a family $\pi=\{\pi_s\}_{s\in S}$
of linear maps $\pi_s\colon \A_s\to \Ls(\hils)$ satisfying
\begin{enumerate}[(i)]
\item $\pi_s(a)\pi_t(b)=\pi_{st}(ab)$ and $\pi_s(a)^*=\pi_{s^*}(a^*)$ for all $s,t\in S$, $a\in \A_s$, $b\in \A_t$;\label{def:RepresentationFellBundle:item:AlgebraicOperations}
\item $\pi_t(j_{t,s}(a))=\pi_s(a)$ for all $s,t\in S$ with $s\leq t$ and $a\in \A_s$. Recall that $j_{t,s}\colon\A_s\to \A_t$
denotes the inclusion maps as in Definition~\ref{def:Fell bundles over ISG}.\label{def:RepresentationFellBundle:item:InclusionMaps}
\end{enumerate}
We usually view a representation of $\A$ as a map $\pi\colon\A\to \Ls(\hils)$ whose restriction to $\A_s$ gives the maps $\pi_s$ satisfying
the conditions above.
\end{definition}
 Next, we need a notion of covariant representation for twisted crossed products.
It turns out that, although our definition of twisted action is more general than Sieben's
one (see Section~\ref{sec:RelationToSiebensTwistedActions}), the notion of representation still remains the same (see \cite[Definition~3.2]{SiebenTwistedActions}):

\begin{definition}\label{def:CovariantRepresentation}
Let $(\beta,\omega)$ be a twisted action of an inverse semigroup $S$ on a \cstar{}algebra $B$.
A \emph{covariant representation} of $(\beta,\omega)$ on a Hilbert space $\hils$
is a pair $(\rho,v)$ consisting of a $*$-homomorphism $\rho\colon B\to \Ls(\hils)$ and a family $v=\{v_s\}_{s\in S}$ of
partial isometries $v_s\in \Ls(\hils)$ satisfying
\begin{enumerate}[(i)]
\item $\rho(\beta_s(b))=v_s\rho(b)v_s^*$ for all $s\in S$ and $b\in B$;\label{def:CovariantRepresentation:item:rho(beta_s(b))=v_s rho(b) v_s*}
\item $\rho(\omega(s,t))=v_sv_tv_{st}^*$ for all $s,t\in S$; and \label{def:CovariantRepresentation:item:rho(omega(s,t))=v_sv_tv_st*}
\item $v_s^*v_s=\rho(1_{s^*s})$ and $v_sv_s^*=\rho(1_{ss^*})$. \label{def:CovariantRepresentation:item:v_s*v_s=pi(1_s*s)}
\end{enumerate}
Recall that $1_e$ denotes the unit of the multiplier algebra of $\D_e=\dom(\beta_e)$ for all $e\in E(S)$. The third condition above is equivalent
to the requirements $\rho(D_{s^*s})H=v_s^*v_sH$ and $\rho(\D_{ss^*})H=v_sv_s^*H$. In both axioms (ii) and (iii) above we have
tacitly extended $\rho$ to the enveloping von Neumann algebra $B''$ of $B$ in order to give a meaning to $\rho(\omega(s,t))$ and $\rho(1_e)$.
\end{definition}

To each notion of representation is attached a universal \cstar{}algebra which encodes all the representations.
In the case of a Fell bundle $\A$, it is the so called \emph{(full) cross-sectional \cstar{}algebra} $C^*(\A)$
defined in \cite{Exel:noncomm.cartan}. For a twisted action
$(B,S,\beta,\omega)$, the \emph{(full) crossed product} $B\rtimes_{\beta,\omega}S$ defined in
\cite{SiebenTwistedActions} can still be used although our definition of twisted action is more general.
Our aim in this section is to relate these notions when
$\A$ is the Fell bundle associated to $(B,S,\beta,\omega)$ as in the previous sections.

\begin{theorem}\label{theo:CorresponceRepresentations}
Let $(\beta,\omega)$ be a twisted action of and inverse semigroup $S$ on a \cstar{}algebra $B$,
and let $\A=\{\A_s\}_{s\in S}$ be the associated Fell bundle as in Section~\ref{sec:TwistedActions}.
Then there is a bijective correspondence between covariant representations of $(\beta,\omega)$ and representations of $\A$.
Moreover, there exists an isomorphism $B\rtimes_{\beta,\omega}S\cong C^*(\A)$.
\end{theorem}
\begin{proof}
Recall that the fiber $\A_s=\D_s\delta_s$ is a copy of $\D_s=\D_{ss^*}=\ran(\beta_s)$. During the proof we
write $u_s$ for the unitary multiplier $1_{ss^*}\delta_s\in \mult(\A_s)$ (here $\A_s$ is viewed as an imprimitivity Hilbert $A_{ss^*},A_{s^*s}$\nb-bimodule). With this notation, we have $\A_s=\D_{ss^*}u_s$ for all $s\in S$.
In particular, $\A_e=\D_e\delta_e\cong\D_e$ as \cstar{}algebras and this induces an isomorphism $C^*(\E)\cong B$,
where $\E$ is the restriction of $\A$ to $E(S)$. Moreover, as we have seen
in the previous section, the unitaries $u_s$ can be used to recover the twisted action $(\beta,\omega)$ through the formulas:
\begin{equation}\label{eq:CharacterizationOfBetaViaU_s}
\beta_s(b)\delta_{ss^*}=u_s(b\delta_{s^*s})u_s^*,
\end{equation}
\begin{equation}\label{eq:CharacterizationOfOmegaViaU_s}
\omega(s,t)\delta_{stt^*s^*}=u_su_tu_{st}^*.
\end{equation}
Let $\pi\colon \A\to \Ls(\hils)$ be a representation of $\A$. This representation integrates to a \Star{}homomorphism
$\tilde\pi\colon C^*(\A)\to \Ls(\hils)$. We may view the fiber $\A_e$ for $e\in E(S)$ and $C^*(\E)$ as subalgebras of $C^*(\A)$.
By restriction of $\tilde\pi$ to $C^*(\E)$, we get a \Star{}homomorphism $\rho$ from $B\cong C^*(\E)$ into $\Ls(\hils)$ which is characterized
by
\begin{equation*}
\rho(b)=\tilde\pi(b\delta_e)\quad\mbox{whenever }e\in E(S)\mbox{ and }b\in \D_e.
\end{equation*}
By Proposition~\ref{prop:strictlyAssociated=>weakClosure},
the unitary multipliers $u_s$ may be viewed as elements of the enveloping von Neumann algebra $A''$ of $A\defeq C^*(\A)$.
Considering the unique weakly continuous extension of $\tilde\pi$ to $A''$ (and still using the same notation for it), we define
the partial isometries $v_s\defeq \tilde\pi(u_s)\in \Ls(\hils)$. The pair $(\rho,v)$ is a covariant representation of $(\beta,\omega)$.
In fact, using Equation~\eqref{eq:CharacterizationOfBetaViaU_s}, we get
\begin{align*}
\rho(\beta_s(b))&=\tilde\pi(\beta_s(b)\delta_{ss^*})=\tilde\pi(u_s(b\delta_{s^*s})u_s^*)\\
                &=\tilde\pi(u_s)\tilde\pi(b\delta_{s^*s})\tilde\pi(u_s)^*=v_s\rho(b)v_s^*
\end{align*}
for all $s\in S$ and $b\in \D_{s^*s}$. Since the ideals $\D_e\sbe B$ span a dense subspace of $B$, this proves~\eqref{def:CovariantRepresentation:item:rho(beta_s(b))=v_s rho(b) v_s*} in Definition~\ref{def:CovariantRepresentation}.
Definition~\ref{def:CovariantRepresentation}\eqref{def:CovariantRepresentation:item:rho(omega(s,t))=v_sv_tv_st*} follows in the same way
using Equation~\eqref{eq:CharacterizationOfOmegaViaU_s}:
\begin{equation*}
\rho(\omega(s,t))=\tilde\pi(\omega(s,t)\delta_{stt^*s^*})=\tilde\pi(u_su_tu_{st}^*)
=\tilde\pi(u_s)\tilde\pi(u_t)\tilde\pi(u_{st})^*=v_sv_tv_{st}^*.
\end{equation*}
Since $u_s^*u_s=1_{s^*s}\delta_{s^*s}$ and $u_su_s^*=1_{ss^*}\delta_{ss^*}$, axiom~\eqref{def:CovariantRepresentation:item:v_s*v_s=pi(1_s*s)}
in Definition~\ref{def:CovariantRepresentation} also follows easily. Hence we have a map $\pi\mapsto (\rho,v)$ from
the set of representations of $\A$ on $\hils$ to the set of covariant representation of $(\beta,\omega)$ on $\hils$.

Conversely, let us now start with a covariant representation $(\rho,v)$ of $(\beta,\omega)$ on $\hils$. Then we can
define $\pi\colon\A\to \Ls(\hils)$ by $\pi(a\delta_s)\defeq \rho(a)v_s$ for all $s\in S$ and $a\in \D_{ss^*}$.
Before we prove that $\pi$ is a representation of $\A$, observe that $v_e=\pi(1_e)$ for all $e\in E(S)$ (see proof of Proposition~3.5 in \cite{SiebenTwistedActions}) so that $v_s^*v_s=\rho(1_{s^*s})=v_{s^*s}$ and $v_sv_s^*=\rho(1_{ss^*})=v_{ss^*}$ for all $s\in S$.
By Lemma~\ref{lem:ConsequencesDefTwistedAction}\eqref{lem:ConsequencesDefTwistedAction:item:beta_s*=Ad_omega(s*,s)beta_s inv}, we have
\begin{align*}
v_s\rho(\beta_s\inv(a))&=v_s\rho(\omega(s^*,s))^*\rho(\beta_{s^*}(a))\rho(\omega(s^*,s))\\
                &=v_sv_{s^*s}v_s^*v_{s^*}^*v_{s^*}\rho(a)v_{s^*}^*v_{s^*}v_sv_{s^*s}=\rho(a)v_s=\pi(a\delta_s).
\end{align*}
Thus, for all $s,t\in S$, $a\in \D_{ss^*}$ and $b\in \D_{tt^*}$,
\begin{align*}
\pi\big((a\delta_s)\cdot(b\delta_t)\big)&=\pi\big(\beta_s(\beta_s\inv(a)b)\omega(s,t)\delta_{st}\big)\\
                                        &=\rho\big(\beta_s(\beta_s\inv(a)b)\big)\rho(\omega(s,t))v_{st}\\
                                        &=v_s\rho(\beta_s\inv(a))\rho(b)v_s^*v_sv_tv_{st}^*v_{st}\\
                                        &=v_s\rho(\beta_s\inv(a))\rho(b)v_t=\pi(a\delta_s)\pi(b\delta_t).\\
\end{align*}
Similarly,
\begin{align*}
\pi\big((a\delta_s)^*\big)&=\pi(\beta_s\inv(a^*)\omega(s^*,s)^*\delta_{s^*})=\rho(\beta_s\inv(a^*))\rho(\omega(s^*,s)^*)v_{s^*}\\
                            &=\rho(\beta_s\inv(a^*))(v_{s^*}v_sv_{s^*s})^*v_{s^*}=\rho(\beta_s\inv(a^*))v_s^*\\
                            &=v_s^*\rho(a^*)v_sv_s^*=(\rho(a)v_s)^*=\pi(a\delta_s)^*.
\end{align*}
This proves axiom~\eqref{def:RepresentationFellBundle:item:AlgebraicOperations} in Definition~\ref{def:RepresentationFellBundle}.
To prove Definition~\ref{def:RepresentationFellBundle}\eqref{def:RepresentationFellBundle:item:InclusionMaps},
take $s,t\in S$ with $s\leq t$ and $a\in \D_{s^*s}$. Note that
\begin{equation*}
v_{s^*s}v_t^*v_t=\rho(1_{s^*s})\rho(1_{t^*t})=\rho(1_{s^*st^*t})=\rho(1_{s^*s})=v_{s^*s}.
\end{equation*}
Hence
\begin{multline*}
\pi(j_{t,s}(a))=\pi(a\omega(t,s^*s)^*\delta_t)=\rho(a\omega(t,s^*s)^*)v_t=\rho(a)(v_tv_{s^*s}v_{ts^*s}^*)^*v_t\\
=\rho(a)v_sv_{s^*s}v_t^*v_t=\rho(a)v_sv_{s^*s}=\rho(a)v_s=\pi(a).
\end{multline*}
Therefore $\pi$ is a representation of $\A$ on $\hils$. It is not difficult to see that the assignments $\pi\mapsto (\rho,v)$
and $(\rho,v)\mapsto \pi$ between representations $\pi$ of $\A$ and covariant representations $(\rho,v)$ of $(\beta,\omega)$
are inverse to each other and hence give the desired bijective
correspondence as in the assertion. The isomorphism $B\rtimes_{\beta,\omega}S\cong C^*(\A)$ now follows from the universal properties of
$B\rtimes_{\beta,\omega}S$ and $C^*(\A)$ with respect to (covariant) representations
(see \cite{Exel:noncomm.cartan,SiebenTwistedActions} for details).
\end{proof}

In \cite{Exel:noncomm.cartan} the second named author also defines the \emph{reduced cross-sectional algebra} $C^*_\red(\A)$ of a Fell bundle $\A$.
It is the image of $C^*(\A)$ by (the integrated form of) a certain special representation of $\A$, the so called
\emph{regular representation} of $\A$. Using Theorem~\ref{theo:CorresponceRepresentations} we can now also define reduced crossed products
for twisted actions (this was not defined in \cite{SiebenTwistedActions}):

\begin{definition}
Let $(B,S,\beta,\omega)$ be a twisted action and let $\A$ be the associated Fell bundle.
Let $\Lambda$ denote the regular representation of $\A$ as defined in \cite[Section~8]{Exel:noncomm.cartan}.
The \emph{regular covariant representation} of $(B,S,\beta,\omega)$ is the representation $(\lambda,\upsilon)$ corresponding to $\Lambda$
as in Theorem~\ref{theo:CorresponceRepresentations}.
The \emph{reduced crossed product} is
\begin{equation*}
B\rtimes_{\beta,\omega}^\red S\defeq \lambda\rtimes\upsilon\big(B\rtimes_{\beta,\omega}S\big),
\end{equation*}
where $\lambda\rtimes\upsilon$ is the \emph{integrated form} of $(\lambda,\upsilon)$ defined in \cite[Definition~3.6]{SiebenTwistedActions}.
\end{definition}

Notice that, by definition, we have a canonical isomorphism
\begin{equation}\label{eq:IsomorphismReducedCrossedProductAndReducedGrupoideAlgebra}
B\rtimes_{\beta,\omega}^\red S\cong C^*_\red(\A).
\end{equation}

\section{Relation to twisted groupoids}
\label{sec:motivating example}

Let $\G$ be a locally compact étale groupoid with unit space $X=\Gz$ (see
\cite{Exel:inverse.semigroups.comb.C-algebras, KhoshkamSkandalis:CrossedProducts, Paterson:Groupoids, RenaultThesis} for
further details on étale groupoids). Recall that a \emph{twist} over $\G$ is a topological
groupoid $\Sigma$ that fits into a groupoid extension of the form
\begin{equation}\label{eq:ExtensionTwistedGroupoid}
\Torus\times X\into \Sigma \onto \G.
\end{equation}
The pair $(\G,\Sigma)$ is then called a \emph{twisted étale groupoid}.
We refer the reader to \cite{Deaconi_Kumjian_Ramazan:Fell.Bundles,Muhly.Williams.Continuous.Trace.Groupoid,RenaultCartan} for more details.
Twists over $\G$ can be alternatively described as principal circle bundles (these are the $\Torus$\nb-groupoids defined
in \cite[Section~2]{Muhly.Williams.Continuous.Trace.Groupoid}) and the classical passage to (complex) line bundles enables us to view
twists over $\G$ as \emph{Fell line bundles} over $\G$, that is,
(locally trivial) one-dimensional Fell bundles over the groupoid $\G$ in the sense of Kumjian \cite{Kumjian:fell.bundles.over.groupoids}.

As we have seen in \cite{BussExel:Fell.Bundle.and.Twisted.Groupoids},
there is a correspondence between twisted étale groupoids and \emph{semi-abelian} saturated Fell bundles over inverse semigroups, that is, saturated Fell bundles $\A=\{\A_s\}_{s\in S}$ for which $\A_e$ is an abelian \cstar{}algebra for all $e\in E(S)$.
This enables us to apply our previous results and describe twisted étale groupoids from the point of view of twisted actions of
inverse semigroups on abelian \cstar{}algebras (compare with \cite[Theorem~3.3.1]{Paterson:Groupoids}, \cite[Theorem~8.1]{Quigg.Sieben.C.star.actions.r.discrete.groupoids.and.inverse.semigroups} and \cite[Theorem~9.9]{Exel:inverse.semigroups.comb.C-algebras}):

\begin{theorem}\label{theo:CorrespondenceTwistedGroupoidsAndTwistedActions}
Given a twisted étale groupoid $(\G,\Sigma)$, there is an inverse semigroup $S$ consisting of bisections of $\G$
and a twisted action $(\beta,\omega)$ of $S$ on $\contz\big(\Gz\big)$ such that the (reduced) groupoid \cstar{}algebra $C^*_\red(\G,\Sigma)$
is isomorphic to the (reduced) crossed product $\contz\big(\Gz\big)\rtimes_{\beta,\omega}^\red S$. Conversely, if $(\beta,\omega)$ is a twisted action of
an inverse semigroup $S$ on a commutative \cstar{}algebra $\contz(X)$ for some locally compact Hausdorff space $X$, then
there is a twisted étale groupoid $(\G,\Sigma)$ with $\Gz=X$ such that $C^*_\red(\G,\Sigma)\cong \contz(X)\rtimes^\red_{\beta,\omega} S$.

If, in addition, the groupoid $\G$ is Hausdorff or second countable, then we also have an isomorphism of full \cstar{}algebras
$C^*(\G,\Sigma)\cong \contz(X)\rtimes_{\beta,\omega} S$.
\end{theorem}
\begin{proof}
Let $(\G,\Sigma)$ be a twisted étale groupoid and let $L$ be the associated Fell line bundle over $\G$. Recall that an open subset $s\sbe \G$ is a \emph{bisection} (also called \emph{slice} in \cite{Exel:inverse.semigroups.comb.C-algebras}) if the restrictions of the source and range maps $\domain, \range\colon \G\to \Gz$ to $s$ are homeomorphisms onto their images. The set $S(\G)$ of all bisections of $\G$ forms an inverse semigroup with respect to the product $st\defeq\{\alpha\beta\colon \alpha\in s,\beta\in t,\domain(\alpha)=\range(\beta)\}$ and the involution $s^*\defeq\{\alpha\inv\colon \alpha\in s\}$ (see \cite{Exel:inverse.semigroups.comb.C-algebras}). Given $s\in S(\G)$, let $L_s$ be the restriction  of $L$ to $s$. Let $S$ be the subset of $S(\G)$ consisting
of all bisections $s\in S(\G)$ for which the line bundle $L_s$ is trivial.
If $L_s$ and $L_t$ are trivial, then so are $L_{st}$ and $L_{s^*}$  because the (convolution) product
$\xi\cdot\eta$ (defined by $(\xi\cdot\eta)(\g)\defeq\xi(\alpha)\eta(\beta)$ whenever $\g=\alpha\beta\in st$)
is a unitary section of $L_{st}$ provided $\xi$ and $\eta$ are unitary sections of $L_s$ and $L_t$, respectively (note
that in a Fell line bundle $L$, we have $\|ab\|=\|a\|\|b\|$ for all $a,b\in L$); and the involution $\xi^*(\g)\defeq \xi(\g\inv)^*$
also provides a bijective correspondence between unitary sections of $L_s$ and $L_{s^*}$.
Thus $S$ is an inverse sub-semigroup of $S(\G)$.
Since $L$ is locally trivial, $S$ is a covering for $\G$ and we obviously have $s\cap t\in S$ for $s,t\in S$. In particular,
$S$ is \emph{wide} in the sense of \cite[Definition~2.14]{BussExel:Fell.Bundle.and.Twisted.Groupoids}, that is, it satisfies the condition:
\begin{equation}\label{eq:BasisCondition}
\mbox{for all }s,t\in S \mbox{ and } \g\in s\cap t,
\mbox{ there is } r\in S\mbox{ such that }\g\in r\sbe s\cap t.
\end{equation}
This condition also appears in \cite[Proposition~5.4.ii]{Exel:inverse.semigroups.comb.C-algebras}.
For each $s\in S$, we define $\A_s$ to be the space $\contz(L_s)$ of continuous sections of $L_s$ vanishing at infinity.
Then, with respect to the (convolution) product and the involution of sections defined above, the family $\A=\{\A_s\}_{s\in S}$ is a saturated Fell bundle over $S$ (see \cite[Example~2.11]{BussExel:Fell.Bundle.and.Twisted.Groupoids} for details). Moreover, by Proposition~\ref{prop:Contz(L)RegularIFFLTopologicallyTrivial},
the Hilbert $\contz(ss^*),\contz(s^*s)$\nb-bimodule $\A_s=\contz(L_s)$
is regular because $L$ is trivial over every $s\in S$. Thus $\A$ is a saturated, regular Fell bundle.
By our construction in Section~\ref{sec:regular Fell bundles}, $\A$ gives rise to a twisted action $(\beta,\omega)$ of $S$ on $B=C^*(\E)$, where $\E$ is the restriction of $\A$ to $E(S)$. Observe that $E(S)$ is a covering for $\Gz$ and since $L$
is trivial over $\Gz$ (because it is a one-dimensional continuous \cstar{}bundle)
we have $\contz(L_e)\cong\contz(e)$ for every $e\in E(S)$ (note that $e\sbe \Gz$).
This implies that $B\cong\contz\big(\Gz\big)$ by an application of Proposition~4.3 in \cite{Exel:noncomm.cartan}.
By Theorem~\ref{theo:CorresponceRepresentations} and Equation~\eqref{eq:IsomorphismReducedCrossedProductAndReducedGrupoideAlgebra}, we have canonical isomorphisms $\contz\big(\Gz\big)\rtimes_{\beta,\omega}S\cong C^*(\A)$ and $\contz\big(\Gz\big)\rtimes_{\beta,\omega}^\red S\cong C^*_\red(\A)$. Moreover,
by \cite[Theorem~4.11]{BussExel:Fell.Bundle.and.Twisted.Groupoids}, $C^*_\red(\A)\cong C^*_\red(\G,\Sigma)$ and by \cite[Proposition~2.18]{BussExel:Fell.Bundle.and.Twisted.Groupoids}, $\contz\big(\Gz\big)\rtimes_{\beta,\omega}S\cong C^*(\A)$ provided $\G$ is Hausdorff
or second countable.

Conversely, starting with a twisted action $(\beta,\omega)$ of an inverse semigroup $S$ on a commutative \cstar{}algebra $\contz(X)$,
let $\A=\{\A_s\}_{s\in S}$ be the associated Fell bundle as in Section~\ref{sec:TwistedActions}.
Then $\A$ is saturated and semi-abelian, so the construction in
\cite[Section~3]{BussExel:Fell.Bundle.and.Twisted.Groupoids} provides a twisted étale groupoid $(\G,\Sigma)$ with $\Gz=X$ together with isomorphisms
$C^*_\red(\G,\Sigma)\cong C^*_\red(\A)\cong \contz\big(\Gz\big)\rtimes_{\beta,\omega}^\red S$ again by \cite[Theorem~4.11]{BussExel:Fell.Bundle.and.Twisted.Groupoids} and Equation~\ref{eq:IsomorphismReducedCrossedProductAndReducedGrupoideAlgebra};
and if $\G$ is Hausdorff or second countable then
$C^*(\G,\Sigma)\cong C^*(\A)\cong \contz\big(\Gz\big)\rtimes_{\beta,\omega}S$ by \cite[Proposition~3.40]{BussExel:Fell.Bundle.and.Twisted.Groupoids}, \cite[Proposition~2.18]{BussExel:Fell.Bundle.and.Twisted.Groupoids} and Theorem~\ref{theo:CorresponceRepresentations}.
\end{proof}

We shall next study the question of whether or not the twisted action constructed from a twisted groupoid as above
satisfies Sieben’s condition (see Definition~\ref{def:SiebensTwistedAction}).

\begin{proposition}\label{prop:SiebenTwistedActions=TopologicalTrivial}
Let $(\G,\Sigma)$ be a twisted groupoid  and let $(\contz\big(\Gz\big),S,\beta,\omega)$ be a twisted action associated to $(\G,\Sigma)$
as in Theorem~\ref{theo:CorrespondenceTwistedGroupoidsAndTwistedActions}.
Then the cocycles $\omega(s,t)$ can be chosen to satisfy Sieben's condition~\eqref{eq:SiebensCondition}
if and only if the twist $\Sigma$ is topologically trivial \textup(that is,
$\Sigma\cong\Torus\times\G$ as circle bundles\textup) or, equivalently, if the Fell line bundle $L$ associated to $(\G,\Sigma)$
is topologically trivial \textup(that is, $L\cong \C\times\G$ as complex line bundles\textup).
\end{proposition}
\begin{proof}
We shall use the same notation as in the proof of Theorem~\ref{theo:CorrespondenceTwistedGroupoidsAndTwistedActions}.
By definition, $(\beta,\omega)$ is the twisted action associated to the regular Fell bundles $\A$ as in Section~\ref{sec:regular Fell bundles}.
Thus, the cocycles $\omega(s,t)$ are given by $u_su_tu_{st}^*$
for a certain choice of unitary multipliers $u_s$ of $\A_s=\contz(L_s)$.
Since $\mult(\contz(L_s))\cong \contb(L_s)$, $u_s$ may be viewed as a unitary section of $L_s$.
Suppose the cocycles $\omega(s,t)$ satisfy~\eqref{eq:SiebensCondition}. We are going to prove that $L$ is topologically trivial, that is, that
there is a global continuous unitary section of $L$. By Proposition~\ref{prop:CoherentHomogeneous},
the family $\{u_s\}_{s\in S}$ of unitary sections satisfy $u_s\leq u_t$ whenever $s,t\in S$ with $s\leq t$.
This means that $u_s$ is the restriction of $u_t$ whenever $s\sbe t$ (this is the order relation of $S\sbe S(\G)$).
Now if $s,t$ are arbitrary elements of $S$,
the sections $u_s$ and $u_t$ have to coincide on the intersection $s\cap t$. If fact, since $S$ satisfies~\eqref{eq:BasisCondition},
given $\g\in s\cap t$, there is $r\in S$ such that $\g\in r\sbe s\cap t$. Hence $r\sbe s$ and $r\sbe t$, so that $u_s(\g)=u_r(\g)=u_t(\g)$.
Thus the map $u\colon\G\to L$ defined by
\begin{equation*}
u(\g)\defeq u_s(\g) \quad\mbox{whenever $s$ is an element of $S$ containing }\g\in \G
\end{equation*}
is a well-defined continuous unitary section of $L$ and therefore $L$ is topologically trivial.
Conversely, if $L$ is topologically trivial and $u\colon \G\to L$ is a continuous unitary section, then we may take $u_s$ to
be the restriction of $u$ to the bisection $s\sbe\G$. In this way $u_s$ is the restriction of $u_t$ whenever $s\sbe t$. Again by Proposition~\ref{prop:CoherentHomogeneous}, the cocycles $\omega(s,t)=u_su_tu_{st}^*$ satisfy Sieben's condition~\eqref{eq:SiebensCondition}.
\end{proof}

It is well-known (see \cite[Section 2]{Muhly.Williams.Continuous.Trace.Groupoid})
that the topologically trivial twists are exactly those associated to
a $2$-cocycle $\tau\colon \Gt\to \Torus$ in the sense of Renault \cite{RenaultThesis}.
Moreover, by Example~2.1 in \cite{Muhly.Williams.Continuous.Trace.Groupoid} there are twisted groupoids that are not topologically trivial.
This example shows that Sieben's condition~\eqref{eq:SiebenAxiomTwistedAction} cannot be expected to hold in general and therefore our notion of twisted action (Definition~\ref{def:twisted action}) properly generalizes Sieben’s \cite[Definition~2.2]{SiebenTwistedActions}.

In what follows we briefly describe the twists associated to $2$-cocycles and relate them to our cocycles $\omega(s,t)$.
Let $\tau\colon \Gt\to \Torus$ be a $2$-cocycle. Recall that $\tau$ satisfies the cocycle condition
\begin{equation*}
\tau(\alpha,\beta)\tau(\alpha\beta,\gamma)=\tau(\beta,\gamma)\tau(\alpha,\beta\gamma)\quad\mbox{for all }
        (\alpha,\beta),(\beta,\gamma)\in \Gt.
\end{equation*}
It is interesting to observe the similarity of this condition with our cocycle condition
appearing in Definition~\ref{def:twisted action}\eqref{def:twisted action:item:CocycleCondition}.
In addition, we also assume that the $2$-cocycle $\tau$ is \emph{normalized} in the sense that
(compare with Definition~\ref{def:twisted action}\eqref{def:twisted action:item:omega(r,r*r)=omega(e,f)=1_ef_and_omega(rr*,r)=1})
\begin{equation*}
\tau(\alpha,\s(\alpha))=\tau(\r(\alpha),\alpha)=1\quad\mbox{for all }\alpha\in \G.
\end{equation*}
Given a $2$-cocycle $\tau$ on $\G$ as above, the associated twisted groupoid $(\G,\Sigma)$ is defined as follows. Topologically, $\Sigma$ is the trivial circle bundle $\Torus\times\G$. And the operations are defined by
\begin{equation}\label{eq:multiplicationTwistedGroupoid}
(\lambda,\alpha)\cdot (\mu,\beta)=(\lambda\mu\tau(\alpha,\beta),\alpha\beta)\quad\mbox{for all }\lambda,\mu\in \Torus\mbox{ and }(\alpha,\beta)\in \Gt
\end{equation}
\begin{equation}\label{eq:inversionTwistedGroupoid}
(\lambda,\alpha)\inv =(\overline{\lambda \tau(\alpha\inv,\alpha)},\alpha\inv)\quad\mbox{for all }\lambda\in\Torus\mbox{ and }\alpha\in \G.
\end{equation}
In this way, $\Sigma$ is a topological groupoid and the trivial maps $\Torus\times \Gz\into \Sigma$ and
$\Sigma\onto \G$ give us a groupoid extension as
in~\eqref{eq:ExtensionTwistedGroupoid}. The twisted groupoid $(\G,\Sigma)$ corresponds
to the (topologically trivial) Fell line bundle $L=\C\times\G$ with algebraic
operations of multiplication and involution given by the same formulas as in equations~\eqref{eq:multiplicationTwistedGroupoid} and~\eqref{eq:inversionTwistedGroupoid} (only replacing $\Torus$ by $\C$ and the inversion $\inv$ on the
left hand side of~\eqref{eq:inversionTwistedGroupoid} by the involution $^*$ sign).

\begin{proposition}\label{prop:RelationRenaultAndSiebensCocycles}
Let $(\G,\Sigma)$ be the twisted groupoid associated to a $2$-cocycle $\tau$ on $\G$ as above,
and let $(\contz\big(\Gz\big),S,\beta,\omega)$ be the twisted action associated to $(\G,\Sigma)$
as in Theorem~\ref{theo:CorrespondenceTwistedGroupoidsAndTwistedActions}
through the unitary sections $u_s$ of $L_s=\C\times s$ given by $u_s(\gamma)=(1,\gamma)$ for all $\gamma\in s$.
Then
\begin{equation}
\beta_s(a)(\r(\g))=a(\s(\g))\quad\mbox{and}\quad \omega(s,t)(\r(\alpha\beta))=\tau(\alpha,\beta)
\end{equation}
for all $s,t\in S$, $a\in \contz(s^*s)$, $\alpha,\g\in s$ and $\beta\in t$ with $\s(\alpha)=\r(\beta)$.
\end{proposition}
\begin{proof}
Note that $\dom(\beta_s)=\D_{s^*s}=\contz(s^*s)$ and $\ran(\beta_s)=\D_{ss^*}=\contz(ss^*)$.
Using the definitions \eqref{eq:multiplicationTwistedGroupoid} and \eqref{eq:inversionTwistedGroupoid}
and the (easily verified) relation $\tau(\g\inv,\g)=\tau(\g,\g\inv)$, and using the canonical identification $\contz(L_{s^*s})\cong\contz(s^*s)$
to view $a$ as a continuous section of $L_{s^*s}$, we get
\begin{align*}
\beta_s(a)(\r(\g))&=(u_sau_s^*)(\g\s(\g)\g\inv)=u_s(\g)\cdot a(\s(\g))\cdot u_s(\g)^*\\
                    &= (1,\g)\cdot(a(\s(\g)),\s(\g))\cdot(1,\g)^*\\
                    &=(a(\s(\g))\tau(\g,\s(\g)),\g)\cdot (\overline{\tau(\g\inv,\g)},\g\inv)\\
                    &=(a(\s(\g))\overline{\tau(\g\inv,\g)}\tau(\g,\g\inv),\s(\g))=a(\s(\g)).
\end{align*}
To yield the relation between the cocycles $\omega(s,t)$ and $\tau(\alpha,\beta)$,
first observe that $\omega(s,t)$ is a unitary element of $\contb(stt^*s^*)\cong\contz(L_{stt^*s^*})$, that is,
a continuous function $\omega(s,t)\colon stt^*s^*\to \Torus$. Now, since $s,t$ are bisections, every element of $stt^*s^*$
can be uniquely written as $\alpha\beta\beta\inv\alpha\inv=\r(\alpha\beta)$ for $\alpha\in s$ and $\beta\in t$. Thus
\begin{align*}
\omega(s,t)(\r(\alpha\beta))&=(u_su_tu_{st}^*)(\alpha\beta\beta\inv\alpha\inv)\\
                                    &=u_s(\alpha)\cdot u_t(\beta)\cdot u_{st}(\alpha\beta)^*\\
                                    &=(1,\alpha)\cdot (1,\beta)\cdot (1,\alpha\beta)^*\\
                                    &=(\tau(\alpha,\beta),\alpha\beta)\cdot (\overline{\tau((\alpha\beta)\inv,\alpha\beta)},(\alpha\beta)\inv)\\
                                    &=(\tau(\alpha,\beta)\overline{\tau((\alpha\beta)\inv,\alpha\beta)}\tau(\alpha\beta,(\alpha\beta)\inv),\alpha\beta\beta\inv\alpha\inv)\\
                                    &=(\tau(\alpha,\beta),\r(\alpha\beta))=\tau(\alpha,\beta).
\end{align*}
\vskip-18pt
\end{proof}

Summarizing the results of this section, we have seen how to describe twisted étale groupoids in terms of inverse semigroup twisted
actions. While the $2$-cocycles on groupoids can only describe topologically trivial twisted groupoids, our twists have not such a limitation and
allow us to describe arbitrary twisted étale groupoids. As we have seen above, the topologically trivial twisted étale groupoids
essentially correspond to Sieben's twisted actions, a special case of our theory.

\section{Refinements of Fell bundles}
\label{sec:Refinements}

In this section, we introduce a notion of \emph{refinement} for Fell bundles and prove that several
constructions from Fell bundles, including cross-sectional \cstar{}algebras and twisted groupoids (in the semi-abelian case),
are preserved under refinements. Our main point in this section is to prove that locally regular Fell bundles admit a regular, saturated refinement.
This puts every locally regular Fell bundle into the setting of
Section~\ref{sec:regular Fell bundles} and enables us to describe it as a twisted action.

\begin{definition}\label{def:refinement}
Let $S,T$ be inverse semigroups and let $\A=\{\A_s\}_{s\in S}$ and $\B=\{\B_t\}_{t\in T}$ be Fell bundles.
A \emph{morphism} from $\B$ to $\A$ is a pair $(\phi,\psi)$, where $\phi\colon T\to S$ is a semigroup homomorphism,
and $\psi\colon\B\to \A$ is a map satisfying:
\begin{enumerate}[(i)]
\item $\psi(\B_t)\sbe \A_{\phi(t)}$ and the restriction $\psi_t\colon\B_t\to \A_{\phi(t)}$ is a linear map;
\item $\psi$ respects product and involution: $\psi(ab)=\psi(a)\psi(b)$ and $\psi(a^*)=\psi(a)^*$, for all $a,b\in \B$;
\item $\psi$ commutes with the inclusion maps: whenever $t\leq t'$ in $T$, we get a commutative diagram \label{def:refinement:inclusionmaps}
\[
\xymatrix{
    \B_t\ar[dd]_{\psi_t}\ar[rr]^{j_{t',t}^{\B}} &   & \B_{t'}\ar[dd]^{\psi_{t'}} \\
                                                &   &                            \\
    \A_{\phi(t)}\ar[rr]_{j_{\phi(t'),\phi(t)}^\A} &  & \A_{\phi(t')}
}
\]
We say that $\B$ is a \emph{refinement} of $\A$ if there is a morphism $(\phi,\psi)$ from $\B$ to $\A$ with
$\phi\colon T\to S$ surjective and \emph{essentially injective} in the sense that $\phi(t)\in E(S)$ implies $t\in E(T)$,
with $\psi_t\colon\B_t\to\A_{\phi(t)}$ injective for all $t\in T$, and such that
\begin{equation}\label{eq:ConditionRefinement}
\A_s=\overline{\sum\limits_{t\in \phi^{-1}(s)}\!\!\psi(\B_t)}\quad\mbox{for all }s\in S.
\end{equation}
\end{enumerate}
\end{definition}

\begin{remark}\label{rem:Refinement}
{\bf (1)} If $\B=\{\B_t\}_{t\in T}$ is a refinement of $\A=\{\A_s\}_{s\in S}$ via some morphism $(\phi,\psi)$,
then $\psi(\B_t)$ is an ideal of $\A_{\phi(t)}$ (as {\tros}) for all $t\in T$ (see comments before Definition~\ref{def:TROLocallyRegular}).
In fact, it is enough to check
that $\psi(\B_f)$ is an ideal of $\A_{e}$ for every idempotent $f\in T$ with $\phi(f)=e$.
Equation~\eqref{eq:ConditionRefinement} implies
\begin{equation*}
\A_e=\overline{\sum\limits_{\phi(g)=e}\!\!\psi(\B_g)}.
\end{equation*}
Since $\phi$ is essentially injective, each $g\in T$ with $\phi(g)=e$ is necessarily idempotent.
Hence,
\begin{equation*}
\psi(\B_f)\A_e=\overline{\sum\limits_{\phi(g)=e}\!\!\psi(\B_f)\psi(\B_g)}=\overline{\sum\limits_{\phi(g)=e}\!\!\psi(\B_f\B_g)}\sbe
\overline{\sum\limits_{\phi(g)=e}\!\!\psi(\B_{fg})}.
\end{equation*}
Given $g\in E(T)$ with $\phi(g)=e$, we have $fg\leq f$. Definition~\ref{def:refinement}\eqref{def:refinement:inclusionmaps} yields
\begin{equation*}
\psi(\B_{fg})=j_{e,e}^\A\big(\psi(\B_{fg})\big)=\psi\big(j_{f,fg}^\B(\B_{fg})\big)\sbe \psi(\B_f).
\end{equation*}
It follows that $\psi(\B_f)\A_e\sbe\psi(\B_f)$. Similarly, $\A_e\psi(\B_f)\sbe\psi(\B_f)$.

{\bf (2) }In the realm of (discrete) groups the notion of refinement is not interesting.
Indeed, if $S,T$ are groups and $\phi\colon T\to S$ is an essentially injective, surjective homomorphism, then
is it is automatically an isomorphism because the only idempotents are the group identities.
Hence for groups $S,T$ and Fell bundles $\A$ and $\B$ over $S$ and $T$, respectively, a refinement $(\phi,\psi)$ from $\B$ to $\A$
is the same as an isomorphism $(\phi,\psi)\colon(T,\B)\congto (S,\A)$.
However, we may have a Fell bundle $\A=\{\A_s\}_{s\in S}$ over a group $S$, and an interesting refinement
$\B=\{\B_t\}_{t\in T}$ allowing $T$ to be an inverse semigroup. For instance, we are going to prove (Proposition~\ref{prop:RefinementForLocRegularFellBundle}) that every Fell bundle admits a saturated refinement and this can, of course, be applied to Fell bundle over groups, but one has to allow the refinement itself to be a Fell bundle over an
inverse semigroup.
\end{remark}

First, we show that refinements preserve the Fell bundle cross-sectional \cstar{}algebras (see \cite{Exel:noncomm.cartan}
for details on the construction of these algebras).

\begin{theorem}\label{theo:RefinementPreserveFullC*Algebras}
Let $\B=\{\B_t\}_{t\in T}$ be a refinement of $\{\A_s\}_{s\in S}$ through a morphism $(\phi,\psi)$,
and let $\E_\B$ and $\E_\A$ be the restrictions of $\B$ and $\A$ to the idempotent parts of $T$ and $S$, respectively.
Then there is a \textup(unique\textup) isomorphism $\Psi\colon C^*(\B)\congto C^*(\A)$ satisfying $\Psi(b)=\psi(b)$ for all $b\in \B$.
Here we view each fiber $\B_t$ \textup(resp. $\A_s$\textup) as a subspace
of $C^*(\B)$ \textup(resp. $C^*(\A)$\textup) via the universal representation.
Moreover, $\Psi$ factors through an isomorphism $C^*_\red(\B)\congto C^*_\red(\A)$ and
restricts to an isomorphism $\Psi\rest{}\colon C^*(\E_\B)\congto C^*(\E_\A)$
\end{theorem}
\begin{proof}
Since $(\phi,\psi)$ is a morphism from $\B$ to $\A$, it induces a (unique) \Star{}ho\-mo\-mor\-phism
$\Psi\colon C^*(\B)\to C^*(\A)$ satisfying $\Psi(b)=\psi(b)$ for all $b\in \B$.
Moreover, it also induces a map from $\Rep(\A)$ to $\Rep(\B)$ (the classes of representations of $\A$ and $\B$, respectively)
that takes $\pi\in \Rep(\A)$ and associates the representation $\tilde\pi\defeq \pi\circ\psi\in \Rep(\B)$.
All this holds for any morphism of Fell bundles. Now, to prove that the induced map $\Psi$ is an isomorphism,
we need to use the extra properties of refinement. The surjectivity of $\Psi$ follows from
Equation~\eqref{eq:ConditionRefinement}. In fact, this equation implies that
$\Psi$ has dense image, and since any \Star{}homomorphism between $C^*$-algebras has closed image,
the surjectivity of $\Psi$ follows. To show that $\Psi$ is injective, it is enough to show that
any representation $\rho\in \Rep(\B)$ has the form $\rho=\tilde\pi$ for some representation $\pi\in \Rep(\A)$
(necessarily unique by Equation~\eqref{eq:ConditionRefinement}).
Given $\rho\in \Rep(\B)$ and $s\in S$, we define
\begin{equation*}
\pi_s(b)\defeq \sum\limits_{t\in \phi^{-1}(s)}\rho(b_t),
\end{equation*}
whenever $b\in \A_s$ is a finite sum of the form
\begin{equation*}
b=\sum\limits_{t\in \phi^{-1}(s)}\psi(b_t)
\end{equation*}
with all but finitely many non-zero $b_t$'s in $\B_t$. All we have to show is that $\pi_s$ is well-defined and extends
to $\A_s$. Since $\A_s$ is the closure of $b$'s as above, it is enough to show that
\begin{equation*}
\left\|\sum\limits_{t\in \phi^{-1}(s)}\rho(b_t)\right\|\leq \|b\|.
\end{equation*}
First, note that
\begin{equation*}
\|b\|^2=\left\|\sum\limits_{t,r\in \phi^{-1}(s)}\psi(b_t^*b_r)\right\|.
\end{equation*}
Given $t,s\in \phi^{-1}(s)$, we have $\phi(t^*r)=s^*s$, and because $\phi$ is essentially injective, this implies that $t^*r$ is idempotent.
Consequently, we may view $\sum\limits_{t,r\in \phi^{-1}(s)}b_t^*b_r$ as an element of $C^*(\E_\B)$.
Using \cite[Proposition 4.3]{Exel:noncomm.cartan}, it is easy to see that the \Star{}homomorphism
$\Psi\colon C^*(\B)\to C^*(\A)$ is injective (hence isometric) on $C^*(\E_\B)$.
Therefore,
\begin{multline*}
\|b\|^2=\left\|\sum\limits_{t,r\in \phi^{-1}(s)}\psi(b_t^*b_r)\right\|=\left\|\psi\left(\sum\limits_{t,r\in \phi^{-1}(s)}b_t^*b_r\right)\right\|
        \\ =\left\|\sum\limits_{t,r\in \phi^{-1}(s)}b_t^*b_r\right\|\geq \left\|\sum\limits_{t,r\in \phi^{-1}(s)}\rho(b_t^*b_r)\right\|
        =\left\|\sum\limits_{t\in \phi^{-1}(s)}\rho(b_t)\right\|^2.
\end{multline*}
Thus $\pi_s$ is well defined and extends to a (obviously linear) map $\pi_s\colon\A_s\to \bound(\hils_\rho)$.
Since $s$ is arbitrary, we get a map $\pi\colon\A\to \bound(\hils_\rho)$ which is easily seen to be a representation because $\rho$ is.
And, of course, we have $\tilde\pi=\rho$. This shows that $\Psi$ is injective and, therefore,
an isomorphism $C^*(\B)\to C^*(\A)$. It is clear that it restricts to an isomorphism
$\Psi\rest{}\colon C^*(\E_\B)\to C^*(\E_\A)$.

Finally, we show that $\Psi$ factors through an isomorphism $C^*_\red(\B)\congto C^*_\red(\A)$.
First, let us recall that the reduced cross-sectional \cstar{}algebra $C^*_\red(\A)$ is the image of $C^*(\A)$ by the
regular representation $\Lambda_\A\colon C^*(\A)\to C^*_\red(\A)$ of $\A$ (see \cite[Proposition~8.6]{Exel:noncomm.cartan}),
which is defined as the direct sum of all GNS-representations associated
to states $\tilde\varphi$ of $C^*(\A)$, where $\varphi$ runs over the set of all pure states of $C^*(\E_\A)$
and $\tilde\varphi$ is the canonical extension of $\varphi$ as defined in \cite[Section~7]{Exel:noncomm.cartan}.
Of course, the same is true for the regular representation $\Lambda_\B\colon C^*(\B)\to C^*_\red(\B)$ of $\B$.
Since $\Psi\rest{}\colon C^*(\E_\B)\to C^*(\E_\A)$ is an isomorphism, the assignment $\varphi\mapsto \varphi\circ\Psi\rest{}$ defines
a bijective correspondence between pure states of $C^*(\E_A)$ and $C^*(\E_\B)$. To prove that $\Psi$ factors through an
isomorphism $C^*_\red(\B)\to C^*_\red(\A)$ it is enough to show that the canonical extension of $\varphi\circ\Psi\rest{}$ coincides with
$\tilde{\varphi}\circ\Psi$, that is, $\widetilde{\varphi\circ\Psi\rest{}}=\tilde{\varphi}\circ\Psi$ for all pure states $\varphi$ of $C^*(\E_\A)$.
Let $t\in T$ and $b\in \B_t$. According to Proposition~7.4(i) in \cite{Exel:noncomm.cartan}, we have two cases to consider:

\emph{Case 1.} Assume there is an idempotent $f\in E(T)$ lying in $\supp(\varphi\circ\Psi\rest{})$, the support
of $\varphi\circ\Psi\rest{}$ (see \cite[Definition~7.1]{Exel:noncomm.cartan}), with $f\leq t$. Let $e\defeq \phi(f)$ and $s\defeq\phi(t)$.
Note that $e\leq s$. Moreover, since $\Psi\rest{}\colon C^*(\E_\B)\to C^*(\E_\A)$ is an isomorphism and $\varphi\circ\Psi\rest{}$ is
supported on $\B_f$, it follows that $\varphi$ is supported on $\psi(\B_f)$ and hence on $\A_e$ (by \cite[Proposition~5.3]{Exel:noncomm.cartan})
because $\psi(\B_f)$ is an ideal of $\A_e$ by Remark~\ref{rem:Refinement}(1). This implies that if $(u_i)$ is an approximate
unit for $\B_f$, then
\begin{equation*}
\lim_i\varphi\big(\psi(b)\psi(u_i)\big)=\tilde\varphi_e^s(\psi(b)).
\end{equation*}
Here use the same notation of \cite{Exel:noncomm.cartan} and write $\varphi_e$ for the restriction of
$\varphi$ to $\A_e$ and $\tilde\varphi_e^s$ for the canonical extension of $\varphi_e$ to $\A_s$ (see \cite[Proposition~6.1]{Exel:noncomm.cartan}).
It follows that
\begin{multline*}
\widetilde{\varphi\circ\Psi\rest{}}(b)=\big(\widetilde{\varphi\circ\Psi\rest{}}\big)_f^t(b)=\lim_i(\varphi\circ\Psi\rest{})(bu_i)\\
=\lim_i\varphi\big(\psi(b)\psi(u_i)\big)=\tilde\varphi_e^s(\psi(b))=\tilde\varphi\big(\psi(b)\big)=(\tilde\varphi\circ\Psi)(b).
\end{multline*}

\emph{Case 2.} Suppose there is no $f\in \supp(\varphi\circ\Psi\rest{})$ with $f\leq t$. In this
case, we have $\widetilde{\varphi\circ\Psi\rest{}}(b)=0$ by \cite[Proposition~7.4(i)]{Exel:noncomm.cartan}.
If, on the other hand, there is no $e\in \supp(\varphi)$ with $e\leq s\defeq \phi(t)$, then we also have $(\tilde\varphi\circ\Psi)(b)=0$.
Assume there is $e\in \supp(\varphi)$ such that $e\leq s$. Since $\phi$ is surjective, there is $g\in E(T)$ with
$\phi(g)=e$. Since $\phi(tg)=se=e$, we have $tg\in E(T)$ because $\phi$ is essentially injective.
Note that $tg\leq t$. By assumption, $tg\notin\supp(\varphi\circ\Psi\rest{})$, so that $(\varphi\circ\Psi\rest{})(\B_{tg})=\{0\}$ by
\cite[Proposition~5.5]{Exel:noncomm.cartan}. Let $(v_j)$ be an approximate unit for $\A_e$. Then
\begin{equation}\label{eq:tildeVarphiCircPsi}
(\tilde\varphi\circ\Psi)(b)=\tilde\varphi_e^s(\psi(b))=\lim_j\varphi(\psi(b)v_j).
\end{equation}
By~\eqref{eq:ConditionRefinement}, $\A_e$ is the closed linear span of $\psi(\B_g)$ with $\phi(g)=e$.
Thus each $v\in \A_e$ is a limit of finite sums of the form $\sum \psi(u_n)$ with $u_n\in \B_{g_n}$, where $g_n\in E(T)$ and $\phi(g_n)=e$.
Since $(\varphi\circ\Psi\rest{})(\B_{tg_n})=\{0\}$, we have $\varphi(\psi(bu_n))=0$ for all $n$.
It follows that $\varphi(\psi(b)v)=0$ for all $v\in \A_e$ and
hence Equation~\eqref{eq:tildeVarphiCircPsi} yields $(\tilde\varphi\circ\Psi)(b)=0$.

Therefore, in any case we have $\widetilde{\varphi\circ\Psi\rest{}}(b)=\tilde{\varphi}\circ\Psi(b)$ for all $b\in \B$
and hence for all $b\in C^*(\B)$, and this concludes the proof.
\end{proof}

Next, we show that any locally regular Fell bundle admits a regular, saturated refinement.

\begin{proposition}\label{prop:RefinementForLocRegularFellBundle}
\begin{enumerate}[(a)]
\item Every Fell bundle admits a saturated refinement.
\item Every locally regular Fell bundle admits a regular, saturated refinement.
\item If a Fell bundle admits a regular refinement, then it is locally regular.
\end{enumerate}
\end{proposition}
\begin{proof}
(a) Let $\A=\{\A_s\}_{s\in S}$ be a Fell bundle over an inverse semigroup $S$.
We may assume that $\A$ is a concrete Fell bundle in $\bound(\hils)$ and the $C^*$-algebras $B=C^*(\E)$ and $A=C^*(\A)$ are all realized as operators in $\bound(\hils)$ for some Hilbert space $\hils$. Consider the set
$S_B$ of all {\tro} $\troa\sbe\bound(H)$ satisfying
\begin{equation}\label{eq:InverseSemigroupOfTROs}
\troa B,B\troa \sbe \troa \quad\mbox{and}\quad\troa^*\troa ,\troa \troa^*\sbe B.
\end{equation}
Note that $S_B$ is an inverse semigroup with respect to the multiplication $\troa \cdot \trob \defeq \troa\trob =\cspn(\troa\trob)$. In fact,
first we have to show that the multiplication is well-defined. For this, take $\troa ,\trob \in S_B$. It is easy to see that $\troa\trob$
satisfies~\eqref{eq:InverseSemigroupOfTROs}. To see that $\troa\trob$ is again a {\tro},
observe that $I=\troa^*\troa $ and $J=\trob\trob^*$ are ideals in $B$, and
ideals of any $C^*$-algebra always commute as sets ($IJ=I\cap J=JI$). Hence
\begin{equation*}
(\troa\trob)(\troa\trob)^*(\troa\trob)=\troa (\trob \trob^*)(\troa^*\troa )\trob
        =\troa (\troa^*\troa )(\trob \trob^*)\trob =\troa\trob.
\end{equation*}
Thus $\troa\trob$ is a {\tro}. Now, because each $\troa \in S_B$ is a {\tro},
we have $\troa \troa^*\troa =\troa $ and $\troa^*=\troa^*\troa \troa^*$. Thus $\troa^*$ is an inverse of $\troa $ in $S_B$.
To show the uniqueness of inverses, it suffices to show that idempotents commute (see Theorem~3 in \cite[Chapter~1]{Lawson:InverseSemigroups}).
Let $\troa \in S_B$ be an idempotent, that is, $\troa^2=\troa \troa =\troa $.
We have $\troa^*=\troa^*\troa \troa^*=(\troa^*\troa )(\troa \troa^*)=(\troa \troa^*)(\troa^*\troa )=\troa \troa^*\troa^*\troa =\troa \troa^*\troa =\troa $. Thus, idempotents of $S_B$ have the form $\troa^*\troa $ with $\troa \in S_B$, and these commute because they are ideals of $B$. Hence $S_B$ is an inverse semigroup.

Observe that each $\troa =\A_s$ is a {\tro} in $\bound(\hils)$ that satisfies $\troa B,B\troa \sbe \troa $ and $\troa^*\troa ,\troa \troa^*\sbe B$
so that $\A_s\in S_B$. Consider
\begin{equation*}
T\defeq \{(s,\troa )\in S\times S_B\colon \troa \sbe \A_s\}.
\end{equation*}
Note that $T$ is an inverse sub-semigroup of $S\times S_B$ and we have a canonical
surjective homomorphism $\phi\colon T\to S$ given by the projection onto the first coordinate. Moreover, $\phi$ is essentially injective.
Indeed, suppose that $\phi(e,\troa )=e$ is idempotent in $S$. Since $(e,\troa )\in T$, we have $\troa \sbe \A_e$. Thus
\begin{equation*}
\A_e\troa^*\troa \sbe \A_e\A_e^*\troa =\A_e\troa \sbe \troa =\troa \troa^*\troa \sbe \A_e\troa^*\troa .
\end{equation*}
Therefore $\A_e\troa^*\troa =\troa $ and because $\troa^*\troa $ is an ideal of $\A_e$ we have $\A_e\troa^*\troa =\troa^*\troa $, that is, $\troa =\troa^*\troa $ is an idempotent of $S_B$.
We conclude that $(e,\troa )$ is an idempotent in $T$, whence $\phi$ is essentially injective.

Now we define a Fell bundle $\B=\{\B_t\}_{t\in T}$ over $T$ with fibers $\B_t\defeq \troa $ whenever $t=(s,\troa )$.
Notice that the order relation in $S_B$ is just the inclusion, that is, $\troa \leq \trob $ in $S_B$ if and only if $\troa \sbe \trob $. In fact,
if $\troa \leq \trob $, then $\troa =\trob \troa^*\troa \sbe \trob B\sbe \trob$. And if $\troa \sbe \trob $, then
\begin{equation*}
\trob \troa^*\troa \sbe \trob \trob^*\troa \sbe B\troa \sbe \troa =\troa \troa^*\troa \sbe \trob \troa^*\troa ,
\end{equation*}
so that $\trob \troa^*\troa =\troa $, that is, $\troa \leq \trob $. This allows us to define inclusion maps $j_{t',t}\colon\B_{t}\to \B_{t'}$ for $\B$ whenever $t\leq t'$ in $T$.
Of course, the algebraic operations of $\B$ are inherited from $\bound(\hils)$. With this structure, $\B$ is a saturated Fell bundle over $T$.
Moreover, by construction, it is a concrete Fell bundle in $\bound(\hils)$ and we have $\B_t\sbe \A_{\phi(t)}$ for all $t\in T$.
Thus, we get a canonical map $\psi\colon\B\to \A$ whose restriction to $\B_t$ is the inclusion $\B_t\into \A_{\phi(t)}$.
The pair $(\phi,\psi)$ is a morphism from $\B$ to $\A$ and through this morphism $\B$ is a refinement of $\A$. Therefore every Fell bundle
has a saturated refinement.

(b) Now we assume that $\A$ is locally regular. Then we redefine $S_B$ of part (a) taking only the \emph{regular} {\tro} $\troa\sbe\Ls(\hils)$ satisfying~\eqref{eq:InverseSemigroupOfTROs}.
If $\troa ,\trob $ are regular {\tro},
there are $u\in \troa $ and $v\in \trob $ with $u\sim \troa $ and $v\sim \trob $. We have
\begin{multline*}
uv(\troa\trob)^*(\troa\trob)=uv\trob^*\troa^*\troa\trob =u(\trob\trob^*)(\troa^*\troa )\trob\\
        =u(\troa^*\troa )(\trob \trob^*)\trob =\troa\trob .
\end{multline*}
Analogously, $(\troa\trob)(\troa\trob)^*uv=\troa\trob$, so that $uv\sim \troa\trob$ and therefore $\troa\trob$ is a regular {\tro} in $\bound(\hils)$.
It follows that $S_B$ is also an inverse semigroup. With the same definition for $T$, $\B$ and $(\phi,\psi)$ as above, we get the desired regular refinement of $\A$. In fact, by construction, $\B$ is a regular, saturated Fell bundle which is a refinement of $\A$ through the morphism $(\phi,\psi)$. The only non-trivial axiom to be checked is~\eqref{eq:ConditionRefinement}. But this follows from the definition of local regularity.

(c) If $\A$ has a regular refinement $\B=\{\B_t\}_{t\in T}$ via some morphism $(\phi,\psi)$, then each $\psi(\B_t)$ is a regular ideal
in $\A_{\phi(t)}$ (see Remark~\ref{rem:Refinement}(1)) and therefore $\A_{\phi(t)}$ is locally regular by~\eqref{eq:ConditionRefinement}. Since $\phi$ is surjective, this shows that $\A_s$ is locally regular for all $s\in S$, that is, $\A$ is locally regular.
\end{proof}

Observe that Theorem~\ref{theo:RefinementPreserveFullC*Algebras} and Proposition~\ref{prop:RefinementForLocRegularFellBundle}
enable us to apply our main results to every (not necessarily saturated)
locally regular Fell bundle $\A$ and describe their cross-sectional \cstar{}algebras $C^*(\A)$ and $C^*_\red(\A)$ as
(full or reduced) twisted crossed products.

\begin{remark}
Let $\A=\{\A_s\}_{s\in G}$ be a regular (not necessarily saturated) Fell bundle over a (discrete) group $G$.
By Theorem~7.3 in \cite{Exel:twisted.partial.actions}, this corresponds to a \emph{twisted partial action} $(\alpha,\upsilon)$ of $G$ on the unit fiber $A\defeq\A_1$. On the other hand, applying Proposition~\ref{prop:RefinementForLocRegularFellBundle}, and taking a saturated, regular refinement $\B=\{\B_t\}_{t\in T}$ of $\A$ (here $T$ is an inverse semigroup), we may describe the same system as a twisted action of $T$.
More precisely, by Corollary~\ref{cor:CorrespondenceRegularFellBundlesAndTwistedActions}, there is a twisted action $(\beta,\omega)$ of $T$ on $C^*(\E_\B)\cong A$ (by Theorem~\ref{theo:RefinementPreserveFullC*Algebras}) such that the crossed product
$A\rtimes_{\beta,\omega}^{(\red)}T\cong C^*_{(\red)}(\B)$ is (again by Theorem~\ref{theo:RefinementPreserveFullC*Algebras}) isomorphic to the crossed product $A\rtimes_{\alpha,\upsilon}^{(\red)}G\cong C^*_{(\red)}(\A)$.

By the way, it should be also possible to define \emph{twisted partial actions} of inverse semigroups, generalizing at the same time our
Definition~\ref{def:twisted action} of twisted action and that of \cite[Definition~2.1]{Exel:twisted.partial.actions} for groups. As in our main result (Corollary~\ref{cor:CorrespondenceRegularFellBundlesAndTwistedActions}), twisted partial actions should correspond to regular (not necessarily saturated) Fell bundles. However, the (full or reduced) cross-sectional \cstar{}algebra
of any such Fell bundle can also be described as a (full or reduced) crossed product by
some twisted action in our sense (again by Theorem~\ref{theo:RefinementPreserveFullC*Algebras} and Proposition~\ref{prop:RefinementForLocRegularFellBundle}). This is the reason why we have chosen not to consider twisted partial actions.
\end{remark}

Recall from \cite{BussExel:Fell.Bundle.and.Twisted.Groupoids} that a Fell bundle $\A=\{\A_s\}_{s\in S}$ is called \emph{semi-abelian} if the fibers $\A_e$ are commutative \cstar{}algebras for every idempotent $e\in S$.

\begin{corollary}\label{cor:RefinementForSemiAbelianFellBundle}
Every semi-abelian Fell bundle has a regular, saturated \textup(necessarily semi-abelian\textup) refinement.
\end{corollary}
\begin{proof}
By Proposition~\ref{prop:commutative=>loc.regular}, every semi-abelian Fell bundle is locally regular.
Hence the assertion follows from Proposition~\ref{prop:RefinementForLocRegularFellBundle}
\end{proof}

In \cite{BussExel:Fell.Bundle.and.Twisted.Groupoids} we have shown that there is a close relationship between
saturated, semi-abelian Fell bundles and twisted étale groupoids, or equivalently, Fell line bundles over étale groupoids.
The following result shows that the twisted étale groupoids associated to semi-abelian Fell bundles are not changed under refinements:

\begin{proposition}\label{prop:RefinementPreserveTwistedGroupoids}
If $\B=\{\B_t\}_{t\in T}$ is a saturated refinement of a saturated, semi-abelian Fell bundle $\A=\{\A_s\}_{s\in S}$,
then the twisted étale groupoids associated to $\B$ and $\A$ as in \cite[Section~3.2]{BussExel:Fell.Bundle.and.Twisted.Groupoids} are isomorphic.
\end{proposition}
\begin{proof}
During the proof we shall write $(\G,L)$ and $(\G',L')$ for the twisted groupoids associated to $\B$ and $\A$, respectively,
as constructed in \cite[Section~3.2]{BussExel:Fell.Bundle.and.Twisted.Groupoids}. Here $\G$ and $\G'$ are étale groupoids
and $L$ and $L'$ are Fell line bundles over $\G$ and $\G'$, respectively. Let $(\phi,\psi)$ be the morphism from
$\B$ to $\A$ that gives $\B$ as a refinement of $\A$ as in Definition~\ref{def:refinement}.
The map $\psi\colon\B\to \A$ preserves all the algebraic operations and inclusion maps of $\B$ and $\A$ and
restricts to an injective map $\B_t\into \A_{\phi(t)}$ for all $t\in T$. So, we may suppress $\psi$ and identify each $\B_t$
as a closed subspace of $\A_{\phi(t)}$. By Theorem~\ref{theo:RefinementPreserveFullC*Algebras},
these inclusions extend to the level of $C^*$-algebras yielding an isomorphism $C^*(\B)\cong C^*(\A)$
which restricts to an isomorphism of the underlying commutative $C^*$-algebras
$C^*(\E_\B)$ and $C^*(\E_\A)$. Let us say that $C^*(\E_\B)\cong C^*(\A)\cong \cont_0(X)$ for some locally compact space $X$.
Through these identifications, $\B_f\sbe\A_e\sbe\cont_0(X)$ becomes an inclusion of ideals whenever $e\in E(S)$, $f\in E(T)$ and $\phi(f)=e$.
Moreover,
\begin{equation}\label{eq:U_f is the union of U_f's}
\A_e=\overline{\sum\limits_{f\in \phi^{-1}(e)}\B_f}\quad \Longrightarrow\quad\U_e=\bigcup\limits_{f\in \phi^{-1}(e)}\U_f,
\end{equation}
where $\U_e$ is the open subset of $X$ that corresponds to the ideal $\A_e\sbe\cont_0(X)$, that is, $\A_e=\cont_0(\U_e)$.
By definition, the $\G$ and $\G'$ are groupoid of germs for certain canonical actions $\theta$ and $\theta'$ of $T$ and $S$ on $X$, respectively
(see \cite[Proposition~3.5]{BussExel:Fell.Bundle.and.Twisted.Groupoids} for the precise definition).
We shall use the same notation as in \cite{BussExel:Fell.Bundle.and.Twisted.Groupoids}
and write elements of $\G$ as equivalence classes $\germ tx$ of pairs $(t,x)$ where $t\in T$ and $x\in \dom(\theta_t)=\U_{t^*t}$,
and elements of $ L$ as equivalence classes $\qtrip btx$ of triples $(b,t,x)$ with $b\in \B_t$ and $x\in \dom(b)=\{y\in X\colon (b^*b)(y)>0\}$.
See \cite[Section~3.2]{BussExel:Fell.Bundle.and.Twisted.Groupoids} for further details.
Of course, we also use similar notations for elements of $\G'$ and $ L'$.
Now, consider the maps $\Phi\colon\G\to\G'$ and $\Psi\colon L\to  L'$ defined by
$\Phi\germ ty=\germ{\phi(t)}{y}$ and $\Psi\qtrip bty=\qtrip{\psi(b)}{\phi(t)}{y}$. We are going to show that these maps
give us the desired isomorphism $(\G, L)\cong (\G', L')$.
It is easy to see that the definition of $\Phi$ and $\Psi$ do not depend on representant choices for the equivalence classes,
that is, $\Phi$ and $\Psi$ are well-defined maps. Since $\phi$ is surjective, so is $\Phi$.
By \cite[Proposition~3.5]{BussExel:Fell.Bundle.and.Twisted.Groupoids}, $\theta_t\colon\U_{t^*t}\to \U_{tt^*}$ is the union of the partial homeomorphisms $\theta_b$ with $b\in \B_t$ defined in
\cite[Lemma~3.3]{BussExel:Fell.Bundle.and.Twisted.Groupoids}. Of course, the same holds for the partial homeomorphisms $\theta_s'\colon\U_{s^*s}\to \U_{ss^*}$, $s\in S$, associated to the Fell bundle $\A$. By~\eqref{eq:ConditionRefinement},
each $\theta_s'$ is the union of the partial homeomorphisms $\theta_t$ with $t\in \phi^{-1}(s)$.
This implies that $\Psi$ is surjective.
To show that $\Phi$ is injective, suppose that $\germ ty,\germ{t'}{y}\in \G$ and
$\germ{\phi(t)}{y}=\germ{\phi(t')}{y}$ in $\G'$, that is, there is $e\in E(S)$ such that $y\in \U_e$
and $\phi(t)e=\phi(t')e$. Equation~\eqref{eq:U_f is the union of U_f's} yields $f\in E(T)$ such that $y\in \U_f$ and $\phi(f)=e$.
Thus $\phi(t_1)=\phi(t_2)$, where $t_1=tf$ and $t_2=t'f$. Since $\phi$ is essentially injective, $g\defeq t_1^*t_2$ belongs to $E(T)$. Note that $t_1g=ht_2$, where $h\defeq t_1t_1^*\in E(T)$. Moreover,
\begin{equation*}
ht_2=ht_2t_2^*t_2=t_2t_2^*ht_2=t_2(t_2^*t_1)(t_1^*t_2)=t_2g^*g=t_2g.
\end{equation*}
Hence $t_1g=t_2g$ and since $y\in \U_{t_1^*t_1}\cap\U_{t_2^*t_2}$ and $g=t_1^*t_2$ is idempotent, it follows that $y\in \U_g$
(see proof of \cite[Lemma~3.4]{BussExel:Fell.Bundle.and.Twisted.Groupoids}). We conclude that $tfg=t'fg$ and $y\in \U_{f}\cap\U_g=\U_{fg}$, so that
$\germ ty=\germ{t'}{y}$. This shows that $\Phi$ is injective. We now prove that $\Psi$ injective. Assume that $\qtrip{\psi(b)}{\phi(t)}{y}=\qtrip{\psi(b')}{\phi(t')}{y}$
in $ L'$, that is, there is $c,c'\in\E_\A$ such that $c(y),c'(y)>0$ and $\psi(b)c=\psi(b')c'$. Equation~\eqref{eq:U_f is the union of U_f's}
implies the existence of $d,d'\in\E_\B$ with $c=\psi(d)$ and $c'=\psi(d')$. Since $\psi$ induces the isomorphism $C^*(\E_\B)\cong C^*(\E_\A)\cong \cont_0(X)$, we must have $d(y),d'(y)>0$. Let $a\defeq bd$ and $a'\defeq b'd'$. Then $a\in \B_s$, $a'\in \B_{s'}$ for some $s,s'\in T$,
and we have $\psi(a)=\psi(a')$. Note that $\qtrip{a}{s}{y}=\qtrip bty$ and $\qtrip{a'}{s'}{y}=\qtrip{b'}{t'}{y}$.
The injectivity of $\Psi$ will follow if we show that $\qtrip{a}{s}{y}=\qtrip{a'}{s'}{y}$.
By definition of refinement, $\psi$ is injective when restricted to the fibers of $\B$. The only small problem is that $a$ and $a'$
might not be in the same fiber. However we may circumvent this problem: suppose that $a\in \A_t$ and $a'\in \A_{t'}$.
Since $\psi(a)=\psi(a')$, $\psi(a)\in \B_{\phi(t)}$ and $\psi(a')\in \B_{\phi(t')}$, we must have $\phi(t)=\phi(t')$.
Hence $\Phi(\germ ty)=\Phi(\germ{t'}{y})$. The (already checked) injectivity of $\Phi$ yields $f\in E(T)$ with $y\in \U_f$ and $r\defeq tf=t'f$.
Now take any function $a_f\in \B_f=\cont_0(\U_f)$ with $a_f(y)>0$. Note that $\qtrip{a}{r}{y}=\qtrip{aa_f}{r}{y}$,
$\qtrip{a'}{r'}{y}=\qtrip{a'a_f}{r}{y}$ and $\psi(aa_f)=\psi(a'a_f)$. Since now both $aa_f,a'a_f\in \B_r$ the injectivity of $\psi_r\colon\B_r\to \A_{\phi(r)}$ implies that $aa_f=a'a_f$.
Therefore $\qtrip{a}{t}{y}=\qtrip{a'}{t'}{y}$, whence the injectivity of $\Psi$ follows.

It is easy to see that $\Phi\colon \G\to \G'$ and $\Psi\colon L\to L'$ preserve all the algebraic operations involved.
Moreover, $\Phi$ and $\Psi$ are homeomorphisms: given $t\in T$ and an open subset $U\sbe\U_{t^*t}$,
the basic neighborhood $\open(t,U)=\{\germ sy\colon y\in U\}$ in $\G$
(see Equation~(3.6) in \cite{BussExel:Fell.Bundle.and.Twisted.Groupoids}) is mapped by $\Phi$ to the
basic neighborhood $\open(s,U)$ in $\G'$, where $s=\phi(t)$ (note that $\U_{t^*t}\sbe\U_{s^*s}$).
Moreover, since any $\U_e$ with $e\in E(S)$ is the union of $\{\U_f\colon f\in \phi^{-1}(e)\}$,
the neighborhoods of the form $\open(s,U)$ with $s=\phi(t)$ and $U$ an open subset of $\U_{t^*t}$, generate the topology of $\G'$.
This implies that $\Phi$ is a homeomorphism. To prove that $\Psi$ is a homeomorphism, we apply \cite[Proposition~II.13.17]{fell_doran}.
For this it is enough to check that $\Psi$ is continuous and isometric on the fibers. But, since $L$ and $L'$ are both Fell bundles, the
injectivity of $\Psi$ (already checked above) implies that it is isometric. To see that $\Psi$ is continuous, let us recall from
\cite[Proposition~3.25]{BussExel:Fell.Bundle.and.Twisted.Groupoids} that the topology on $L$ is generated by the local sections
$\hat{b}\germ ty\defeq\qtrip bty$ for $b\in \B_t$, and similarly for $L'$. Now the continuity of
$\Psi$ follows from the equality $\Psi\circ \hat{b}=\widehat{\psi(b)}\circ\Phi$.

The bundle projections $\pi\colon L\onto \G$ and $\pi'\colon L'\onto \G'$
are defined by $\pi\qtrip bty= \germ ty$ and $\pi'\qtrip asx=\germ sx$.
From this, it is clear that the following diagram is commutative:
\begin{equation*}
\xymatrix{
L \ar[d]_{\Psi}\ar[r]^{\pi} &    \G  \ar[d]_{\Phi} \\
L' \ar[r]_{\pi'}              &    \G'
}
\end{equation*}
Therefore, the pair $(\Phi,\Psi)\colon (\G, L)\to(\G', L')$ is an isomorphism of Fell bundles.
\end{proof}

The semi-abelian Fell bundle associated to a twisted étale groupoid as in \cite{BussExel:Fell.Bundle.and.Twisted.Groupoids}
is automatically saturated. On the other hand, if $\A$ is a non-saturated, semi-abelian Fell bundle, we may apply
Corollary~\ref{cor:RefinementForSemiAbelianFellBundle} to find a saturated, semi-abelian refinement $\B$ of $\A$
and then apply the results of \cite{BussExel:Fell.Bundle.and.Twisted.Groupoids}
to find the associated twisted étale groupoid $(\G,\Sigma)$. By Theorem~\ref{theo:RefinementPreserveFullC*Algebras}
and Propositions~2.18 and~3.40 in \cite{BussExel:Fell.Bundle.and.Twisted.Groupoids}, we have isomorphisms
\begin{equation*}
C^*_{(\red)}(\A)\cong C^*_{(\red)}(\B)\cong C^*_{(\red)}(\G,\Sigma).
\end{equation*}
For the isomorphism between the full \cstar{}algebras above it is necessary to assume that
$\G$ is Hausdorff or second countable because this is part of the hypothesis in \cite[Propositions~2.18]{BussExel:Fell.Bundle.and.Twisted.Groupoids}.


\end{document}